\documentclass{amsbook}

\usepackage{amssymb,latexsym}
\usepackage{enumerate}
\usepackage{amsmath}
\usepackage{graphicx}
\usepackage{verbatim}
\usepackage{hyperref} 
\usepackage[all]{xy}

\usepackage{amssymb}
 \usepackage{latexsym}
\usepackage{amsmath}
\usepackage{graphicx}
\usepackage[all]{xy}  
\newtheorem{theorem}{Theorem}[chapter]
\newtheorem{proposition}[theorem]{Proposition}
\newtheorem{lemma}[theorem]{Lemma}
\newtheorem{corollary}[theorem]{Corollary}
\theoremstyle{definition}
\newtheorem{definition}[theorem]{Definition}
\newtheorem{example}[theorem]{Example}
\newtheorem{remark}[theorem]{Remark}

\newcommand{\lV}{\left\Vert}
\newcommand{\rV}{\right\Vert}
\newcommand{\norm}{\left\Vert\,\cdot\,\right\Vert}
\newcommand{\N}{\mathbb N}
\newcommand{\ip}[2]{\langle #1, #2 \rangle}
\newcommand{\proten}{\,\widehat\otimes\,}

\newcommand{\widecheck}[1]{\stackrel{\textrm{\tiny \rotatebox[origin=c]{90}{$\!\!\!\langle\!$}}}{#1}}
\newcommand{\projectivetensor}{\,\widehat{\otimes}\,}
\newcommand{\injectivetensor}{\!\widecheck{\otimes}\!}

\newcommand{\lp}{\left(} 
\newcommand{\rp}{\right)} 

\newcommand{\italic}[1]{{\it #1\/}}

\newcommand{\Imath}{{\rm i}}

\newcommand{\tnorm}[1]{
  \left\vert\kern-0.9pt\left\vert\kern-0.9pt\left\vert #1
    \right\vert\kern-0.9pt\right\vert\kern-0.9pt\right\vert}

\newcommand{\dd}{\,{\rm d}}

\newcommand{\ex}{\mathop{\rm ex}}

\def\#{\flat}

\newcommand{\s}{\smallskip}

\newcommand{\nsubset}{\not\subset}
\newcommand{\enproof}{\hspace*{\stretch{1}}\qed}  

\newcommand{\flexiblenorm}[1]{\left\|#1\right\|}

\newcommand{\Lspace}{\mathit{L}}
\newcommand{\Lone}{\Lspace^1}

\newcommand{\lspace}{\ell\,}
\newcommand{\lone}{\lspace^1}

\newcommand{\linfty}{\lspace^\infty}
\newcommand{\co}{\mathit{c}_{\,0}}

\newcommand{\set}[1]{\left\{#1\right\}}  
\newcommand{\closure}[1]{\overline{#1}}  

\newcommand{\C}{\mathbb{C}}    
\newcommand{\conjugate}[1]{\overline {#1}} 
\newcommand{\D}{\mathbb{D}}
\newcommand{\T}{\mathbb{T}}
\newcommand{\R}{\mathbb{R}}    
\newcommand{\naturals}{\mathbb{N}}    
\newcommand{\integers}{\mathbb{Z}}    

\newcommand{\tuple}[1]{\boldsymbol{#1}}

\newcommand {\B}{{\mathcal B}}

\newcommand{\abs}[1]{\left|#1\right|}

\newcommand{\dual}[1]{#1^\prime}
\newcommand{\bidual}[1]{#1^{\prime\prime}}

\newcommand{\duality}[2]{\left\langle #1,#2 \right\rangle}

\newcommand{\operators}{\mathcal{B}}
\newcommand{\compactoperators}{\mathcal{K}}

\newcommand{\id}{I}

\newcommand{\inner}[1]{\left[#1\right]}

\newcommand{\extremepoints}{\mathop{\rm ex}}

\batchmode
 
\begin{document}

\title{Equivalence of multi-norms}

\author{H.\ G.\ Dales}
\address{Fylde College\\
University of Lancaster\\
Lancaster LA1 4YF\\
United Kingdom}
\email{g.dales@lancaster.ac.uk}

\author{M.\ Daws}
\address{School of Mathematics\\
University of Leeds \\ Leeds LS2 9JT,  UK}
\email{mdaws@maths.leeds.ac.uk}

\author{H.\ L.\ Pham}
\address{School of Mathematics, Statistics, and Operations Research\\
Victoria University of Wellington \\ Wellington 6140, New Zealand}
\email{hung.pham@vuw.ac.nz}

\author{P.\ Ramsden}
\address{5 Brookhill Crescent\\
Leeds LS17 8QB, UK}
\email{paul@virtualdiagonal.co.uk}

\begin{abstract} 
 The theory of multi-norms was developed by H.\ G.\ Dales and M.\ E.\ Polyakov in a memoir  that was published in \emph{Dissertationes Mathematicae}. 
In that memoir, the notion of `equivalence' of multi-norms was defined. In the present  memoir, we make a systematic study of when various pairs of
 multi-norms are mutually equivalent.

In particular, we study when $(p,q)$-multi-norms defined on spaces $\Lspace^r(\Omega)$ are equivalent, resolving most cases; we have stronger results
 in the case where $r=2$. We also show that the standard $[t]$-multi-norm defined on $\Lspace^r(\Omega)$ is not equivalent to a $(p,q)$-multi-norm 
in most cases, leaving some cases open.
 We discuss the equivalence of the Hilbert space multi-norm, the $(p,q)$-multi-norm, and the maximum multi-norm based on a Hilbert space. We calculate the value of some constants that arise.

Several results depend on the classical theory of $(q,p)$-summing operators.

\bigskip
\bigskip
\bigskip
\bigskip

\noindent\textbf{Acknowledgements}
We would like to acknowledge the financial support of EPSRC under grant EP/H019405 awarded to H.\ G.\ Dales, of the London Mathematical Society under 
grant 2901 (Scheme 2) awarded  to M. Daws, and of the Victoria University of Wellington and the  Marsden Fund (the Royal Society of New Zealand)  awarded to H.\ L.\ Pham. 

 We are also grateful to Nico Spronk for arranging a room and facilities for us to work together at  the University of Waterloo in August 2011. 
\end{abstract}


 
 \maketitle

\tableofcontents

\bigskip

\bigskip

\bigskip

 \bigskip

\newpage

\chapter{Introduction}

\noindent The theory of multi-norms was developed by H. G. Dales and M. E. Polyakov in a memoir \cite{DP}, which was published in {\it Dissertationes Mathematicae}. 
One motivation for the development of this theory was to resolve a question on the injectivity of the Banach left modules $\Lspace^p(G)$ 
over the group algebra $\Lone(G)$ of a locally compact group $G$: indeed, for $p>1$, $\Lspace^p(G)$ is injective if and only if $G$ is amenable \cite{DDPR1}.

However, the theory of multi-norms developed a life of its own: it is shown in \cite{DP} that the theory has connections with tensor norms 
on the spaces $\co\otimes E$, with the theory of $(q,p)$-summing operators, and with Banach algebras of operators, through the concept of a `multi-bounded' operator.

In \cite{DP}, there are many examples of multi-norms based on a normed space. For example, this memoir introduced the maximum and minimum multi-norms, 
the $(p,q)$-multi-norm based on a normed space (for $1\le p\le q<\infty$), the standard $t$-multi-norm based on a space $\Lspace^r(\Omega)$
 (for $1\le r\le t<\infty$), and the Hilbert multi-norm based on a Hilbert space.

There is a natural notion of `equivalence' of two multi-norms based on the same normed space, and we find it of interest to establish 
when various pairs of the known examples are indeed mutually equivalent. This often leads to questions of the equality of various classes 
of summing operators on certain Banach spaces. However, this relationship to summing operators is not entirely straightforward: results 
on such operators in the literature seem to give only partial indications. For example, in the case of $(p,q)$-multi-norms on a Hilbert 
space $H$, we would like information about $\Pi_{q,p}(H,\co)$, but classical results determine $\Pi_{q,p}(H)$.

Some easy results on the equivalences of pairs of multi-norms were given in \cite{DP} and in \cite{DDPR1}. In the present paper, we shall present 
a more systematic study of these equivalences.\s

In Chapter 1, we shall recall some background in functional analysis, including the theory of summing norms and tensor norms.  In particular, 
we shall define the Banach space  $(\Pi_{q,p}(E,F), \pi_{q,p})$ of $(q,p)$-summing  operators between Banach spaces $E$ and $F$.

In  Chapter 2, we shall give the  definition of a multi-norm, and introduce the notions of the  rate of growth   $(\varphi_n(E))$ of a multi-norm based on a space $E$ 
 and our notion of the mutual equivalence of two multi-norms based on the same normed space.   Two equivalent multi-norms have similar rates of growth, 
but the converse is, in  general,  not true.  We shall recall the definitions
of the maximum and minimum multi-norms, $(\norm_n^{\max}: n\in\N)$ and $(\norm_n^{\min}: n\in\N)$, based on a normed space.

We shall define the $(p,q)$-multi-norm $(\norm_n^{(p,q)}: n\in\N)$ based on such a space $E$ in the case where $1\leq p\leq q< \infty$, and we shall 
related these multi-norms to certain $\co$-norms on the algebraic tensor product $\co\otimes E$; for example, it is shown in Theorem 
 \ref{connection with Chevet--Saphar norm}   that the $(p,p)$-multi-norm corresponds to the Chevet--Saphar norm on $\co\otimes E$.  We shall show  in Corollary
\ref{2.9} that the multi-norms corresponding to points $(p_1,q_1)$ and $(p_2,q_2)$  are mutually equivalent if and only if  the Banach spaces 
  $\Pi_{q_1,p_1}(\dual{E},\co)$ and $ \Pi_{q_2,p_2}(\dual{E},\co) $  are the same.
  
  We shall begin to study the relations between $(p,q)$-multi-norms in $\S2.5$, giving first indications in a diagram on page \pageref{PICTURE1};
 this   diagram follows from standard results on $(q,p)$-summing  operators given by Diestel, Jarchow, and Tonge in the fine text \cite{DJT}.  In Examples 
 \ref{a calculation of (p,q)-norm on lr space} and \ref{another calculation of (p,q)-norm on lr space}, we shall calculate some explicit $(p,q)$-multi-norms; these results
will be used later to show that certain $(p,q)$-multi-norms are not mutually equivalent.  It was already known that the $(1,1)$-multi-norm 
is the maximum multi-norm on each normed space.

In $\S2.6$, we shall describe the standard $t$-multi-norm on a Banach space $L^r(\Omega, \mu)$, where $(\Omega, \mu)$ 
is a measure space; these multi-norms  played an important role in \cite{DP}, especially in connection with the theory of multi-bounded operators between Banach  
lattices.  In $\S2.7$, we shall describe the Hilbert multi-norm based on a Hilbert space; in fact, this is equal to the $(2,2)$-multi-norm based on the same space.

Our first aim in Chapter 3 is to determine when two $(p,q)$-multi-norms based on a space $L^r(\Omega, \mu)$ are mutually equivalent;
 here $1\leq p\leq q < \infty$ and $r\geq 1$.  In the case where $r=1$, complete results are given in $\S3.1$. 
 The case where $r>1$ is more difficult, and there is a clear distinction between the cases where $r<2$ and $r\geq 2$.  
To discuss the question, it is helpful to consider certain curves ${\mathcal C}_c$  and   ${\mathcal D}_c$, defined for  for $0\le c<1$; the union of 
these curves fills out the `triangle' ${\mathcal T} =\{(p,q): 1\leq p\leq q\}$.  A picture of these curves in the case where $r>2$ is given on page \pageref{PICTURE2}.

We say that two points $P_1=(p_1,q_1)$ and $P_2 = (p_2,q_2)$ in ${\mathcal T}$ are equivalent if the corresponding  $(p,q)$-multi-norms are equivalent on $L^r(\Omega)$.
In Theorem \ref{non equivalent of (p,q) mns on Lr, on different curves}, we shall show that in the `upper-left' of our diagram, $P_1$ and $P_2$ are 
mutually equivalent, and that the corresponding multi-norms are equivalent to the minimum multi-norm.  It is also shown that, otherwise,  
 $P_1$ and $P_2$ are  not equivalent whenever they lie on distinct curves  ${\mathcal D}_c$.   Thus we must turn to consideration of points on 
the same curve ${\mathcal D}_c$  (for $c < 1/\bar{r}$, where $\bar{r} = \min\{2,r\}$).  In $\S3.6$,  we shall use Khintchine's inequalities
 to show that $P_1$ and $P_2$ are not equivalent on the space $\ell^{\,r}$ whenever they are not equivalent on $\ell^{\,2}$, and hence whenever the  spaces 
  $\Pi_{q_1,p_1}(\lspace^2) $ and $\Pi_{q_2,p_2}(\lspace^2)$ are distinct; the latter question is classical, and full results are given in \cite{DJT}. 
Thus we are able to resolve most questions of mutual equivalence of $(p,q)$-multi-norms on   $L^r(\Omega, \mu)$. Results in the case where $r\in (1,2)$   are given in
Theorem  \ref{equivalent of (p,q) mns on Lr when r<2}, and those in the case where $r\geq 2$    are given in
Theorem  \ref{equivalent of (p,q) mns on Lr when r >=2}. Some cases are left open in Theorems  \ref{equivalent of (p,q) mns on Lr when r<2}
 and \ref{equivalent of (p,q) mns on Lr when r >=2}, but a full solution is given in the case where $r=2$. Some of the remaining cases will be resolved in \cite{BDP}.

Let $\Omega$ be a measure space, and take $r\ge 1$.  
In $\S3.8$, we shall consider the conjecture that the multi-norms   $(\norm^{[t]}_n)$ and $(\norm^{(p,q)}_n)$ are not mutually equivalent 
whenever $r>1$ and $\Lspace^r(\Omega)$ is infinite dimensional.  (By Theorem   \ref{2.13}, $(\norm_n^{[q]})=(\norm^{(1,q)}_n)$ on
 $\Lspace^1(\Omega)$ for   $q\ge 1$.)  We shall prove this conjecture for many, but not all,  values of $p$, $q$, and $r$ in
 Theorem \ref{equivalence of (p,q) and [q] multi-norms}.
 
 Let $H$ be a complex Hilbert space. Then the Hilbert multi-norm, the $(p,p)$-multi-norms for $p\in [1,2]$, and the maximum multi-norm based on
 $H$ are all pairwise equivalent.  In  Chapter 4, we shall discuss these norms in more detail.  For example, we know that, for each $p\in [1,2]$, 
there is a constant $C_p$ such that  $\left\Vert  \tuple{x} \right\Vert^{\max}_n = \left\Vert  \tuple{x} \right\Vert^{(1,1)}_n\le
 C_p\left\Vert  \tuple{x} \right\Vert^{(p,p)}_n$ for  all  $ \tuple{x} \in H^n$ and all $n\in\N$. In $\S4.1$, we shall show that $2/\sqrt{\pi}$ is the
 best value of $C_2$; this is a consequence of the `Little Grothendieck Theorem'.
 
In the remainder of Chapter 4, we shall  consider the best constant $c_n$, defined for each fixed $n\in\naturals$, such that 
$\left\Vert{\tuple{x}}\right\Vert^{\max}_n\le c_n\left\Vert{\tuple{x}}\right\Vert^{(2,2)}_n$ for   $\tuple{x}\in H^n$.  
We shall show that $c_2 =1$, but that $c_3>1$ in the real case; however, a rather long calculation will show that $c_3=1$ in the complex case; 
 finally, we shall show in $\S4.5$ that $c_4>1$ even in the complex case.
\medskip

Two points left open in the present work will be resolve in \cite{BDP}; see Remarks \ref{remark 3.17}  and \ref{remark 3.19}.

\medskip

We first give some background to the material of this paper, and recall some definitions from earlier works.

\section{Basic notation} The natural numbers and the integers are $\naturals$ and $\integers$, respectively. For $n \in \naturals$,
 we set $\naturals_n=\set{1,\ldots,n}$.  The complex field is $\C$; the unit circle and open unit disc in $\C$ are $\T$ and $\D$, respectively. 

Let $1\leq p\leq \infty$. Then the conjugate to $p\,$ is denoted by $p'$, so that
 $1\leq p'\leq \infty$ and satisfies $1/p+1/p'=1$.
 
 Let $(\alpha_n)$ and $(\beta_n)$ be two sequences of complex numbers.  Then $(\alpha_n)$ and $(\beta_n)$  are {\it similar\/}, written 
$\alpha_n\sim \beta_n$, if there are constants $C_1,C_2>0$ such that 
\[
	C_1 \left\vert\alpha_n\right\vert \leq \left\vert \beta_n \right\vert\leq C_2\left\vert \alpha_n\right\vert\quad(n\in\naturals)\,.
\]

\smallskip

\section{Linear  and Banach spaces}

Let $E$ be a linear space (always taken  to be over the complex field, $\C$, unless otherwise stated). 

Let $C$ be a convex set in $E$. An element $x \in C$ is an \emph{extreme   point} if $C\setminus \{x\}$ is also convex; the set of extreme points of $C$ is denoted by $\ex C$.  Let $x \in C$. Then, to show that $x \in \ex C$,
 it suffices to show that $u=0$ whenever $u\in E$ and $x\pm u \in C$.

For a linear space $E$ and $n \in \naturals$, we denote by $E^n$ the linear space direct product of $n$ copies of $E$. Let $F$ be another linear space. 
Then the linear space of all linear operators from $E$ to $F$ is denoted by ${\mathcal L}(E,F)$. The identity operator on $E$ is $\id_E$, or just $I$ when the
space is obvious.

Let $E$ be a normed space. The closed unit ball and unit sphere of $E$ are denoted by $E_{[1]}$ and $S_E$, respectively, so that $\ex E_{[1]}\subset S_E$.
 We denote the dual space of $E$ by $\dual{E}$; the action of
 $\lambda \in \dual{E}$ on an element $x \in E$ is written as $\duality{x}{\lambda}$, and the  canonical embedding of $E$ into its bidual
 $\bidual{E}$ is $\kappa_E : E \to E''$.

Let $E$ and $F$ be normed  spaces. Then $\operators(E, F)$ is the normed  space of all bounded linear operators from $E$ to $F$; it is a Banach space
 whenever $F$ is complete.  The ideal of finite-rank operators in $\operators(E,F)$ is denoted by ${\mathcal F}(E,F)$. We set  $\operators(E)=\operators(E,E)$, so that $\operators(E)$ is a unital normed algebra; it is a Banach algebra whenever 
$E$ is complete. The {\it dual\/} of  $T\in \operators(E, F)$ is $T'\in \operators(F', E')$, so that $\lV T'\rV=\lV T \rV$.  The closed ideal 
 of $\operators(E)$ consisting of the  compact operators  is denoted by $\compactoperators(E)$. 

A closed subspace $F$   of a normed space $E$ is {\it $\lambda$-complemented\/} if there exists $P\in  \operators(E)$ with $P^2= P$, with $P(E)=F$, 
and with $\lV P \rV\leq \lambda$.

We write $E\cong F$ when two Banach spaces $(E,\norm)$ and $(F,\norm)$ are {\it isometrically isomorphic}.

Let $(\Omega,\mu)$ be a measure space, and take  $p\geq 1$. Then we denote by $\Lspace^p(\Omega)=\Lspace^p(\Omega, \mu)$ (or $\Lspace^p(\mu)$) the Banach space
 of (equivalence classes of) complex-valued, $p-$integrable functions on $\Omega$, equipped with the norm $\norm_p$, which is given by
\[
\lV f \rV_p= \lp \int_{\Omega} \abs{f(x)}^p \dd \mu(x)\rp^{1/p}=\lp \int_{\Omega} \abs{f}^p \dd \mu\rp^{1/p}\quad (f \in \Lspace^p(\Omega))\,.
\]
We also define the related space $\Lspace^{\infty}(\Omega)=\Lspace^{\infty}(\Omega, \mu)$. All these spaces are Dedekind complete (complex) Banach lattices in the standard way. For some background on Banach lattices that is sufficient for our purposes, see \cite[\S1.3]{DP}.

Let $\co$ and $\lspace^p$ be the usual Banach spaces of sequences, where $1\leq p\leq \infty$.  We shall write $(\delta_n)_{n=1}^\infty$ for the standard
 unit Schauder  basis for $\co$ and $\lspace^p$ (when $p\geq 1$). 
 For $n\in\naturals$, we write $\linfty_n$ and $\lspace_n^p$ for the linear space $\C^n$ with the supremum and $\lspace^p$ norms, 
respectively; we regard each $\linfty_n$ as a subspace of $\co$, and hence regard $(\delta_i)_{i=1}^n$ as a basis for $\linfty_n$.
The space of all continuous functions on a compact Hausdorff space $K$ is denoted by $C(K)$. 

We shall several times use the following two results.

\begin{proposition}\label{1.1}  Take $p\geq1$, and let $\Omega$ be a measure space  such that  $\Lspace^p(\Omega)$ is infinite dimensional.  
Then there is an isometric lattice homomorphism $J: \ell^{\,p} \to \Lspace^p(\Omega)$ and a positive contraction of $\Lspace^p(\Omega)$ onto $J(\ell^{\,p})$,
 so that $J(\ell^{\,p})$ is $1$-complemented in $\Lspace^p(\Omega)$.
\end{proposition}\s

 \begin{proof}  This is \cite[4.1]{AV}, for example.
 \end{proof}\s

\begin{proposition}\label{1.2a} Let $E$ be an infinite-dimensional Banach space, and take $\varepsilon >0$ and $n\in\naturals$.  Then   there exist
 $x_1,\dots,x_n \in E$ such that 
 \[1-\varepsilon \leq \|x_n\| \leq 1\quad(n\in\naturals)\] and
\[ 
 \left\Vert\sum_{i=1}^n \alpha_i x_i \right\Vert \leq \left( \sum_{i=1}^n |\alpha_i|^2 \right)^{1/2} \quad  (\alpha_1,\dots,\alpha_n\in \C )\,. 
\]
\end{proposition}

\begin{proof}  By Dvoretzky's theorem, $E$ contains near-isometric copies of $\ell^{\,2}_n$, and this gives the result.  Actually, our claim is somewhat weaker, and follows from more elementary arguments, given in  \cite[Lemma~1.3]{DJT},  for example.
\end{proof}\s

 We shall refer to   Lorentz sequence spaces.  Suppose that $1\leq p\leq q< \infty$. Then the {\it Lorentz sequence space\/} $\ell^{\,p,q}$
 consists of the sequences $x=(x_n)\in \co$ such that 
\[
\lV x \rV_{p,q} = \left(\sum_{n=1}^\infty n^{(q/p)-1}(x_n^*)^q   \right)^{1/q}< \infty\,,
\]
where $x^*$ is the  decreasing re-arrangement  of $\left\vert x\right\vert $; 
the version based on $\naturals_n$  is $\ell^{\,p,q}_n$. For this definition, see \cite[p.~207]{DJT}, for example.  The spaces $(\ell^{\,p,q},\norm_{p,q})$ are Banach spaces. 
In the case where $q=p$, we obtain the usual spaces $\ell^{\,p}$ and $\ell^{\,p}_n$.\s

We shall also refer to Schatten classes.  Let $H$ be a Hilbert space.  For $p\geq 1$, the $p\,$-th {\it Schatten class\/}  ${\mathcal S}_p(H)$  
consists  of the compact operators $T\in{\mathcal K}(H)$ such that the positive operator $ (T^*T)^{p/2}$ has finite trace; the norm 
$\norm_{{\mathcal S}_p}$ on ${\mathcal S}_p(H)$  is given by 
\[
\lV T \rV_{{\mathcal S}_p} = \left({\rm tr}((T^*T)^{p/2})\right)^{1/p}\quad(T\in {\mathcal S}_p(H))\,.
\]
Equivalently, $T \in {\mathcal S}_p(H)$ if and only if the operator $\left\vert T \right\vert = (T^*T)^{1/2}$ is compact and
$\lambda = (\lambda_n) \in \ell^{\,p}$, where $(\lambda_n) $ is the (decreasing) sequence of non-zero eigenvalues of $\left\vert T \right\vert$, counted according 
to their multiplicities; now  $\lV T \rV_{{\mathcal S}_p}= \lV \lambda\rV_p$.
 The space  $({\mathcal S}_p(H), \norm_{{\mathcal S}_p})$  is a Banach operator ideal in ${\B}(H)$; the ideal ${\mathcal S}_2(H)$   coincides with the space  of 
{\it Hilbert--Schmidt operators\/} on $H$, and the corresponding norm is the {\it Hilbert--Schmidt norm\/}.

In the case where $2<p<q< \infty$, the space ${\mathcal S}_{2q/p,q}(H)$  consists of the operators $T \in {\B}(H)$ such that the above sequence of 
eigenvalues belongs to the Lorentz sequence space $\ell^{\,2q/p,q}$, and so satisfies the condition that 
\[
\lV T \rV_{{\mathcal S}_{2q/p,q}} = \left(\sum_{n=1}^\infty n^{(p/2)-1}\lambda_n^q\right)^{1/q}< \infty\,.
\]

Suppose that $H$ is an  infinite-dimensional Hilbert space, and let $(e_n)$   be an ortho\-normal sequence in $H$. For $\alpha>0$, set 
$T_\alpha e_n = n^{-\alpha}e_n\,\;(n\in \N)$, so that $T_\alpha$ extends  to an operator in ${\B}(H)$ in an obvious way. Then $T_\alpha \in {\mathcal S}_p(H)$  if and only 
if $\alpha p> 1$.   Thus ${\mathcal S}_p(H)\neq {\mathcal S}_q(H)$ whenever $p,q\geq 1$ with $p\neq q$.  Further,
$T_\alpha \in {\mathcal S}_{2q/p,q}(H)$  if and only if $\alpha > p/2q$, and so ${\mathcal S}_r(H) \neq {\mathcal S}_{2q/p,q}(H)$  whenever $r\neq 2q/p$. 

\label{Remark on Schatten class}

Now suppose that $r= 2q/p$.  We take an infinite subset $X$ of $\N$, and define $T\in  {\B}(H)$ by setting $Te_n= n^{-\alpha}e_n\,\;(n\in X)$  and   $Te_n=0\,\;(n\in\N \setminus X)$, where $q\alpha = 1- p/2$, and again extending $T$ to belong to $\operators(H)$.  Then $T  \in {\mathcal S}_r(H)$  if and only if 
  \[
  \sum_{n\in X} n^{(2/p) - 1} < \infty\,,
\]
and so $T  \in {\mathcal S}_r(H)$ for a suitably  `sparse' set $X$, noting that $(2/p)-1<0$. However, $T\in {\mathcal S}_{2q/p,q}(H)$  if and only 
if $\sum_{n\in X}1 < \infty$, and this is never the case for infinite $X$.  Thus it is always true that the spaces ${\mathcal S}_r(H)$ and
 ${\mathcal S}_{2q/p,q}(H)$  are distinct.  

  Similarly, the spaces ${\mathcal S}_{2q/p,q}(H)$ corresponding to  pairs $(p_1,q_1)$ and $(p_2,q_2)$  are distinct whenever  $(p_1,q_1)\neq (p_2,q_2)$. 
 \medskip

\section{Summing norms and summing operators} Let $E$ be a normed space, and let $n\in\N$. Following the notation of \cite{DP, DDPR1,  Jameson}, we define the 
\italic{weak $p\,$--summing norm} (for $1\leq p<\infty$) on $E^n$ by
\[
\mu_{p,n}(\tuple{x})=\sup \set{\left(\sum_{i=1}^{n}\abs{\duality{x_{i}}{\lambda}}^{p}\right)^{1/p}: \lambda \in \dual{E}_{[1]}}\,,
\]
where $\tuple{x}=(x_1,\ldots,x_n) \in E^n$. We set $\lspace^{\,p}_n(E)^{w}=(E^n, \mu_{p,n})$. It follows from \cite[p.~26]{Jameson}
 that, for each $\tuple{x}=(x_1,\ldots,x_n) \in E^n$, we have
\begin{equation}\label{weak--summing-norm 1}
\mu_{p,n}(\tuple{x})=\sup\set{\flexiblenorm{\sum_{i=1}^{n}\zeta_{i}x_i}\colon \zeta_1,\dots,\zeta_n \in \C,\ \sum_{i=1}^{n}\abs{\zeta_i}^{p'}\leq 1 }\,.
\end{equation}
We also have
\begin{align}\label{weak--summing-norm 2a}
	\mu_{p,n}(\tuple{x})&=\|T_{\tuple{x}}:\lspace_n^{p'}\to E\|\,,
\end{align}
where $T_{\tuple{x}}:(\beta_1,\ldots,\beta_n)\mapsto \sum_{i=1}^n\beta_ix_i$ belongs to $\operators(\lspace_n^{p'}, E)$. 
Thus the map $\tuple{x}\mapsto T_{\tuple{x}}$ is an isometric isomorphism from $(E^n,\mu_{p,n})$ onto $\operators(\lspace_n^{p'}, E)$. Also, let $F$ be another normed space, and take $T \in {\B}(E,F)$. Then clearly
\[
\mu_{p,n}(Tx_1,\dots,Tx_n)\leq \lV T \rV\mu_{p,n}(x_1,\dots,x_n)\quad (x_1,\dots,x_n \in E,\,n\in\N)\,.
\]
 
We note that 
\[
\mu_{p_1,n}(\tuple{x}) \geq \mu_{p_2,n}(\tuple{x}) \quad (\tuple{x} \in E^n,\,n\in\N)
\]
whenever $1\leq p_1\leq p_2<\infty$.

We also define the \italic{weak $p\,$--summing norm} of a sequence $\tuple{x}=(x_i)$ of elements in $E$ by
\begin{align*}
\mu_{p}(\tuple{x})=&\sup \set{ \lp\sum_{i=1}^{\infty}\abs{\duality{ x_{i}}{ \lambda}}^{p}\rp^{1/p}: \lambda \in \dual{E}_{[1]}}=
\lim_{n\to\infty}\mu_{p,n}(x_1,\ldots,x_n)\,;
\end{align*}
thus $\mu_p(\tuple{x})$ takes values in $[0,\infty]$. The sequences $\tuple{x}$ such that $\mu_p(\tuple{x})<\infty$ are the
 \emph{weakly $p\,$-summable sequences} in $E$, and the space of these sequences is $\lspace^p(E)^w$; see \cite[p.~32]{DJT} and \cite[p.~134]{Ryan},
 where $\mu_p(\,\cdot\,)$ is denoted by $\norm^{\textrm{weak}}_p$ and $\norm^{w}_p$, respectively. It follows from \cite[p.~26]{Jameson}
 that, for each sequence $\tuple{x}=(x_i)$  in $E$, we have
\begin{align}\label{weak--summing-norm 2}
\mu_{p}(\tuple{x})&=\sup\set{\flexiblenorm{\sum_{i=1}^{\infty}\zeta_{i}x_i}\colon (\zeta_i)\in(\lspace^{p'})_{[1]}}\,.
\end{align}

Suppose that  $1\le p\le q<\infty$. We recall from \cite[Chapter 10]{DJT} that an operator $T$ from a normed space $E$ into another normed space
 $F$ is  \emph{$(q,p)$--summing} if  there exists a constant $C$ such that
\[
	\left(\sum_{i=1}^n\lV Tx_i\rV^q\right)^{1/q}\le C\,\mu_{p,n}(x_1,\ldots, x_n)\quad (x_1,\ldots, x_n\in E,\ n\in\naturals)\,.
\]
The smallest such constant $C$ is denoted by $\pi_{q,p}(T)$. The set of these $(q,p)$--summing operators, which is denoted by $\Pi_{q,p}(E,F)$, 
 is a linear subspace of ${\B}(E,F)$ and a normed space when equipped with the norm $\pi_{q,p}$\,; $(\Pi_{q,p}(E,F), \pi_{q,p}) $ is a Banach space when $E$ and $F$
 are Banach spaces. In the case where  $p=q$, we shall write $\Pi_p$  and $\pi_p$ instead of $\Pi_{p,p}$ and $\pi_{p,p}$, respectively. The space $(\Pi_p,\pi_p)$ of all \emph{$p\,$--summing operators} has been studied by many authors; see \cite{DF, DJT, Gar3, Jameson, Ryan},  
for example. In the case where $E=F$, we shall write $\Pi_{q,p}(E)$ instead of $\Pi_{q,p}(E,E)$, $\pi_{q,p}(E)$  instead of $\pi_{q,p}(E,E)$, \ldots etc.

A basic inclusion theorem \cite[Theorem 2.8]{DJT} shows that $\Pi_p(E,F)\subset \Pi_q(E,F)$ whenever $1\le p\leq q<\infty$. 
A more complicated inclusion theorem \cite[Theorem 10.4]{DJT} will be used in Theorem \ref{comparing (p,q) multi-norms}, given below.

 Let us make some obvious remarks about summing operators. Let $E$, $F$, and $G$ be Banach spaces, and take $T\in {\B}(E,F)$ and  $1\leq p\leq q<\infty$. Then:  
\begin{itemize}
\item  $T\in \Pi_{q,p}(E,F)$ if and only if $S\,\circ \,T\in \Pi_{q,p}(E,G)$,
  with equal norm, for any isometry $S:F\rightarrow G\,$;
\item   $T\in \Pi_{q,p}(E,F)$ if and only if $T\,\circ\,  P \in \Pi_{q,p}(G,F)$,   with equal norm, for any contractive projection $P:G\rightarrow E$.
\end{itemize}
These remarks will be used implicitly at some future points. 

  The Pietsch domination theorem can be stated in the following way ({\it cf.}  the discussion
 after \cite[Theorem~6.18]{Ryan}).  Take $p\geq 1$.
 A map $T\in\mathcal B(E,F)$ is $p\,$--summing if and only if we can find a non-empty, compact Hausdorff space $K$ and a
probability measure $\mu$ on $K$, together with operators $V\in\mathcal B(E,C(K))$ and $U\in\mathcal B(L^p(\mu),\ell^{\,\infty}(I))$ such that the following diagram
commutes:
\[ \xymatrix{ E \ar[r]^-{T} \ar[d]^V & F \ar@{^{(}->}[r] & \ell^{\,\infty}(I) \\
C(K) \ar[rr] && L^p(\mu) \ar[u]^U\,.} \]
Here the map $C(K)\rightarrow L^p(\mu)$ is the canonical inclusion map, $I$ is a suitable index set, and $\ell^{\,\infty}(I)$ can be replaced by
 any injective Banach space $G$ such that $F$ is isometric to a subspace of $G$.\s 
 
Let $E$ and $F$ be   normed spaces. Take  $n\in\N$,  and  suppose that $1\le p\le q<\infty$. Then  the {\it $(q,p)$-summing constants} of the operator
 $T\in {\B}(E, F)$  are the numbers
\[
\pi_{q,p}^{(n)}(T) : = \sup\!\left\{\left(\sum_{i=1}^n \lV Tx_i\rV^{\,q}\right)^{1/q}\! :\ x_1,\dots, x_n\in E,\, \mu_{p,n}(x_1,\dots, x_n)\leq 1\right\}.
\]
Further,  $\pi_{q,p}^{(n)}(E)= \pi_{p,q}^{(n)}(I_{E})$; these are the {\it $(q,p)$-summing constants} of the normed space $E$.  
We write $\pi_p^{(n)}(T)$ for  $\pi_{p,p}^{(n)}(T)$ and $\pi_p^{(n)}(E)$ for $\pi_{p,p}^{(n)}(E)$.  It follows that   
\begin{equation}\label{(3.10a)}
 \pi_{q,p}^{(n)}(E )= \sup\left\{\left(\sum_{i=1}^n \left\Vert x_i\right\Vert^{\,q}\right)^{1/q} :\ x_1,\dots, x_n\in E,\,  \mu_{p,n}(x_1,\dots, x_n)\leq 1\right\}\,.
\end{equation}

\medskip

\begin{proposition}\label{1.10}  Suppose that $1 \leq p\leq q< \infty$ and that $n\in\N$.  Then:
\begin{enumerate}
\item[{\rm (i)}] $\pi_{q,p}^{(n)}(E)\leq n^{1/q}$ for each normed space $E$;\s

\item[{\rm (ii)}]   $\pi_{q,p}^{(n)}(E)=  n^{1/q}$ for each infinite-dimensional normed space $E$  whenever $p\geq 2$\,;\s

\item[{\rm (iii)}]   $\pi_{q,p}^{(n)}(E)\ge  n^{1/2-1/p+1/q}$ for each infinite-dimensional normed space $E$  whenever $p\leq 2$\,;\s

\item[{\rm (iv)}] $\pi_{q,p}^{(n)}(\ell^{\,s}) =n^{1/q}$ whenever $s\in[1,\infty]$ and $p\geq \min\set{s',2}$\,.\s
\end{enumerate}
\end{proposition}

\begin{proof} (i) This is immediate.\s

(ii)  Take $\varepsilon >0$, and choose $x_1,\dots,x_n\in E$ to be as specified in Proposition \ref{1.2a}. 
 For each $\zeta_1,\dots,\zeta_n\in \C$  with $\sum_{i=1}^n\left\vert\zeta_i\right\vert^{p'} \leq 1$, we have
\[
\left\Vert \sum_{i=1}^n\zeta_ix_i\right\Vert \le
 \left(\sum_{i=1}^n\left\vert\zeta_i\right\vert^{2}\right)^{1/2}\leq \left(\sum_{i=1}^n\left\vert\zeta_i\right\vert^{p'}\right)^{1/p'}\,
\]
because $p'\le 2$. Thus, by equation \eqref{weak--summing-norm 1},  $\mu_{p,n}(x_1,\dots, x_n)\leq 1$, and so
\[
	\pi_{q,p}^{(n)}(E)\geq (1-\varepsilon)n^{1/q}\,. 
\]
This holds true for each $\varepsilon >0$, and so $\pi_{q,p}^{(n)}(E)\geq n^{1/q}$. By (i), $\pi_{q,p}^{(n)}(E)= n^{1/q}$.\s

(iii) Take $\varepsilon >0$ and choose $x_1,\dots,x_n\in E$ as in (ii). Now, since $p'\ge 2$, the argument in (ii) shows that 
$\mu_{p,n}(x_1,\dots, x_n)\leq n^{1/2-1/p'}$, and so
\[
	\pi_{q,p}^{(n)}(E)\geq (1-\varepsilon)n^{1/2-1/p+1/q}
\] 
for every $\varepsilon>0$. Hence $\pi_{q,p}^{(n)}(E)\ge  n^{1/2-1/p+1/q}$.\s

(iv) In the case where $p\ge 2$, this follows from (ii). Now suppose that $p\ge s'$. Take $x_j=\delta_j\,\;(j\in\N_n)$. 
As in the proof of (ii), we see that $\mu_{p,n}(x_1,\dots, x_n)\leq 1$, and so  $\pi_{q,p}^{(n)}(\ell^{\,s}) \geq n^{1/q}$.\s
\end{proof}\s

We shall also need the following simple interpolation result.

\begin{proposition}\label{interpolating pi sequence}
Let $E$ be a normed space. Suppose that $1\le p\le q_1<q<q_2<\infty$, so that
\[
	\frac{1}{q}=\frac{1-\theta}{q_1}+\frac{\theta}{q_2}
\]
for some $\theta\in (0,1)$. Then
\[
	\pi_{q,p}^{(n)}(E)\le \left(\pi_{q_1,p}^{(n)}(E)\right)^{1-\theta}\,\cdot\,\left(\pi_{q_2,p}^{(n)}(E)\right)^\theta\quad(n\in\naturals)\,.
\]
\end{proposition}
\begin{proof}
Take $x_1,\ldots,x_n\in E$ with $\mu_{p,n}(x_1,\dots, x_n)\leq 1$. Using a version of H\"{o}lder's inequality, we see that
\begin{align*}
	\left(\sum_{i=1}^n \left\Vert x_i\right\Vert^{\,q}\right)^{1/q}&\le \left(\sum_{i=1}^n \left\Vert x_i\right\Vert^{\,(1-\theta)\,[{q_1}/{(1-\theta)}]}\right)^{(1-\theta)/q_1}\,\cdot\, \left(\sum_{i=1}^n \left\Vert x_i\right\Vert^{\,\theta\,[{q_2}/{\theta}]}\right)^{\theta/q_2}\\
	&\le \left(\pi_{q_1,p}^{(n)}(E)\right)^{1-\theta}\,\cdot\,\left(\pi_{q_2,p}^{(n)}(E)\right)^\theta\,,
\end{align*}
which implies the result.
\end{proof}

\medskip

\section{Tensor norms}  Let $E$ and $F$ be linear spaces. Then $E\otimes F$ denotes the algebraic tensor product of $E$ and $F$.

 Let $E_1,E_2,F_1,F_2$ be linear spaces, and take $S\in {\mathcal L}(E_1,E_2)$ and $T\in {\mathcal L}(F_1,F_2)$. Then $S\otimes T$ denotes the unique linear operator from $E_1\otimes F_1$ to 
$E_2\otimes F_2$   such that 
\[
(S\otimes T)(x\otimes y)= Sx\otimes Ty\quad(x\in E_1,\,y\in F_1)\,.
\]
In particular, we have defined $\lambda\otimes \mu$ whenever $\lambda$ and $\mu$ are linear  functionals on $E_1$ and $F_1$, respectively.

Now suppose that $E$ and $F$ are normed spaces.  We shall discuss various norms on the space $E\otimes F$. For the definitions  and  properties stated below,  
see \cite[Chapter I]{DF}, \cite{DJT}, and \cite[Section~6.1]{Ryan}, for example.

We shall often regard $E\otimes F$ as a linear subspace of ${\B}(F',E)$ by setting 
\[
(x\otimes y)(\lambda) = \langle y,\,\lambda\rangle x\quad
(x\in E,\,y\in F,\,\lambda \in F')\,;
\]
in this way, we identify  $E\otimes F$ with ${\mathcal F}(F',E)\subset {\B}(F',E)$.  Similarly, we can identify $E\otimes F$ with ${\mathcal F}(E',F)\subset {\B}(E',F)$.

\smallskip

The injective and projective tensor norms on $E\otimes F$ are denoted by $\norm_\varepsilon$ and $\norm_\pi$, respectively; 
the completions of $E\otimes F$ with  respect to these norms are denoted by   
\[
	(E\injectivetensor F, \norm_\varepsilon)\quad \textrm{and}\quad (E\projectivetensor F, \norm_\pi)\,,
\] 
respectively.  

For $\mu \in (E\projectivetensor F)'$, define $T_\mu $ by
\[
	\langle y,\,T_\mu x\rangle = \langle x\otimes y,\,\mu\rangle\quad (x\in E,\,y\in F)\,.
\] 
Then $T_\mu x \in F'\,\;(x\in E)$, $T_\mu \in \B(E,F')$, and the map  
\begin{equation}\label{(1.3)}
\mu\mapsto T_\mu\,,\quad(E\projectivetensor F)'\to  \B(E,F')\,,
\end{equation}
is an isometric isomorphism, and so $(E\projectivetensor F)'\cong  \B(E,F')$.

 \smallskip
 
A norm $\norm$ on $E\otimes F$  is a {\it sub-cross-norm\/}  if 
\[
\left\Vert x\otimes y\right\Vert \leq \left\Vert x \right\Vert \left\Vert y\right\Vert\quad(x\in E,\,y\in F)
\]
 and a  {\it  cross-norm\/} if
\[
\left\Vert x\otimes y\right\Vert = \left\Vert x \right\Vert \left\Vert y\right\Vert\quad(x\in E,\,y\in F)\,.
\]
Further, a sub-cross-norm $\norm$ on $E\otimes F$ is a {\it reasonable cross-norm} if the linear functional 
$\lambda \otimes \mu$ is bounded and $\left\Vert \lambda \otimes \mu\right\Vert \leq \left\Vert \lambda \right\Vert\left\Vert \mu\right\Vert$ 
 for each $\lambda \in E'$ and $\mu\in F'$. In fact, a sub-cross-norm  is reasonable if and only if 
\[
	\left\Vert{z}\right\Vert_\varepsilon\le\left\Vert{z}\right\Vert\le \left\Vert{z}\right\Vert_\pi \quad (z\in E\otimes F)\,.
\] 

\smallskip

Let $\alpha$ be a reasonable cross-norm on $E\otimes F$.  Then the completion of the normed space $(E\otimes F, \alpha)$ is denoted
 by $E\,\widehat{\otimes}_\alpha\,F$.  The map in (\ref{(1.3)}) identifies the dual of $E\,\widehat{\otimes}_\alpha\,F$ with a linear subspace of $\B(E,F')$.
 
 \smallskip

A \emph{uniform cross-norm} is an assignment of a cross-norm  to $E\otimes F$ for all pairs of Banach spaces $(E,F)$, with the property
that, for each  operator $S\in\mathcal B(E_1,E_2)$ and $T\in\mathcal B(F_1,F_2)$, the linear map $S\otimes T:E_1\otimes F_1 \rightarrow E_2\otimes F_2$  
 is bounded, with norm at most $\|S\|\|T\|$, with respect to the assigned norms  on $E_1\otimes F_1$ and $E_2\otimes F_2$.    
The projective and injective tensor norms are uniform cross-norms.  For further details, see
  \cite[$\S12.1$]{DF} and \cite[$\S6.1$]{Ryan}.
  
  \smallskip

For Banach spaces $E$ and $F$, the (right) \italic{Chevet--Saphar norm}
$d_p$ on $E\otimes F$ is defined as 
\[
d_p(z)=\inf_{n\in\N} \set{ \mu_{p',n}(x_1,\ldots,x_n)\lp \sum_{i=1}^n\left\Vert y_i\right\Vert^p\rp^{1/p}: z=\sum_{i=1}^nx_i\otimes y_i\in E\otimes F}\,;
\]
see, for example, \cite[Chapter 12]{DF} and \cite[p. 135]{Ryan}. This norm is a reasonable cross-norm; in fact, it is a uniform cross-norm.

Given a tensor $z\in E\otimes F$, let $z^t$ be the `flipped' tensor in $F\otimes E$.  We define the left Chevet--Sapher norm $g_p$ by
$g_p(z) = d_p(z^t)$ \cite[p.\ 135]{Ryan}.

\smallskip

Let $\alpha$ be a uniform cross-norm.  Following \cite[Chapter~7]{Ryan},
we define the \emph{Schatten dual tensor norm} $\alpha^s$ by
\[ \alpha^s(z) = \sup\{ |\ip{z}{\lambda}| : \lambda\in E'\otimes F',\,
\alpha(\lambda)\leq 1\} \quad (z\in E\otimes F)\,, \]
using the obvious dual pairing between $E\otimes F$ and $E'\otimes F'$. In general, this does not lead to a satisfactory duality theory, as it may
happen that $(\alpha^s)^s \not= \alpha$.  To correct this, we define the \emph{dual tensor norm} $\alpha'$ by first setting $\alpha'=\alpha^s$ on
$E\otimes F$ whenever $E$ and $F$ are finite-dimensional spaces, and then extend $\alpha'$ to all Banach spaces by finite generation.  The details are
technical, and we refer the reader to \cite[Chapter~II]{DF} and \cite[Chapter~7]{Ryan}   for further information.

\smallskip

We say that a uniform cross-norm $\alpha$ is \emph{totally accessible} if the embedding of $E \otimes F$ into $(E' \proten_{\alpha'} F')'$ induces the norm $\alpha$
on $E\otimes F$ for all Banach spaces $E$ and $F$.  That is, $\alpha$ is totally accessible if $(\alpha')^s = \alpha$.
In the  case where this is true under the additional hypothesis that at least one of the
two spaces $E$ or $F$ has the metric approximation property, $\alpha$ is said to
be \emph{accessible}.  For us, it is important that $\co$ has the metric
approximation property and that many norms $\alpha$ on spaces $E\otimes F$ are accessible.
For example, by \cite[Proposition~7.21]{Ryan},  $g_p$ is an accessible norm for any $p$
(and hence the same is true of $d_p$).

\smallskip

Let $E$ and $F$ be normed spaces. A bounded  operator $T:E\rightarrow F$ is \emph{$p\,$-integral} if it gives a bounded linear functional on the space 
$E\, \proten_{g_p'}\, F'$, and the $p\,$-{\it integral norm\/} of $T$, denoted by $i_p(T)$, is defined to be the norm of this functional;
see \cite[$\S7.3$]{Ryan}. Such maps have a representation theory which is  analogous to the Pietsch
representation theorem for $p\,$-summing operators; see \cite[Theorem~7.22]{Ryan}, for example.   Indeed, we can factor such an operator $T$  as 
\[
 \xymatrix{ E \ar[r]^T \ar[d] & F \ar[r]^{\kappa_F} & F'' \\
C(K) \ar[rr] && L_p(\mu) \ar[u]  \,.}
\] 
Comparing this to the factorisation result  above for $p\,$-summing maps, we see that the only difference is that here we embed $F$ into its bidual $F''$, but for a
$p\,$-summing map, we embed $F$ into an injective space.  Thus every $p\,$-integral map is $p\,$-summing.  In the special case where $F=\co$,
 we know that $F''=\ell^{\,\infty}$, and so we conclude with the following proposition.

\begin{proposition}\label{1.10c} Let $E$ be a normed space.  Then  the classes of $p\,$-summing and $p\,$-integral maps from $E$ to $\co$ coincide, 
with equal norms.\qed
\end{proposition} \medskip
 
\newpage

\chapter{Basic facts on multi-normed spaces}

\section{Multi-normed spaces}\label{Multi-normed spaces}
 The following definition is due to Dales and Polyakov. For a full account of the theory of multi-normed spaces, see \cite{DP}, and, 
for further work, see \cite{DDPR1}. The main definition is taken from \cite[Definition 2.1]{DP}.

\begin{definition} 
Let $(E, \norm)$ be a normed space, and let $(\norm_n: n \in \naturals)$ be a sequence such that $\norm_n$ is
 a norm on $E^n$ for each $n \in \naturals$, with $\norm_1=\norm$ on $E=E^1$. Then the sequence $(\norm_n: n \in \naturals)$
 is a \italic{multi-norm}  if the following axioms hold (where in each case the axiom is required to hold for all $n\in\N$ and all $x_1,\ldots, x_n \in E$):\smallskip
 
 \noindent (A1)  $\left\Vert(x_{\sigma(1)},\dots,x_{\sigma(n)})\right\Vert_n=\left\Vert(x_1,\dots,x_n)\right\Vert_n$ for each 
permutation $\sigma$ of $\naturals_n$;\smallskip

\noindent(A2) $\left\Vert(\alpha_1x_{1},\ldots,\alpha_nx_{n})\right\Vert_n\leq  \max_{i \in \naturals_n}\abs{\alpha_i} \left\Vert(x_{1},\ldots,x_{n})\right\Vert_n \;
(\alpha_1,\ldots,\alpha_n \in \C^n)$;\smallskip

\noindent(A3) $\left\Vert(x_1,\ldots,x_{n},0)\rV_{n+1}=\left\Vert(x_1,\ldots,x_{n})\right\Vert_{n}$;\smallskip

\noindent(A4) $\left\Vert(x_1,\ldots,x_{n-1},x_{n},x_{n})\right\Vert_{n+1}=\left\Vert(x_1,\ldots,x_{n-1},x_{n})\right\Vert_{n}$.\smallskip

\noindent The space $E$ equipped with a multi-norm is a \italic{multi-normed space}, written in full as $((E^n,\norm_n):\ n\in\naturals)$; we say that 
the multi-norm is \emph{based} on $E$.
\end{definition}\smallskip

In the case where $E$ is a Banach space, $(E^n, \norm_n)$ is a Banach space for each $n\in\naturals$, and we refer to a {\it multi-Banach space\/}.

Let $(\norm^1_n: n \in \naturals)$ and $(\norm^2_n: n \in \naturals)$ be two multi-norms based on a normed space $E$. Then, following \cite[Definition 2.23]{DP},
 we write
\[
	(\norm^1_n)\le (\norm^2_n)\,
\]
if $\left\Vert{\tuple{x}}\right\Vert^1_n\le \left\Vert\tuple{x}\right\Vert^2_n$ for each $\tuple{x}\in E^n$ and $n\in\naturals$, and write 
\[
	(\norm^1_n) =  (\norm^2_n)
\]
if  
$\left\Vert{\tuple{x}}\right\Vert^1_n= \left\Vert\tuple{x}\right\Vert^2_n$ for each $\tuple{x}\in E^n$ and $n\in\naturals$. The multi-norm $(\norm^2_n: n \in \naturals)$  \emph{dominates}
 a multi-norm $(\norm^1_n: n \in \naturals)$ if there is a constant $C>0$ such that 
\begin{align}\label{Eq: Multi-normed spaces 1}
	\left\Vert \tuple{x}\right\Vert^1_n  \le C\left\Vert \tuple{x}\right\Vert^2_n\quad(\tuple{x}\in E^n,\ n\in\naturals)\,,
\end{align}
and, in this case, we write 
\[
	(\norm^1_n)\preccurlyeq (\norm^2_n)\,. 
\]
The two multi-norms are \emph{equivalent}, written
\[
	(\norm^1_n)\cong (\norm^2_n)\,,
\] if each dominates the other; if the two multi-norms are not equivalent, we shall write $(\norm^1_n)\not\cong (\norm^2_n)$. 

We shall be interested in determining when one multi-norm dominates the other (and, in this case, in the best value of the constant $C$ in equation
 \eqref{Eq: Multi-normed spaces 1}) and when two multi-norms are equivalent.  
 
 Let $((E^n,\norm_n): n\in\N)$ be a multi-normed space.  For $n\in\N$, define
 \[
 \varphi_n(E) =\sup\{ \left\Vert (x_1,\dots,x_n)\right\Vert_n : x_1,\dots,x_n \in E_{[1]}\}.
 \]
 Then the sequence $(\varphi_n(E))$ is the {\it rate of growth\/} corresponding to the multi-norm \cite[Definition 3.1]{DP}.  
This sequence depends on both $E$ and the specific multi-norm.\medskip

\section{Multi-norms as a tensor norms} 
 In \cite{DDPR1}, we explained how multi-norms corres\-pond to certain tensor norms. We recall this briefly; details are given in \cite[\S 3]{DDPR1}.

\begin{definition} 
Let $E$ be a normed space. Then a norm $\norm$ on $\co\otimes E$ is a \emph{$\co$-norm} if  $\lV{\delta_1\otimes x}\rV=\lV{x}\rV$
 for each $x\in E$  and if the linear operator $T\otimes \id_E$ is bounded on $(\co\otimes E,\norm)$, with norm at
most $\|T\|$, for each $T\in\compactoperators(\co)$.

Similarly, a norm $\norm$ on $\ell^{\,\infty}\otimes E$ is an \emph{$\ell^{\,\infty}$-norm} if $\lV{\delta_1\otimes x}\rV=\lV{x}\rV$  for each $x\in E$  and if 
$T\otimes \id_E$ is bounded on $(\ell^{\,\infty}\otimes E,\,\norm)$, with norm at most $\|T\|$, for each $T\in\compactoperators(\ell^{\,\infty})$.
\end{definition}\smallskip

By \cite[Lemma 3.3]{DDPR1}, each $\co$-norm on $\co\otimes E$  and each $\ell^{\,\infty}$-norm on $\ell^{\,\infty}\otimes E$ is a  reasonable cross-norm.

 Suppose that $\norm$ is a $\co$-norm on $\co\otimes E$, and set
\[
	\lV{(x_1,\ldots,x_n)}\rV_n=\flexiblenorm{\sum_{i=1}^n\delta_i\otimes x_i}\quad(x_1,\ldots, x_n\in E,\ n\in\naturals)\,.
\]
Then $(\norm_n:\ n\in\naturals)$ is a multi-norm based on $E$. 

A more general and detailed version of the following theorem is given as \cite[Theorem 3.4]{DDPR1}.

\begin{theorem}\label{multi-norm as tensor} 
Let $E$ be a normed space. Then the above construction defines a bijection from the family of $\co$-norms on $\co\otimes E$ to the family of multi-norms 
based on $E$. \enproof
\end{theorem}

We shall be interested in uniform cross-norms, restricted to tensor products of
the form $\co\otimes E$.  This motivates us to give the following definition.\s

\begin{definition} 
A \emph{uniform $\co$-norm} is an assignment of a $\co$-norm $\norm$ to $\co\otimes E$ for all Banach spaces $E$ such that  
  the operator $I\otimes T:\co\otimes E\rightarrow \co\otimes F$ is bounded with respect to the two corresponding norms, with norm $\|T\|$, for each normed spaces $E$ and $F$ and each   $T\in\mathcal B(E,F)$.
\end{definition}\s

Let $E$ be a normed space. As in \cite{DP} and \cite{DDPR1}, there is a \emph{maximum multi-norm} based on $E$; it is denoted by
 $(\norm_n^{\max}: n \in \naturals)$, and is defined by the property that 
\[
	\left\Vert{\tuple{x}}\right\Vert_n\le \left\Vert{\tuple{x}}\right\Vert_n^{\max}\quad(\tuple{x}\in E^n,\ n \in \naturals) 
\] 
for every multi-norm $(\norm_n: n \in \naturals)$ based on $E$. This multi-norm corres\-ponds to the projective tensor norm 
$\norm_\pi$ on $\co\otimes E$ via the above corres\-pondence. By \cite[Theorem 3.33]{DP}, for each $n\in\naturals$ and $\tuple{x}=(x_1,\ldots,x_n)\in E^n$, we have
\begin{equation}\label{max-mn}
\left\Vert{\tuple{x}}\right\Vert^{\max}_n=\sup \set{ \abs{\sum_{i=1}^n\duality{ x_i}{\lambda_i}}: 
\lambda_1,\ldots, \lambda_n \in \dual{E},\, \mu_{1,n}(\lambda_1,\ldots, \lambda_n)\leq 1}\,.
\end{equation}

The rate of growth sequence corresponding to the maximum multi-norm based on $E$ is intrinsic to $E$; it is denoted  
 by $(\varphi^{\max}_n(E))$. The value of this sequence for various examples is calculated in \cite[$\S3.6$]{DP}.  For example, for $n\in\N$, we have 
$\varphi^{\max}_n(\ell^{\,p}) = n^{1/p}$ for $p\in [1,2]$ and  $\varphi^{\max}_n(\ell^{\,p}) = n^{1/2}$ for $p\in [2,\infty]$  \cite[Theorem 3.54]{DP}. It is shown in \cite[Theorem 3.58]{DP} that $\sqrt{n}\le\varphi_n^{\max}(E)\le n$ ($n\in\N$) for  each infinite-dimensional Banach space $E$. \s

Similarly, there is a  \emph{minimum multi-norm} $(\norm^{\min}_n: n \in \naturals)$ based on a normed space $E$. As in \cite[Definition 3.2]{DP}, it is defined by 
\[
	\left\Vert{(x_1,\ldots,x_n)}\right\Vert^{\min}_n=\max_{i\in\naturals}\left\Vert{x_i}\right\Vert\quad(x_1,\ldots, x_n\in E)\,.
\]
The minimum multi-norm based on $E$ corresponds to the injective tensor norm $\norm_\varepsilon$ on $\co\otimes E$ in the above correspondence, 
and so the minimum multi-norm on $\co\otimes E$  is the relative norm on ${\mathcal F}(E',\co)$ from $({\B}(E',\co), \norm)$. Of course, the rate of 
growth sequence of the minimum multi-norm is the constant sequence $1$.

\medskip

\section{The $(p,q)$-multi-norm}
We recall the definition of the $(p,q)$-multi-norm based on a normed space $E$. 

 Let $E$ be a normed space, and take  $p,q\in [1,\infty)$. Following \cite[Definition 4.1.1]{DP} and \cite[\S 1]{DDPR1}, for each $n \in \naturals$
 and each $\tuple{x}=(x_1,\ldots, x_n) \in E^n$, we define
\[
\left\Vert{\tuple{x}}\right\Vert^{(p,q)}_n=\sup \left\{\lp\sum_{i=1}^n\abs{\duality{x_i}{\lambda_i}}^{q}\rp^{1/q}: \mu_{p,n}(\tuple{\lambda})\leq 1 \right\}\,,
\]
where the supremum is take over all $\tuple{\lambda}=(\lambda_1,\ldots,\lambda_n) \in (\dual{E})^n$.  
 It is clear that $\norm^{(p,q)}_n$ is a norm on $E^n$. As noted in \cite[Theorem 4.1]{DP},  $( {\norm^{(p,q)}_n}: n \in \naturals)$
 is a multi-norm based on $E$ whenever $1\le p\le q<\infty$.\smallskip

\begin{definition}
Let $E$ be a normed space, and suppose that $1 \leq p\leq q < \infty$. Then the multi-norm $({\norm^{(p,q)}_n}: n \in \naturals)$
 described above is the \italic{$(p,q)$-multi-norm} over $E$. The corresponding $c_0$-norm on $c_0\otimes E$ is $\norm^{(p,q)}$.
 
 The rate of growth sequence corresponding to the above $(p,q)$-multi-norm is denoted by $(\varphi_n^{(p,q)}(E))$,
 as in \cite[Definition 4.2]{DP}.  
\end{definition}\s

We shall use the following remark, from \cite[Proposition 4.3]{DP}.  Let $F$ be a $1$-comp\-lemented subspace of a Banach space $E$, and take $x_1,\dots,x_n\in F$.
Then the value of $\lV (x_1,\dots,x_n) \rV^{(p,q)}_n$ is independent of whether it be calculated with respect to $F$ or $E$.

Let $E$ and $F$ be normed spaces, and take $T \in {\B}(E,F)$. Then clearly
\[
\lV( Tx_1,\dots,Tx_n)\rV^{(p,q)}_n\leq \lV T \rV\lV (x_1,\dots,x_n)\rV^{(p,q)}_n\quad (x_1,\dots,x_n \in E,\,n\in\N)\,.
\]

The following theorem refers to {\it multi-bounded sets\/} in and {\it multi-bounded operators} on multi-normed spaces; for background  information, see \cite{DDPR1}, and, 
in more  detail, \cite[Chapter 6]{DP}.  For example, the {\it multi-bound\/} of a multi-bounded set $B$ is defined by
\[
c_B =\sup_{n\in\N}\{\left\Vert (x_1,\dots,x_n)\right\Vert_n : x_1,\dots,x_n \in B\}\,.
\]

\begin{theorem}\label{summing}  Let $E$ be a normed space, and suppose that $1\le p\le q<\infty$. Then the $(p,q)$-multi-norm induces the norm on $\co\otimes E$ given by
embedding $\co\otimes E$ into $\Pi_{q,p}(E',\co)$. This norm is a uniform $\co$-norm on $\co\otimes E$.
\end{theorem}
\begin{proof} We start by observing that \cite[Theorem~4.2]{DDPR1} shows that the $\ell^{\,1}$-norm (that is, the dual multi-norm) on 
$\ell^{\,1}\otimes E'$ norms $\co\otimes E$. The converse is also true, so that we have $\ell^{\,1}\otimes E' \subset (\co\otimes E)'$, and the embedding is an 
isometry.

In \cite[Definition~5.4]{DDPR1},  we defined $\B_{p,q}(\ell^{\,1},E)$ to be the set of  those  $T\in{\B}(\ell^{\,1}, E)$ which are multi-bounded 
  when  we take  the minimum multi-norm based on $\ell^{\,1}$ and the   $(p,q)$-multi-norm based on $E$. 
The norm on the space $\B_{p,q}(\ell^{\,1},E)$ is denoted by $\alpha_{p,q}$, so that $\alpha_{p,q}(T)$ is equal to the multi-bound $c_B$ of the set 
$B:= \{T(\delta_k) : k\in\naturals \}$. 
 Thus the natural inclusion of $\co\otimes E$ into $\B_{p,q}(\ell^{\,1},E)$ (where we identify $\co\otimes E$  with ${\mathcal F}(\ell^{\,1}, E)$) induces the $(p,q)$-multi-norm on $\co\otimes E$.   
It follows from \cite[Proposition~5.5]{DDPR1}   that $T$ belongs to  $\B_{p,q}(\ell^{\,1},E)$ if and only if the dual operator  $T'$ belongs to 
$\Pi_{q,p}(E',\ell^{\,\infty})$, with equal norms.  The combination of these two results
 immediately gives the result.

It remains to show that the resulting norm is a uniform $\co$-norm.  Let $T\in\mathcal B(E,F)$, and consider the operator $I\otimes T:\co\otimes E
\rightarrow \co\otimes F$.  It is easy to see that we have the following commutative diagram:
\[ \xymatrix{ \co\otimes E \ar[r]^-{I\otimes T} \ar[d] & \co\otimes F \ar[d] \\
\Pi_{q,p}(E',\co) \ar[r]^-{\varphi} & \Pi_{q,p}(F',\co) \,.} \]
Here $\varphi$ is the map $S\mapsto S\circ T'$.  Since  the vertical arrows are isometries, it suffices to show that $\|\varphi\| \leq \|T\| = \|T'\|$. 
 But this follows immediately from properties of $(q,p)$-summing maps; see \cite[Proposition~10.2]{DJT}.
\end{proof}

\begin{remark}\label{remark1}
A refinement of the above argument shows that, for each normed space $E$, the $(p,q)$-multi-norm based on $\dual{E}$ induces the norm on
 $\co\otimes \dual{E}$ given by embedding $\co\otimes \dual{E}$ into $\Pi_{q,p}(E,\co)$.
\end{remark}\s

It follows immediately from Theorem \ref{summing} and the closed graph theorem that   the $(p_1,q_1)$- and $(p_2,q_2)$-multi-norms are equivalent on $E$
whenever 
\[
	\Pi_{q_1,p_1}(\dual{E},\co) = \Pi_{q_2,p_2}(\dual{E},\co)\,;
\] 
moreover, $\co$ can be replaced by any infinite-dimensional $C(K)$-space. The converse is also true; this is a special case of the following theorem.\s

\begin{theorem}\label{tool for general spaces}
Let $E$ be a Banach space, and take $p_1,q_1,p_2,q_2$ such  that $1\leq p_1\leq q_1<\infty$ and $1\leq p_2\leq q_2<\infty$. Suppose that the $(p_1,q_1)$- and
 $(p_2,q_2)$-multi-norms are mutually equivalent on $E$. Then
\[
	\Pi_{q_1,p_1}(\dual{E},F) = \Pi_{q_2,p_2}(\dual{E},F)  
\]
for every Banach space $F$.
\end{theorem}

\begin{proof}
Let  $F$ be a Banach space. It is standard that there is an isometry $\varphi : F\to \ell^{\,\infty}(I)$ for some index set $I$.  For each finite subset $A\subset I$,
let $P_A:\ell^{\,\infty}(I)\rightarrow \co$ be the projection map 
\[
	\ell^{\,\infty}(I) \rightarrow \ell^{\,\infty}(A)\subset \co\,.
\] 

Assume towards a contradiction that we have $\Pi_{q_1,p_1}(\dual{E},F) \nsubset \Pi_{q_2,p_2}(\dual{E},F)$, and take
 $T\in\Pi_{q_1,p_1}(\dual{E},F) \setminus \Pi_{q_2,p_2}(\dual{E},F)$. From the definition of the $(q,p)$-summing norm, we see that
\[ 
\pi_{q,p}(T) = \pi_{q,p}(\varphi\,\circ\,T) =
\sup_{A} \pi_{q,p}(P_A\,\circ\,\varphi\,\circ\, T)\,; 
\]
here we take the supremum over all finite subsets $A\subset I$. Hence  there exists a sequence $(A_n)$ of finite subsets of $I$ such that
\[
	n\cdot \pi_{q_1,p_1}(T_n)<\pi_{q_2,p_2}(T_n)\quad(n\in\N)\,,
\]
where $T_n:=P_{A_n}\,\circ\,\varphi\,\circ\, T:\dual{E}\to\linfty(A_n)\subset\co\,$.

Take $n\in\naturals$. Since 
$T_n\in {\mathcal F}(E', \co)$,  the operator $T_n$ is induced by a tensor $\tau_n\in \co\otimes \bidual{E}$.   
Remark \ref{remark1} and the previous paragraph then show that
\[
	n\cdot\big\|\tau_n\big\|^{(p_1,q_1)}_{\co\otimes \bidual{E}}=n\cdot \pi_{q_1,p_1}(T_n)<\pi_{q_2,p_2}(T_n)=\big\|\tau_n\big\|^{(p_2,q_2)}_{\co\otimes \bidual{E}}\,.
\]
In fact, since $A_n$ is finite, the tensor $\tau_n$ can be identified with an element $\tuple{x}_n\in (\bidual{E})^{m(n)}$ for some $m(n)\in\naturals$. Thus, this shows that the identity operator
\[
	\left((\bidual{E})^{m(n)},\norm^{(p_1,q_1)}_{m(n)}\right)\to \left((\bidual{E})^{m(n)},\norm^{(p_2,q_2)}_{m(n)}\right)
\]
has norm at least $n$. By \cite[Corollary 4.14]{DP}, it follows that  the identity operator
\[
	\left(E^{m(n)},\norm^{(p_1,q_1)}_{m(n)}\right)\to \left(E^{m(n)},\norm^{(p_2,q_2)}_{m(n)}\right)
\]
has norm at least $n$. This is true for every $n\in\naturals$. But this contradicts the assumption that the $(p_1,q_1)$- and the $(p_2,q_2)$-multi-norms 
are equivalent on $E$.\end{proof}\s

\begin{corollary}\label{2.9}
Let $E$ be a Banach space, and suppose that $1\leq p_1\leq q_1<\infty$ and $1\leq p_2\leq q_2<\infty$.  Then the following are equivalent:\s

{\rm (a)}  $(\norm_n^{(p_1,q_1)}: n\in\N)\cong (\norm_n^{(p_2,q_2)}: n\in\N)$ on $E$;\s

{\rm (b)}  $\Pi_{q_1,p_1}(\dual{E},\co) = \Pi_{q_2,p_2}(\dual{E},\co)\,.$
\qed \end{corollary}

\medskip

\section{The $(p,p)$-multi-norm}  We now give another description of the $(p,p)$-multi-norm.\s

\begin{theorem}\label{connection with Chevet--Saphar norm}
 Let $E$ be a normed space. Then the  tensor norm on $\co\otimes E$ induced from the $(p,p)$-multi-norm is  the Chevet--Saphar norm $d_p$  on $\co\otimes E$.
\end{theorem}

\begin{proof} By Theorem~\ref{summing}, the embedding of $\co\otimes E$ into $\Pi_p(E',\co)$ induces the $(p,p)$-multi-norm.  
By Proposition \ref{1.10c}, $\Pi_p(E',\co)$ agrees isometrically with the class of $p\,$-integral maps from $E'$ to $\co$.  By definition,
the $p\,$-integral norm, $i_p(T)$, of a map $T:E'\rightarrow \co$ is the norm of the induced functional on $E'\proten_{g_p'}\ell^{\,1} = \ell^{\,1}\proten_{d_p'}E'$.
Hence the natural map
\[ \big( \co\otimes E, \norm^{(p,p)} \big)\rightarrow \big( \ell^{\,1}\proten_{d_p'}E' \big)' \]
is an isometry.  Since $\co$ has the metric approximation property and $d_p$ is an accessible tensor norm, as explained in the introduction,
it follows that the $\norm^{(p,p)}$-norm on $\co\otimes E$ is just the $d_p$ norm, as claimed.
\end{proof}\s

Thus we have another description of the $(p,p)$-multi-norm based on a normed space $E$. The value  of this result is that it
 gives an excellent description of the dual space to $(\co\otimes E, \norm^{(p,p)})$, namely  as 
\[
	(\co \proten_{d_p} E)' \cong \Pi_{p'}(\co,E')\,,
\] 
the collection of $p'\,$-summing maps from $\co$ to $E'$; see \cite[Proposition 6.11]{Ryan}.
The maps in $\Pi_{p'}(\co,\dual{E})$ are usually rather well understood. 

In the general case where $q\geq p$,   we can give an abstract description of the dual space of 
$	(\co\otimes E, \norm^{(p,q)})$, as \cite[\S 4.1.4]{DP}, but we lack a good concrete description of this  dual space, and this  means that
we are unable to adapt the arguments of this section to the more general case. 

\section{Relations between $(p,q)$-multi-norms}

Let $E$ be a normed space, and consider the above $(p,q)$-multi-norms based on $E$, where $1\le p\le q<\infty$. It is clear that,
 for each fixed $p\ge 1$ and $q_1\ge q_2\ge p$, we have
$(\norm^{(p,q_1)}_n)\le (\norm^{(p,q_2)}_n)$, and, for each fixed $q\ge 1$ and $p_1\le p_2\le q$, we have
$(\norm^{(p_1,q)}_n)\le (\norm^{(p_2,q)}_n)$. Further, it is proved in \cite[Theorem 4.4]{DP} that 
$(\norm^{(p,p)}_n)\ge (\norm^{(q,q)}_n)$
whenever $1\le p\le q<\infty$, and so $(\displaystyle{\norm^{(1,1)}_n})$ is the maximum among these multi-norms; by (\ref{max-mn}), it is
 the maximum multi-norm. 
 
 The following theorem, which follows immediately from Theorem~\ref{summing} and the analogous result  for $(q,p)$-summing
operators that is given in  \cite[Theorem~10.4]{DJT}, for example, gives more information about the relations between $(p,q)$-multinorms.

To picture the theorem, consider the following.   We write $\mathcal T$ for the extended `triangle'  given by
\[
\mathcal T = \{(p,q) \in {\R}^2 : 1\leq p\leq q\}\,,
\]
 and, for $c\in [0,1)$, we consider the curve
\[
{\mathcal C}_c =  \left\{(p,q) \in {\mathcal T} : \frac{1}{p}-  \frac{1}{q}=c\right\}\,;
\]
 we have 
 \[
 {\mathcal T}=\bigcup\{{\mathcal C}_c : c\in [0,1)\}\,.
 \]  
 Then  the multi-norm $(\norm_n^{(p,q)})$ increases as we move down a fixed curve 
${\mathcal C}_c$, and it increases when we move to a lower point on a curve to the right. 

In the diagram, arrows indicate increasing multi-norms in the ordering $\leq$.\label{PICTURE1}
 \medskip\medskip \medskip

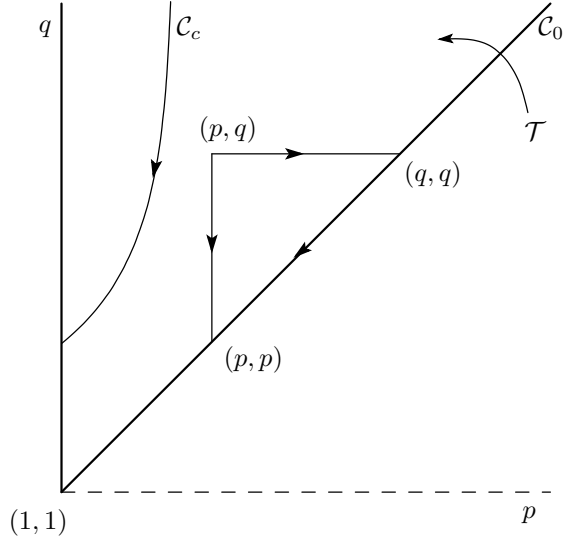
\unitlength 1pt
\begin{picture}(204.8598,188.4995)( 36.9886,-258.9201)
%
\special{pn 13}%
\special{pa 788 3544}%
\special{pa 788 985}%
\special{fp}%
%
\special{pn 13}%
\special{pa 788 3544}%
\special{pa 3347 985}%
\special{fp}%
%
\special{pn 8}%
\special{pa 788 3544}%
\special{pa 3347 3544}%
\special{da 0.070}%
%
\special{pn 8}%
\special{pa 1575 1772}%
\special{pa 2560 1772}%
\special{fp}%
%
\special{pn 8}%
\special{pa 1575 1772}%
\special{pa 1575 2756}%
\special{fp}%
%
\special{pn 13}%
\special{pa 1999 1772}%
\special{pa 2038 1772}%
\special{fp}%
\special{sh 1}%
\special{pa 2038 1772}%
\special{pa 1972 1752}%
\special{pa 1986 1772}%
\special{pa 1972 1792}%
\special{pa 2038 1772}%
\special{fp}%
%
\special{pn 13}%
\special{pa 1575 2215}%
\special{pa 1575 2274}%
\special{fp}%
\special{sh 1}%
\special{pa 1575 2274}%
\special{pa 1595 2208}%
\special{pa 1575 2222}%
\special{pa 1556 2208}%
\special{pa 1575 2274}%
\special{fp}%
%
\special{pn 13}%
\special{pa 2067 2264}%
\special{pa 2018 2304}%
\special{fp}%
\special{sh 1}%
\special{pa 2018 2304}%
\special{pa 2082 2278}%
\special{pa 2060 2271}%
\special{pa 2058 2248}%
\special{pa 2018 2304}%
\special{fp}%
%
\special{pn 8}%
\special{pa 788 2766}%
\special{pa 812 2746}%
\special{pa 836 2725}%
\special{pa 859 2704}%
\special{pa 881 2683}%
\special{pa 904 2661}%
\special{pa 925 2638}%
\special{pa 945 2616}%
\special{pa 966 2592}%
\special{pa 985 2568}%
\special{pa 1003 2545}%
\special{pa 1021 2520}%
\special{pa 1039 2496}%
\special{pa 1056 2470}%
\special{pa 1071 2444}%
\special{pa 1086 2418}%
\special{pa 1101 2392}%
\special{pa 1116 2365}%
\special{pa 1128 2338}%
\special{pa 1142 2311}%
\special{pa 1154 2283}%
\special{pa 1167 2254}%
\special{pa 1178 2226}%
\special{pa 1188 2197}%
\special{pa 1199 2168}%
\special{pa 1209 2138}%
\special{pa 1219 2109}%
\special{pa 1228 2078}%
\special{pa 1237 2049}%
\special{pa 1245 2018}%
\special{pa 1252 1987}%
\special{pa 1259 1956}%
\special{pa 1266 1925}%
\special{pa 1273 1893}%
\special{pa 1279 1862}%
\special{pa 1285 1829}%
\special{pa 1291 1798}%
\special{pa 1296 1765}%
\special{pa 1301 1733}%
\special{pa 1306 1699}%
\special{pa 1310 1667}%
\special{pa 1314 1633}%
\special{pa 1317 1600}%
\special{pa 1321 1566}%
\special{pa 1325 1533}%
\special{pa 1328 1500}%
\special{pa 1331 1466}%
\special{pa 1334 1432}%
\special{pa 1336 1398}%
\special{pa 1339 1364}%
\special{pa 1341 1329}%
\special{pa 1343 1295}%
\special{pa 1345 1260}%
\special{pa 1347 1226}%
\special{pa 1349 1191}%
\special{pa 1351 1157}%
\special{pa 1352 1123}%
\special{pa 1354 1088}%
\special{pa 1356 1053}%
\special{pa 1357 1018}%
\special{pa 1359 984}%
\special{pa 1359 975}%
\special{sp}%
%
\special{pn 13}%
\special{pa 1280 1831}%
\special{pa 1270 1880}%
\special{fp}%
\special{sh 1}%
\special{pa 1270 1880}%
\special{pa 1303 1820}%
\special{pa 1280 1829}%
\special{pa 1263 1813}%
\special{pa 1270 1880}%
\special{fp}%
\put(231.3334,-267.7238){\makebox(0,0)[lb]{$p$}}%
\put(36.9886,-272.4351){\makebox(0,0)[lb]{$(1,1)$}}%
\put(48.4112,-83.9356){\makebox(0,0)[lb]{$q$}}%
\put(130.9959,-113.8110){\makebox(0,0)[rt]{$(p,q)$}}%
\put(208.3957,-140.8204){\makebox(0,0)[rb]{$(q,q)$}}%
\put(140.6656,-210.8183){\makebox(0,0)[rb]{$(p,p)$}}%
\put(105.2752,-80.4922){\makebox(0,0){$\mathcal{C}_c$}}%
\put(242.2752,-80.4922){\makebox(0,0){$\mathcal{C}_0$}}%
%
\special{pn 8}%
\special{pa 3229 1556}%
\special{pa 3222 1525}%
\special{pa 3214 1495}%
\special{pa 3206 1464}%
\special{pa 3197 1435}%
\special{pa 3187 1404}%
\special{pa 3174 1375}%
\special{pa 3159 1345}%
\special{pa 3142 1317}%
\special{pa 3124 1290}%
\special{pa 3103 1265}%
\special{pa 3080 1243}%
\special{pa 3056 1222}%
\special{pa 3030 1205}%
\special{pa 3002 1192}%
\special{pa 2974 1183}%
\special{pa 2944 1176}%
\special{pa 2913 1172}%
\special{pa 2881 1170}%
\special{pa 2849 1169}%
\special{pa 2815 1170}%
\special{pa 2782 1171}%
\special{pa 2766 1172}%
\special{sp}%
%
\special{pn 8}%
\special{pa 2782 1171}%
\special{pa 2766 1172}%
\special{fp}%
\special{sh 1}%
\special{pa 2766 1172}%
\special{pa 2833 1188}%
\special{pa 2818 1169}%
\special{pa 2830 1148}%
\special{pa 2766 1172}%
\special{fp}%
\put(236.1579,-119.5016){\makebox(0,0){$\mathcal{T}$}}%
\end{picture} \medskip\medskip\medskip\medskip

\begin{theorem} \label{comparing (p,q) multi-norms}
Let $E$ be a normed space, and take $(p_1,q_1)$ and $(p_2,q_2)$ in $\mathcal T $.
Then  $(\norm_n^{(p_1,q_1)}) \leq (\norm_n^{(p_2,q_2)})$  whenever both $1/p_2 - 1/q_2 \leq 1/p_1-1/q_1$ and $q_2\leq q_1$.    \enproof
\end{theorem}
\medskip

 It is easy to see that $ \varphi^{(p,q)}_n(E)  = \pi_{q,p}^{(n)}(E')$ for each normed space $E$ and each
 $n\in\N$ \cite[Theorem 4.4]{DP}, and so the following result is immediate from Proposition \ref{1.10}.\s
 
 \begin{proposition}\label{1.11}  
Suppose that  $(p,q)$ is in $\mathcal T $  and that $n\in\N$.  Then:\s
\begin{enumerate}
\item[{\rm (i)}] $\varphi^{(p,q)}_n(E)\leq n^{1/q}$ for each normed space $E$;\s

\item[{\rm (ii)}]   $\varphi^{(p,q)}_n(E)=  n^{1/q}$ for each infinite-dimensional normed space $E$  whenever $p\geq 2$\,;\s

\item[{\rm (iii)}]   $\varphi^{(p,q)}_n(E)\ge  n^{1/2-1/p+1/q}$ for each infinite-dimensional normed space $E$  whenever $p\le 2$\,;\s

\item[{\rm (iv)}] $\varphi^{(p,q)}_n(\ell^{\,r}) =n^{1/q}$ whenever $r\geq 1$ and $p\geq \min\set{r,2}$\,.\enproof
\end{enumerate}
\end{proposition}\s

In Theorem \ref{values of phi sequence for Lr}, we shall improve clause (iv) of the above proposition by giving (asymptotic)
 values of $\varphi^{(p,q)}_n(\lspace^r)$ for all values of $p$ and $q$ with $1\le p\le q<\infty$ in the case where $r>1$.

\begin{corollary}\label{3.2}  
Let $E$ be an infinite-dimensional Banach space, and take $(p_1,q_1)$ and $(p_2,q_2)$ in $\mathcal T $.
Then the $(p_1,q_1)$- and $(p_2,q_2)$-multi-norms based on $E$ are not equivalent  whenever $p_1,p_2\geq 2$ and $q_1\neq q_2$. \enproof 
\end{corollary} \s

Combining the previous proposition with Theorem \ref{comparing (p,q) multi-norms}, we obtain the following.

\begin{corollary}\label{3.2a}
Let $E$ be an infinite-dimensional Banach space. Suppose that  $(p,q)$ is in $\mathcal T $. \s

{\rm (i)} The  $(p,q)$- and the maximum multi-norms based  on $E$ are not equivalent whenever $q>2$. \s

{\rm (ii )} The $(p,q)$- and  the minimum multi-norms based on $E$ are not equivalent whenever $1/p-1/q < 1/2$.  
\end{corollary} 

\begin{proof} (i) Take $p_1\in (2,q)$. Then 
\[
	(\norm^{(p,q)}_n)\le (\norm^{(p_1,p_1)}_n)\le (\norm^{(2,2)}_n)\,.
\]
However, $(\norm^{(p_1,p_1)}_n)\not\cong (\norm^{(2,2)}_n)$ on $E$ by Proposition \ref{1.11}(ii), and so this implies that 
$(\norm^{(p,q)}_n)\not\cong (\norm^{\max}_n)$ on $E$.\s

(ii) This follows from Proposition \ref{1.11}.
\end{proof}\s

We shall compare the $(p,q)$-multi-norms on $\Lspace^r(\Omega)$, and, when $r>1$, we shall compute $\varphi^{(p,q)}_{n}(\ell^{\,r})$ 
asymptotically for all other values of $p$ and $q$ later. 
For these calculations, we shall need to use the following proposition, which is an immediate consequence of Proposition \ref{interpolating pi sequence}.

\begin{proposition}\label{interpolating phi sequence}
Let $E$ be a normed space. Suppose that   $1\leq p\le q_1<q<q_2<\infty$, so that
\[
	\frac{1}{q}=\frac{1-\theta}{q_1}+\frac{\theta}{q_2}
\]
for some $\theta\in (0,1)$. Then 
\[ 
\vspace{-\baselineskip}{\varphi^{(p,q)}_{n}(E)\le \left(\varphi^{(p,q_1)}_{n}(E)\right)^{1-\theta}\,\cdot\,\left(\varphi^{(p,q_2)}_{n}(E)\right)^\theta}\quad
(n\in\N)\,.
\]\enproof
\end{proposition}\s

The following calculations of some specific $(p,q)$-multi-norms will also be useful. 
 
 \begin{example}\label{a calculation of (p,q)-norm on lr space}
Let $r\geq 1$, set $s =r'$, and take  $(p,q)\in \mathcal T $.  Then 
\begin{align*}
\left\Vert (\delta_1,\dots, \delta_n)\right\Vert^{(p,q)}_n
=\Big\Vert \sum_{i=1}^n e_i\otimes \delta_i\Big\Vert_{\textstyle{\co\otimes \ell^{\,r}}}
 = \Big\Vert \sum_{i=1}^n e_i\otimes \delta_i\Big\Vert_{\textstyle{\Pi_{q,p}(\ell^{\,s},\co)}}=\pi_{q,p}(I_n) 
\end{align*}
for each $n\in\N$,  where $I_n$ is the formal identity map from $\ell_n^{\,s}$ to $\ell_n^{\,\infty}$. Here we are now writing $(\delta_i)$ and $(e_i)$ for
 the standard bases in $\ell^{\,r}$ and $\co$, respectively.\s
 
 The value of $\lV (\delta_1,\dots, \delta_n)\rV^{(p,q)}_n$ based on the Banach space $\ell^{\,r}$ is calculated for certain values of $p$ and $q$ in \cite[Example 4.8]{DP}. 
We now calculate this value for all $(p,q)\in \mathcal T $ by elementary means. 
  
 Fix $n\in\naturals$, and, for $(p,q)\in \mathcal T $, write 
\begin{equation}\label{(2.3)}
\Delta_n(p,q)=\lV (\delta_1,\dots, \delta_n)\rV^{(p,q)}_n \,.
\end{equation}
Set $s=r'$ and $u=p'$. Then
\[
	\Delta_n(p,q)=\sup\set{\left(\sum_{i=1}^n\abs{\lambda_{i,i}}^q\right)^{1/q}\colon\ \lambda_1,\ldots,\lambda_n\in\lspace_n^s,\ \mu_{p,n}(\lambda_1,\ldots,\lambda_n)\le 1}\,,
\]
and so, using \eqref{weak--summing-norm 2a}, we see that 
\begin{align}\label{a calculation of (p,q)-norm on lr space: eq1}
	\Delta_n(p,q)=\sup\set{\left(\sum_{i=1}^n\abs{\lambda_{i,i}}^q\right)^{1/q}\colon\ (\lambda_{i,j})_{i,j=1}^n\in\operators(\lspace_n^u,\lspace_n^s)_{[1]}}\,.
\end{align}

We now use  \cite[Proposition 1.c.8]{LT}, which states the following: Suppose that a  matrix $(\lambda_{i,j})_{i,j=1}^n$ defines a contraction from $\ell^{\,u}_n$ to
$\ell^{\,s}_n$. Then the `diagonal' operator obtained by setting all the off-diagonal terms of our matrix to $0$  also defines a contraction between the same spaces. As 
the sum in \eqref{a calculation of (p,q)-norm on lr space: eq1} involves only the terms $\lambda_{i,j}$ with $j=i$, we see that we can make this change without 
changing the value  of $\Delta_n(p,q)$, and thus we can say that
\[
	\Delta_n(p,q)=\sup\set{\|\alpha\|_q\colon\ D_\alpha\in\operators(\lspace_n^u,\lspace_n^s)_{[1]}}\,,
\]
where $D_\alpha x=(\alpha_1x_1,\ldots,\alpha_n x_n)$ for each $\alpha,x\in \C^n$. 

We \emph{claim} that
\begin{align*}
\Delta_n(p,q)=\begin{cases}
n^{1/q}&\textrm{when}\ u\le s\,,\\
n^{1/q+1/u-1/s}&\textrm{when}\ u>s\ \textrm{and}\  1/q+1/u\ge 1/s\,,\\
1&\textrm{when}\  1/q+1/u< 1/s\,.
\end{cases}
\end{align*}

Indeed, suppose first that $u>s$. Then there exists $t>1$ such that $1/s=1/u+1/t$. Let  $\alpha=(\alpha_1,\ldots,\alpha_n)\in \C^n$. 
A version of H\"{o}lder's inequality implies that
\[
	\left(\sum_{i=1}^n\abs{\alpha_ix_i}^s\right)^{1/s}\le 	\left(\sum_{i=1}^n\abs{\alpha_i}^t\right)^{1/t}	\left(\sum_{i=1}^n\abs{x_i}^u\right)^{1/u}
\]
for every $(x_i)\in\lspace_n^u$. Moreover, equality is attained for a suitable choice of $(x_i)$, and so we see that
\[
	\|D_\alpha:\lspace^u_n\to\lspace^s_n\|=\|\alpha\|_t \qquad(\alpha\in \C^n)\,.
\]
Thus the problem now is to compute
 \[
	\Delta_n(p,q)=\sup\set{\|\alpha\|_q\colon\ \alpha\in(\lspace_n^t)_{[1]}}\,.
\]
If $t>q$, the supremum occurs when $\alpha=(\alpha_i)$ is the constant sequence $(n^{-1/t})$, in which case we obtain
\[
	\left(\sum_{i=1}^n\abs{\alpha_i}^q\right)^{1/q}=(n\,\cdot\,n^{-q/t})^{1/q}=n^{1/q-1/t}=n^{1/q+1/u-1/s}\,.
\]
If $t\le q$, the supremum occurs at a point mass, in which case we obtain $\|\alpha\|_q=1$.

Finally, suppose that $u\le s$. Then we see that
\[
	\left(\sum_{i=1}^n\abs{\alpha_ix_i}^s\right)^{1/s}\le 	\left(\sum_{i=1}^n\abs{\alpha_ix_i}^u\right)^{1/u}\le	\|\alpha\|_\infty\,\left(\sum_{i=1}^n\abs{x_i}^u\right)^{1/u},
\]
and equality occurs when $(x_1,\ldots,x_n)$ is a point mass. Thus 
\[
	\|D_\alpha:\lspace^u_n\to\lspace^s_n\|=\|\alpha\|_\infty \qquad(\alpha\in \C^n)\,.
\]
It follows that
 \[
	\Delta_n(p,q)=\sup\set{\|\alpha\|_q\colon\ \alpha\in(\lspace_n^\infty)_{[1]}}=n^{1/q}\,.
\]

This establishes the claim.

We conclude as follows. Suppose that  $r\ge 1$ and  $(p,q)\in \mathcal T $. Then the $(p,q)$-multi-norm based on $\lspace^{r}$ satisfies the following equation for each $n\in\N$:
\s

\begin{equation}\label{a calculation of (p,q)-norm on lr space: eq2}
 \lV(\delta_1,\dots, \delta_n)\rV^{(p,q)}_n =\left\{\begin{array}{cl}n^{1/r+1/q-1/p}  &\mbox{\rm when $p<r$ and ${1}/{p}-{1}/{q}\leq{1}/{r}$,}\\
 \\1&\mbox{\rm when ${1}/{p}-{1}/{q}>{1}/{r}$,}\\
\\
n^{1/q}&\mbox{\rm when $p\geq r$}\,.\end{array}\right.
 \end{equation}
We can also write the above formula more concisely as follows:
\[
	\lV (\delta_1,\dots, \delta_n)\rV^{(p,q)}_n =n^\alpha\quad(n\in\N)\,,
\]
where 
\[
	\alpha=\left(\frac{1}{q}-\left(\frac{1}{p}-\frac{1}{r}\right)^+\right)^+\,.
\] 
Here, $x^+=\max\set{x,0}$ for each $x\in\R$. \enproof
\end{example}
 \medskip

\begin{example}\label{another calculation of (p,q)-norm on lr space}
Suppose that   $r\geq 1$ and $(p,q)\in \mathcal T $.  Set $s =r'$ and $u=p'$, as before.  

Fix $n\in\N$. For $i\in \N_n$, take 
\[
f_i= \frac{1}{n^{1/r}}\sum_{j=1}^n\zeta^{-ij}\delta_j= \frac{1}{n^{1/r}}(\zeta^{-i},
\zeta^{-2i},\dots,\zeta^{-ni},0,0, \dots)\in  \ell^{\,r}\,,
\]
where $\zeta =\exp(2\pi{\rm i}/n)$, so that $\left\Vert f_i\right\Vert_{\ell^{\,r}}
= 1\,\;(i\in \N_n)$, and then set $\tuple{f}=(f_1,\dots,f_n)$. 
Next take $\tuple{\lambda}=(\lambda_1,\dots,\lambda_n)$, where 
 \[
 \lambda_i= \sum_{j=1}^n\zeta^{ij}\delta_j=
(\zeta^{i},\zeta^{2i},\dots,\zeta^{ni},0,0, \dots)\in \ell^{\,s}\,.
\]

Note that 
\begin{equation}\label{(2.4)}
\left(\sum_{i=1}^n\left\vert \langle
f_i,\,\lambda_i\rangle\right\vert^q\right)^{1/q}  =  n^{1+1/q-1/r}\,.
\end{equation}

We take $\zeta_1,\dots,\zeta_n\in\C$ with $\sum_{i=1}^n\left\vert
\zeta_i\right\vert^u \leq 1$, and set 
$z_i  = \sum_{j=1}^n \zeta_j\zeta^{ij}\;\, (i\in\N_n)$, 
  so that 
\[
\sum_{i=1}^n \left\vert z_i\right\vert^2 = n\sum_{i=1}^n \left\vert
\zeta_i\right\vert^2\quad{\rm and} \quad
\left\Vert \sum_{i=1}^n\zeta_i\lambda_i\right\Vert_{\ell^{\,s}} =
\left(\sum_{i=1}^n\left\vert z_i\right\vert^s\right)^{1/s}\,.
\]

Now suppose that $r\geq 2$, so that $1\le s\leq 2$. 

In the case where $p\geq 2$, so that $u \leq 2$, we have $\sum_{i=1}^n\left\vert
\zeta_i\right\vert^2  \leq \sum_{i=1}^n\left\vert \zeta_i\right\vert^u \leq 1$,
and so 
\[
\mu_{p,n}(\tuple{\lambda}) \leq  \frac{n^{1/s}}{n^{1/2}}\left(\sum_{i=1}^n \left\vert
z_i\right\vert^2\right) \leq n^{1/s}\,.
\]
Hence, by (\ref{(2.4)}), 
\[
\left\Vert \tuple{f} \right\Vert^{(p,q)}_n \geq \frac{n^{1+1/q}}{n^{1/r+1/s}}  = n^{1/q}\,.
\]

In the case where $1\leq p \leq 2$, so that $u \geq 2$, we have
\[
\left(\sum_{i=1}^n\left\vert \zeta_i\right\vert^2\right)^{1/2}\leq 
\frac{n^{1/2}}{n^{1/u}}\left(\sum_{i=1}^n\left\vert
\zeta_i\right\vert^u\right)^{1/u} \leq n^{1/2 - 1/u}\,,
\] 
 and so
 \[
 \mu_{p,n}(\tuple{\lambda}) \leq  \frac{n^{1/s}}{n^{1/2}}\left(\sum_{i=1}^n \left\vert
z_i\right\vert^2\right)^{1/2} \leq n^{1/2 + 1/s -1/u}\,.
 \]
 Hence, again by (\ref{(2.4)}),  
\[
\left\Vert \tuple{f} \right\Vert^{(p,q)}_n \geq \frac{n^{1 + 1/q + 1/u}}{n^{1/2+1/s +1/r}} =
n^{1/2- 1/p+1/q}\,.
\] 
 
 It is always true that $\left\Vert \tuple{f} \right\Vert^{(p,q)}_n \geq 1$.

We conclude that, in the case where $r\ge 2$, we have the following estimates, which hold for each $n\in\N$: 
\begin{equation}\label{(2.5)}
 \left\Vert \tuple{f}\right\Vert^{(p,q)}_n \geq\left\{\begin{array}{cl}n^{1/2-1/p+ 1/q } 
&\mbox{\rm when $1\leq p\leq 2$ and $1/p-1/q\leq 1/2$\,,}\\
 \\1&\mbox{\rm when $1/p-1/q> 1/2$\,,}\\
\\
n^{1/q}&\mbox{\rm when $p\geq 2$}\,.\end{array}\right.
\end{equation}

We shall see from Theorem \ref{values of phi sequence for Lr}, given below, that $ \left\Vert \tuple{f}\right\Vert^{(p,q)}_n $ is always equal
to the term on the right-hand side of (\ref{(2.5)}) to within a constant independent of $n$. \qed
\end{example}\medskip

\section{The standard $t$-multi-norm on $\Lspace^r$-spaces}

\noindent Let $(\Omega,\mu)$ be a measure space, take $r\ge 1$, and suppose that $r\leq t<\infty$. In \cite[\S 4.2]{DP} and \cite[\S 6]{DDPR1}, there is a 
definition and discussion of the   standard $t\,$-multi-norm  on the Banach space $\Lspace^r(\Omega)$. We recall the definition. 

Take $n\in\naturals$.  For each ordered partition ${\bf X}=(X_1,\dots,X_n)$ of $\Omega$ into measurable subsets and each 
$f_1,\ldots,f_n\in \Lspace^r(\Omega)$, we define
\[ 
	r_{\tuple{X}}((f_1,\ldots,f_n)) = \Big( \sum_{i=1}^n \|P_{X_i}f_i\|^t \Big)^{1/t}. 
\]
Here $P_{X_i}:f\mapsto f\mid  X_i $ is the projection of $\Lspace^r(\Omega)$ onto $\Lspace^r(X_i)$, and $\norm$ is the $\Lspace^r$-norm.  Then we define
\[
 	\left\Vert{(f_1,\ldots,f_n)}\right\Vert^{[t]}_n = \sup_{\tuple{X}} r_{\tuple{X}}((f_1,\dots,f_n))\,, 
\]
where the supremum is taken over all such measurable ordered  partitions ${\bf X}$.

As in \cite[\S 4.2.1]{DP}, we see that $(\norm^{[t]}_n :\ n\in\naturals)$ is a multi-norm based on 
$\Lspace^r(\Omega)$; it is the \emph{standard $t\,$-multi-norm} on $\Lspace^r(\Omega)$. 

Clearly the norms $\norm^{[t]}_n$ decrease as a function of $t\in [r,\infty)$, and so the maximum among these norms is $\norm^{[r]}_n$. 

For example, by \cite[(4.9)]{DP}, we have 
\[
 	\left\Vert{(f_1,\ldots,f_n)}\right\Vert^{[t]}_n = \left(\lV f_1\rV^t + \cdots + \lV f_n\rV^t\right)^{1/t}\quad (n\in\N)
\]
whenever $f_1,\dots,f_n$ in $\Lspace^r(\Omega)$ have pairwise disjoint supports, and, in particular,
\[
\left\Vert{(\delta_1,\dots,\delta_n)}\right\Vert^{[t]}_n = n^{1/t}\quad (n\in\N)\,.
\]

As remarked in \cite{DP}, it appears that the definition of the standard $t\,$-multi-norm depends on the concrete representation of the space $\Lspace^r(\Omega)$.
 However, in \cite[\S 4.2.8]{DP}, there is  a definition  of an `abstract $t\,$-multi-norm based on a $\sigma$-Dedekind complete  Banach lattice $E$', and it 
 is shown in \cite[Theorem 4.36]{DP} that the standard $t\,$-multi-norm based on a Banach lattice $\Lspace^r(\Omega)$ depends on only the norm and the lattice 
structure of the space.  For a related result, see \cite[Theorem 4.40]{DP}.

The rate of growth of the standard $t$-multi-norm based on $\Lspace^r(\Omega)$ is denoted by $\varphi^{[t]}_n(\Lspace^r(\Omega))$, as in \cite[Definition 4.21]{DP}. In fact, it is easily seen that 
\begin{align}\label{varphi sequence for standard q-multi-norm}
	\varphi^{[t]}_n(\Lspace^r(\Omega))=n^{1/t}\,
\end{align}
for every infinite-dimensional $\Lspace^r(\Omega)$-space.

In the special case where $t=r$, we have 
\begin{equation}\label{(2.3a)}
\lV (f_1,\dots,f_n)\rV_n^{[r]} = 
 \left(\int_\Omega (\left\vert f_1 \right\vert \vee \dots \vee \left\vert f_n \right\vert)^{\,r}\right)^{1/r} 
\end{equation}
for $n\in\naturals$ and elements $f_1,\dots,f_n \in   \Lspace^r(\Omega)$; this is equation (4.12) in \cite{DP}. 

For a Banach space $E$ and $r\geq 1$, the space $\Lspace^r(\Omega, E)$ consists of  (equiv\-alence classes of)
$\mu$-measurable functions $F: \Omega \to E$ such that  the function $s\mapsto \lV F(s)\rV $ on $\Omega$ belongs to the space
  $\Lspace^p(\Omega)$;   with respect to the norm $\norm$ specified by 
\[
\lV F \rV =\left(\int_\Omega \lV F(s)\rV^r\,{\rm d}\mu(s)\right)^{1/r}\quad (F\in \Lspace^r(\Omega, E))\,,
\]
the  space $\Lspace^r(\Omega, E)$  is a Banach space.  The tensor product  $\Lspace^r(\Omega)\otimes E$ can be identified with a dense subspace of $\Lspace^r(\Omega, E)$.  
Indeed,  $f\otimes x\in \Lspace^r(\Omega)\otimes E$ corresponds to the function $s\mapsto f(s)x$.  See \cite[Chapter 7]{DF} and \cite[$\S2.3$]{Ryan}.

Now take $n\in\naturals$ and $f_1,\dots,f_n \in   \Lspace^r(\Omega)$. Then $(f_1,\dots,f_n) $ corresponds to the element 
$\sum_{i=1}^n\delta_i\otimes f_i$ in $\co \otimes \Lspace^r(\Omega)$, and hence to the function \[
s\mapsto \sum_{i=1}^nf_i(s)\delta_i \in  \Lspace^r(\Omega, \co)\,,
\]
 and its norm in $\Lspace^r(\Omega, \co)$ is exactly $\left\Vert (f_1,\dots,f_n)\right\Vert_n^{[r]}$  by equation (\ref{(2.3a)}). 
 
 Thus, in the case where $t=r$, we can regard the standard $r\,$-multi-norm on $\Lspace^r(\Omega)$ simply as that given by the embedding of 
$\co \otimes \Lspace^r(\Omega)$ in $\Lspace^r(\Omega, \co)$.

There seems to be no similarly useful representation of the standard $t\,$-multi-norm on $\Lspace^r(\Omega)$ in the case where $t >r$.
\medskip

\section{The Hilbert multi-norm}

\noindent We now recall an alternative description of the $(2,2)$-multi-norm based on a Hilbert space. This involves the Hilbert multi-norm that was introduced in \cite[\S 4.1.5]{DP}.

Let $H$ be a Hilbert space, with inner-product denoted by $\inner{\,\cdot\,,\,\cdot\,}$. For $n\in\naturals$, we write 
\[
	H=H_1\oplus_{\perp} \cdots \oplus_\perp H_n
\] 
when $H_1,\ldots, H_n$ are pairwise-orthogonal (closed) subspaces of $H$.

Take $n\in\naturals$. For each family $\tuple{H}=\set{H_1,\ldots,H_n}$ such that $H=H_1\oplus_{\perp} \cdots \oplus_\perp H_n$, we set
\[
	r_{\tuple{H}}((x_1,\ldots, x_n))=\left(\left\Vert{P_1x_1}\right\Vert^2+\cdots+\left\Vert{P_nx_n}\right\Vert^2\right)^{1/2}=\left\Vert{P_1x_1+\cdots+P_nx_n}\right\Vert\,
\]
for $x_1,\ldots, x_n\in H$, where $P_i:H\to H_i$ is the orthogonal projection for $i\in\naturals_n$. Then we set
\[
	\left\Vert{\tuple{x}}\right\Vert^H_n=\sup_{\tuple{H}} r_{\tuple{H}}(\tuple{x})\quad(\tuple{x}\in H^n)\,,
\]
where the supremum is taken over all orthogonal decompositions $\tuple{H}$ of $H$. 
As in \cite[Theorem 4.15]{DP}, $(\norm^H_n:\ n\in\naturals)$ is a multi-norm based on $H$; it is called the \emph{Hilbert multi-norm}.\smallskip

The following result is \cite[Theorem 4.19]{DP}.

\begin{theorem}\label{thm:hilb22}
Let $H$ be a Hilbert space. Then $(\norm^H_n)=(\norm^{(2,2)}_n)$. \qed
\end{theorem}

\smallskip

\section{Relations between multi-norms}

\noindent In this subsection, we shall first summarize some results about the relationships between multi-norms that were already established 
in \cite{DP}. \s 

 \begin{theorem}\label{2.10} Let $E$ be a normed space. Then $ (\norm_n^{(1,1)}) = (\norm_n^{\max})$. 
 \end{theorem}

\begin{proof}
This is \cite[Theorem 4.6]{DP}.
\end{proof}\s

 \begin{theorem}\label{2.13} Take $r,t$ with $1 \leq r\le t<\infty$, and let $\Omega$ be a measure space. Then
\[
	(\norm^{[t]}_n)\le(\norm^{(r,t)}_n)
	\quad {\rm on}\quad  \Lspace^r(\Omega)\,. 
\]
Moreover, when $r=1$, these two multi-norms are equal on $\Lone(\Omega)$ whenever $t\in [1,\infty)$.  Further, $(\norm_n^{[1]})=(\norm^{\max}_n)$ on $\Lone(\Omega)$.
\end{theorem}

\begin{proof}
This combines  \cite[Theorems 4.22, 4.23, and 4.26]{DP}. 
\end{proof}\s

By \eqref{varphi sequence for standard q-multi-norm}, different standard $t$-multi-norms on an infinite-dimensional $\Lspace^r(\Omega)$ 
space are not equivalent to each other, and they are never equivalent to the minimum multi-norm; we shall see in Theorem
 \ref{equivalence of (p,q) and [q] multi-norms} that they are never equivalent to the maximum multi-norm.\s

\begin{theorem}\label{2.12} 
Take $r\ge 1$, and suppose that  $r\leq t< \infty$.  Suppose that either $2\leq r\leq t$ or that $1< r <2$ and $r\leq t < r/(2-r)$. Then  the multi-norms 
$(\norm_n^{[t]}  : n\in \naturals)$ and $(\norm_n^{(r,t)}  : n\in \naturals)$ based on $\ell^{\,r}$ are not equivalent.
\end{theorem}

\begin{proof} This is \cite[Theorem 4.27]{DP}.
\end{proof}\s

We shall extend and complement the above results in the present memoir.
 
\chapter{Comparing $(p,q)$-multi-norms on $\Lspace^r$ spaces}

\noindent In this section, we aspire to determine when two $(p,q)$-multi-norms based on a space $\Lspace^r(\Omega)$ are equivalent; 
we shall obtain a reasonably complete classification, but cannot give a fully comprehensive account.\s
 
\section{The case where $r=1$}
 
\noindent In this section, we investigate the equivalence of various $(p,q)$-multi-norms on spaces of the form $\Lspace^1(\Omega)$.

By Example \ref{a calculation of (p,q)-norm on lr space}, $(\norm^{(p_1,q_1)}_n)$  is not equivalent to  $(\norm^{(p_2,q_2)}_n)$ on 
$ \Lone(\Omega)$ whenever $\Lone(\Omega)$ is infinite dimensional and $q_1\neq q_2$ because $\Delta_n(p,q)=n^{1/q}\,\;(n\in\N)$ for each $(p,q)\in {\mathcal T}$, 
in the notation of that example; it remains to investigate the case where $q_1=q_2$.  

The following result is \cite[Theorem 5.6]{DDPR1}.  It is also a consequence of
Theorem~\ref{summing}  and the corresponding result 
in \cite[Theorem 10.9]{DJT}.\smallskip

\begin{theorem}
Let $\Omega$ be a measure space, and take $p,q,s\in \R$ with $1\leq p<q<s<\infty$. Then
\[	\vspace{-\baselineskip}
(\norm^{(p,q)}_n)\cong(\norm^{(1,q)}_n)\succcurlyeq (\norm^{(s,s)}_n)\quad{\rm on}\quad \Lone(\Omega)\,. 
\]
\enproof
\end{theorem}\smallskip

The following result shows that the condition `$p<q$' in the above theorem is sharp.  Note also that
\[
 \left\Vert (\delta_1,\dots, \delta_n)\right\Vert^{(q,q)}_n = \left\Vert (\delta_1,\dots, \delta_n)\right\Vert^{(1,q)}_n \;\;(= n^{1/q}) \quad(n\in\naturals)\,,
\] 
for $q\geq 1$, and so the above equation is not sufficient to enforce the non-equivalence of $(\norm^{(q,q)}_n)$ and $(\norm^{(1,q)}_n)$.\s
 
\begin{theorem}\label{thm:3.2}
Let $\Omega$ be a measure space such that  $\Lspace^1(\Omega)$ is infinite dimen\-sional. 
Take $q >1$. Then $(\norm^{(q,q)}_n)\ge (\norm^{(1,q)}_n)$, but $(\norm^{(q,q)}_n)\not\cong (\norm^{(1,q)}_n)$ on $\Lone(\Omega)$.
\end{theorem}

\begin{proof} First, suppose that our multi-norms are based on $\ell^{\,1}$.

Take $n\in\naturals$, and let $I_n$ be the identity map from $\ell^{\,\infty}_n$ to the Lorentz space $\ell^{\,q,1}_n$. A calculation  of Montgomery-Smith \cite{MS}
(see \cite{CD} for a statement of this example) shows that
\[
 \pi_{q,q}(I_n)  \sim  n^{1/q}(1 + \log n)^{ 1-1/q}\,,\quad  \pi_{q,1}(I_n)  \sim  n^{1/q}\,.
 \]

For each $n\in\naturals$, we can find $m=m(n)\in\mathbb N$, with $m(n)\geq n$,
and an operator $\varphi_n:\ell^{\,q,1}_n  \rightarrow \ell^{\,\infty}_{m(n)}$ with 
\[
	\left(1-\frac{1}{n}\right) \|x\|_{q,1} \leq \|\varphi_n(x)\|_\infty \leq \|x\|_{q,1}\quad(x\in \ell^{\,q,1}_n)\,.
\] 
Let $p_{n}:
\ell^{\,\infty}\rightarrow\ell^{\,\infty}_n$ be the natural projection, and define
\[
T_n = \frac{1}{n^{1/q}} \varphi_n \circ I_n\circ p_{n}:\ell^{\,\infty} \rightarrow \ell^{\,\infty}_{m(n)}\subset\co\,.
\]
 From the definition of the $(q,p)$-summing norm, it follows that 
\[
\left(1-\frac{1}{n}\right)\frac{1}{n^{1/q}} \pi_{q,p}(I_n) \leq \pi_{q,p}(T_n) \leq \frac{1}{n^{1/q}} \pi_{q,p}(I_n)
\] 
whenever $1\le p\le q<\infty$.  In particular,  $\pi_{q,1}(T_n)\sim 1$, but
$\pi_{q,q}(T_n) \sim (1+\log n)^{1-1/q}$.

Since $T_n=T_n\circ p_n$, we see that $T_n$ is the image of 
\[
	\tuple{x}_n:=\sum_{i=1}^nT_n(e_i)\otimes \delta_i
\]
via the natural inclusion $\co\otimes \lone\hookrightarrow \operators(\linfty,\co)$. The previous paragraph and Theorem~\ref{summing} imply that
\[
	\big\|\tuple{x}_n\big\|^{(q,1)}_{\co\otimes \lone}\sim 1\,,\quad\textrm{but}\quad \big\|\tuple{x}_n\big\|^{(q,q)}_{\co\otimes \lone}\sim (1+\log n)^{1-1/q}.
\]
Hence $(\norm^{(q,q)}_n)\not\cong (\norm^{(1,q)}_n)$ on $\lone$.

For a general measure space $\Omega$, the result follows from Theorem~\ref{1.1}.
\end{proof}

We summarize the situation for $(p,q)$-multi-norms based on $\Lspace^1(\Omega)$. In this special case, we have a full solution to the question of equivalences.

\begin{theorem}\label{equivalency for L1}  Let $\Omega$ be a measure space such that $\Lspace^1(\Omega)$ is infinite dim\-ensional, and suppose that
 $(p_1,q_1), (p_2,q_2)\in {\mathcal T}$\s
\begin{enumerate}
\item[{\rm (i)}]  Suppose that $q_2> q_1$. Then $(\norm_n^{(p_1,q_1)})\succcurlyeq  (\norm_n^{(p_2,q_2)})$, and these multi-norms are not
 equivalent on $\Lspace^1(\Omega)$.\s

\item[{\rm (ii)}]  Suppose that $q_2= q_1=q$ and  $p_2> p_1$.  Then $(\norm_n^{(p_2,q)})\geq  (\norm_n^{(p_1,q)})$; these multi-norms are equivalent 
on $\Lspace^1(\Omega)$ when also  $p_2<q$, but they are not equivalent to $(\norm_n^{(q,q)})$.
\enproof
\end{enumerate}
\end{theorem} \s

\begin{corollary}  Let $\Omega$ be a measure space   such that $\Lspace^1(\Omega)$ is infinite dimensional, and suppose that $(p,q)\in {\mathcal T}$. 
  Then the $(p,q)$-multi-norm on $\Lspace^1(\Omega)$ is not equivalent to the minimum multi-norm, and it is equivalent
 to the maximum multi-norm if and only if $p=q=1$, in which case, it is actually equal to the maximum multi-norm. \enproof
\end{corollary}
\medskip

\section{The case where $r>1$} In this case, it is more difficult to determine  when the $(p,q)$-multi-norms are equivalent on $\Lspace^r(\Omega)$. 

Throughout we suppose that $\Lspace^r(\Omega)$ is infinite dimensional.

In this section, it is convenient to continue to use the earlier notation  ${\mathcal C}_c$  for the curve
\[
{\mathcal C}_c = \left\{(p,q) \in {\mathcal T} : \frac{1}{p}-  \frac{1}{q}=c\right\}\,,
\]
whenever  $c\in [0,1)$. This curve is contained  in the triangle  ${\mathcal T}$.  

We shall consider points $P_1$ and $P_2$ in $\mathcal T$,  
and shall say `{\it $P_1$ and $P_2$ are  equivalent} (respectively, {\it not equivalent}\,) {\it on} $E$' to mean that the multi-norms 
 $(\norm^{(p_1,q_1)}_n)$ and $(\norm^{(p_2,q_2)}_n)$ 
based on a Banach space $E$ are  equivalent (respectively, not equivalent).\s

The first result, which shows that various pairs of multi-norms are not equivalent, follows directly from Proposition \ref{1.11} and the calculation given in
 Example \ref{a calculation of (p,q)-norm on lr space}. Indeed, (i) follows from Proposition \ref{1.11}(iv) and (ii)--(iv) follow from 
equation \eqref{a calculation of (p,q)-norm on lr space: eq2}.

\begin{proposition}\label{non-equivalent mns on Lr}
Let $\Omega$ be a measure space, and take $r\ge 1$. Then two points $P_1\in\mathcal C_{c_1}$ and $P_2\in\mathcal C_{c_2}$ are not
 equivalent on $\Lspace^r(\Omega)$ in the following cases:
\begin{enumerate}
\item[{\rm (i)}] $p_1,p_2\ge \min\set{r,2}$, and $q_1\neq q_2$;\s
\item[{\rm (ii)}] $p_1,p_2\le r$, $\min\set{c_1,c_2}<{1}/{r}$, and $c_1\neq c_2$;\s
\item[{\rm (iii)}] $p_1\le r\le p_2$, and 
\[
	\frac{1}{r}-\frac{1}{p_1}+\frac{1}{q_1}\neq \frac{1}{q_2}\,;
\]
\item[{\rm (iv)}] $p_1\ge r\ge p_2$, and 

\qquad\qquad\qquad\quad\ \,\;$\displaystyle{\frac{1}{r}-\frac{1}{p_2}+\frac{1}{q_2}\neq \frac{1}{q_1}}\,.$ \enproof
\end{enumerate}\s
\end{proposition}

\medskip

We now concentrate on the $(p,p)$-multi-norms and the maximum multi-norm  
on $\Lspace^r(\Omega)$.

Let $E$ be a normed space. We recall that it follows from Theorem \ref{connection with Chevet--Saphar norm} that the dual space of 
$(\co\otimes E, \norm^{(p,p)})$ is $ \Pi_{p'}(\co,E')$; the dual of the maximum multi-norm, identified  with $(\co\projectivetensor E, \norm_\pi)$, is 
${\mathcal B}(\co, E')$.\s

\begin{proposition}\label{equivalence of (p,p)-multi-norms on Lr}
Let $\Omega$ be a measure space. Suppose that 
\[
 {\rm either}\quad  1\le p\le 2\le r<\infty \quad {\rm or}\quad  1\le p< r<2 \,. 
\] 
Then $(\norm_n^{(p,p)})$ is  equivalent to $(\norm_n^{\max})$ on $\Lspace^r(\Omega)$. 
\end{proposition}

\begin{proof}  
In the case where $1\leq p\leq 2\leq r< \infty$, so that $r' \in (1,2]$, we use \cite[Theorem~3.7]{DJT}, which tells us that every operator 
$T:\co\rightarrow L^{r'}(\Omega)$ is $2$-summing, with $\pi_2(T) \leq K_G \|T\|$, where $K_G$ is the Grothendieck constant.  Since $\Pi_2(\co,E)
\subset \Pi_{p'}(\co,E)$ is a norm-decreasing inclusion (for any Banach space $E$), we conclude that
\begin{eqnarray*} (\co\otimes \Lspace^r(\Omega),\norm^{(p,p)})'
= \Pi_{p'}(\co,L^{r'}(\Omega))= {\mathcal B}(\co,L^{r'}(\Omega))
=(\co\otimes \Lspace^r(\Omega),\norm^{\max})', 
\end{eqnarray*}
which gives the result.

Similarly, in the case where $1\leq p<r<2$, so that $r' >2$, we appeal to \cite[Corollary~10.10]{DJT}, which shows in particular that we have 
$\Pi_{p'}(\co,L^{r'}(\Omega)) = \mathcal B(\co,L^{r'}(\Omega))$. The result follows.
\end{proof}\s

\begin{proposition}\label{equivalence of (r,r)-multi-norms on Lr}
Let $\Omega$ be a measure space  such that $\Lspace^r(\Omega)$ is infinite dimen\-sional. 
Suppose that $1\leq r<2$.  Then $(\norm_n^{(r,r)})\not\cong(\norm_n^{\max})$ on $\Lspace^r(\Omega)$.
\end{proposition}

\begin{proof} Here  we appeal to \cite[Theorem~7, clause~2]{Kwapien2}, which, 
using an example of Schwartz \cite{Schwartz}, shows that $\Pi_{s}(\co,\ell^{\,s}) \neq {\mathcal B}(\co,\ell^{\,s})$ for $s>2$. The required conclusion follows.
\end{proof}\s

Thus we have a complete classification of the $(p,p)$-multi-norms on $\Lspace^r(\Omega)$ into equivalence classes, summarized in the following theorem.\s

\begin{theorem}\label{3.6}
Let $\Omega$ be a measure space such that
$\Lspace^r(\Omega)$ is infinite dimen\-sional, where $r\ge 1$.  Set $\overline{r} = \min\{2,r\}$. Then:\s
\begin{enumerate}
\item[{\rm (i)}]  $(\norm_n^{(q,q)}) \not\cong (\norm_n^{(p,p)})$ on $\Lspace^r(\Omega)$ whenever $p,q\geq \overline{r}$ and $p\neq q$;\s 

\item[{\rm (ii)}] $(\norm_n^{(p,p)})\not\cong (\norm_n^{\max})$ on $\Lspace^r(\Omega)$ whenever  $p>\overline{r}$;\s

\item[{\rm (iii)}] $(\norm_n^{(p,p)}) \cong (\norm_n^{\max})$ on $\Lspace^r(\Omega)$ whenever $1\leq p<\overline{r}$; \s 

\item[{\rm (iv)}]  $(\norm_n^{(1,1)}) = (\norm_n^{\max})$  on $\Lspace^r(\Omega)$;\s

\item[{\rm (v)}] if $1<r<2$, then $\overline{r} = r$  and $(\norm_n^{(r,r)}) \not\cong (\norm_n^{\max})$  on $\Lspace^r(\Omega)$;\s

\item[{\rm (vi)}]  if $r\ge 2$, then $\overline{r} = 2$  and $(\norm_n^{(2,2)}) \cong (\norm_n^{\max})$  on $\Lspace^r(\Omega)$. \s\end{enumerate}
\end{theorem}
\begin{proof}
Notice that (ii) follows by applying (i) with $q=\bar{r}$ and (iv) is just a special case of Theorem \ref{2.10}.
\end{proof}

\medskip

\section{The role of Orlicz's theorem}
We shall now determine when the $(p,q)$-multi-norm based on $\Lspace^{\,r}(\Omega)$ is equivalent to the minimum multi-norm. For this, we shall need a form of 
Orlicz's theorem. Indeed, a generalization of Orlicz's theorem given in \cite[Theorem 10.7]{DJT} shows that,  
for each $s\in [1,\infty)$, the identity operator on $\Lspace^s(\Omega)$ is $(\tilde{s},1)$-summing, where $\tilde{s}:=\max\set{s,2}$. 
In the case where $s=2$, so that $\tilde{s}=2$ also, the $(2,1)$-summing norm of the identity operator on $\Lspace^2(\Omega)$ is equal to $1$.

Now suppose that $r>1$, and again set $\bar{r} =\min\set{2,r}$. Set $s =r'$, the conjugate of $r$,  so that 
\[
	\tilde{s}=\max\set{s,2}=\bar{r}\,'\,.
\]
Then, since the identity operator on $\Lspace^{s}(\Omega)$ belongs to $\Pi_{\tilde{s},1}(\Lspace^{s}(\Omega))$, we obtain
\[
	 	\operators(\Lspace^{s}(\Omega),F)=\Pi_{\tilde{s},1}(\Lspace^{s}(\Omega),F)
\] 
for each Banach space $F$; in the case where $r=2$, we have equality of the norms as well.

It follows from Theorem \ref{summing} that the tensor norm on $\co\otimes\Lspace^{r}(\Omega)$ induced from the $(1,\bar{r}\,')$-multi-norm 
is equivalent to the injective tensor norm, which is induced by $\operators(\Lspace^{s}(\Omega),\co)$. That is, the $(1,\bar{r}\,')$- and the 
minimum multi-norms on $\Lspace^{r}(\Omega)$ are equivalent. This and Theorem \ref{comparing (p,q) multi-norms} imply the following.

\begin{theorem}\label{equivalent of (p,q)- and min mns on Lr}
Let $\Omega$ be a measure space, take $r>1$, and set $\bar{r}:=\min\set{r,2}$. 
Suppose that $1\le p\le q<\infty$. 
Then $(\norm^{(p,q)}_n)\cong (\norm^{\min}_n)$ on $\Lspace^{r}(\Omega)$ whenever $1/p-1/q\ge 1/\bar{r}$. 
Moreover $(\norm^{(p,q)}_n)= (\norm^{\min}_n)$ on $\Lspace^{2}(\Omega)$ whenever $1/p-1/q\ge 1/2$.
\enproof
\end{theorem}\medskip

\section{Asymptotic estimates} The next stage of our analysis is to give a complete asymptotic estimate for $\varphi^{(p,q)}_n(\lspace^r)$ 
for all relevant values of $p, q$ when $r>1$.
 
 \begin{theorem}\label{values of phi sequence for Lr}
 Let $\Omega$ be a measure space such that $\Lspace^r(\Omega)$ is infinite dimen\-sional, where $r> 1$. Set $\bar{r}=\min\set{r,2}$,
 and suppose that $1\le p\le q <\infty$. Then:
 \begin{enumerate}
 	\item[{\rm (i)}] $\varphi^{(p,q)}_n\left(\Lspace^r(\Omega)\right)\sim 1$  when $1/p-1/q\ge 1/\bar{r}$;\s\s
	\item[{\rm (ii)}] $\varphi^{(p,q)}_n\left(\Lspace^r(\Omega)\right)= n^{1/q}$  when $p\ge\bar{r}$;\s\s
	\item[{\rm (iii)}] $\displaystyle{\varphi^{(p,q)}_n\left(\Lspace^r(\Omega)\right)\sim n^{1/\bar{r}-1/p+1/q}}$  when $1/p-1/q\le 1/\bar{r}$ and $p\le\bar{r}$.
 \end{enumerate}
 In the case where $r=2$, all three estimates are actual equalities.
 \end{theorem}
 \begin{proof}
 Statements (i) and (ii) follow from Theorem \ref{equivalent of (p,q)- and min mns on Lr} and Proposition \ref{1.11}(iv), res\-pectively. 
 
 (iii) Suppose now that $1/p-1/q< 1/\bar{r}$ and that $p<\bar{r}$. Again, we need to consider only the space $\lspace^r$. By Proposition \ref{1.11}(iii) (when $r\ge2$) or by Example \ref{a calculation of (p,q)-norm on lr space} (when $r\le 2$), we see that
 \[
 	\varphi^{(p,q)}_n (\lspace^r)\ge n^{1/\bar{r}-1/p+1/q}\quad(n\in\naturals)\,.
 \]
 When $q=p$, we know by Theorem \ref{3.6}(iii) that $(\norm_n^{(p,p)}) \cong (\norm_n^{\max})$, and so
 \[
 	\varphi^{(p,p)}_n (\lspace^r)\sim\varphi^{\max}_n (\lspace^r)\sim n^{1/\bar{r}}
 \]
 by \cite[Theorem 3.54]{DP}.  Thus we need to consider only the case where $q>p$. 
 
 Set $q_1 =p$ and  $q_2=p\bar{r}/(\bar{r}-p)$, so that $1/p-1/q_2=1/\bar{r}$. We also see that $q_1<q<q_2$, and so 
 \[
 	\frac{1}{q}=\frac{1-\theta}{q_1}+\frac{\theta}{q_2},
 \]
 where $\theta=\bar{r}(1/p-1/q)$. Using Proposition \ref{interpolating phi sequence}, we deduce from (i) and the previous paragraph that
 \begin{align*}
 	\varphi^{(p,q)}_{n}(\lspace^r)\le \left(\varphi^{(p,p)}_{n}(\lspace^r)\right)^{1-\theta}\,\cdot\,\left(\varphi^{(p,q_2)}_{n}(\lspace^r)\right)^\theta
	\le C_r\,n^{(1-\theta)/\bar{r}}= C_r\,n^{1/\bar{r}-1/p+1/q}
 \end{align*}
 for all $n\in\N$, where $C_r$ is a constant depending only on $r$; when $r=2$, this constant can be taken to be $1$. 
 
 This completes the proof.
 \end{proof}
 
 We now obtain the following asymptotic estimates,  where $\tuple{f}$ is as in Example \ref{another calculation of (p,q)-norm on lr space} and the multi-norm is calculated with respect to $\ell^{\,r}$,  where $r\ge 2$:
 \begin{equation}\label{(2.5a)}
 \left\Vert \tuple{f}\right\Vert^{(p,q)}_n \sim \left\{\begin{array}{cl}n^{1/2-1/p+ 1/q } 
&\mbox{\rm when $1\leq p\leq 2$ and $1/p-1/q\leq 1/2$\,,}\\
 \\1&\mbox{\rm when $1/p-1/q> 1/2$\,,}\\
\\
n^{1/q}&\mbox{\rm when $p\geq 2$}\,.\end{array}\right.
\end{equation}

It is interesting to see where the maximum rate of growth is attained. Indeed,
suppose that $(p,q)\in {\mathcal T}$ and we are considering the rate of growth 
of the $(p,q)$-multi-norm on $\ell^{\,r}$, where $r\geq 1$. Then it follows from equation  \eqref{a calculation of (p,q)-norm on lr space: eq2}
 in Example  \ref{a calculation of (p,q)-norm on lr space} that
\[
\varphi^{(p,q)}_n(\ell^{\,r}) \sim \left\Vert (\delta_1, \dots,
\delta_n)\right\Vert^{(p,q)}_n\quad {\rm when}\quad r\leq 2
\] 
and from equation  (\ref{(2.5)})  in  Example  \ref{another calculation of (p,q)-norm on lr space} that
\[
\varphi^{(p,q)}_n(\ell^{\,r}) \sim \left\Vert (f_1, \dots,
f_n)\right\Vert^{(p,q)}_n\quad {\rm when}\quad r\geq 2\,,
\] 
where, for $i \in \N_n$, we are setting 
\[
f_i= \frac{1}{n^{1/r}}\sum_{j=1}^n\zeta^{-ij}\delta_j 
\quad\textrm{with}\ \zeta =\exp(2\pi{\rm i}/n)\,.
\]
Thus the maximum rate of growth is attained at either $(\delta_1, \dots, \delta_n)$
or at $(f_1, \dots, f_n)$.
\medskip

\section{Classification theorem} We now give our main classification result obtained in the case where $r>1$. For this, let us modify 
the curves ${\mathcal C}_c$ to obtain curves  ${\mathcal D}_c$ for $0\le c<1$ as follows. Set $\bar{r}=\min\set{2,r}$.
\begin{enumerate}
	\item[{\rm (i)}] The case where $c\in [1/\bar{r},1)$: Set ${\mathcal D}_c={\mathcal C}_c$.\s
	
	\item[{\rm (ii)}] The case where $c\in[0,1/r)$: Set $u_c=r/(1-cr)$, so that ${\mathcal C}_c$ meets the vertical line $p=r$ at $(r,u_c)$. 	Set
	\[
		{\mathcal D}_c=\set{(p,q)\in {\mathcal C}_c\colon p\in[1,r]}\cup\set{(p,u_c)\colon p\in [r,u_c]}\,.
	\]
	Thus ${\mathcal D}_c$ agrees with  ${\mathcal C}_c$  on the interval $[1,r]$ and is the horizontal line $q=u_c$ on the interval
 $[r,u_c]$.   In the case where $r<2$ and $c \in (1/2, 1/r)$, the point at which the line $q=u_c$ meets the curve ${\mathcal C}_{1/2}$ is denoted by 
$x_c$, \label{x_c} so that $r< x_c < 2$.

Note that $\mathcal D_0$ is the diagonal line segment $\set{(p,p)\colon 1\le p\le r}$.\s
 
	\item[{\rm (iii)}] The case where $c\in[1/r,1/2)$ (which only occurs when $r>2$): Set $v_c=2/(1-2c)$, so that ${\mathcal C}_c$ meets the 
vertical line $p=2$ at $(2,v_c)$, and set $w_c:=rv_c/(r-v_c)$, so that the horizontal line $q=v_c$ meets the curve ${\mathcal C}_{1/r}$ at $(w_c,v_c)$. Set
	\[
		{\mathcal D}_c=\set{(p,q)\in {\mathcal C}_c\colon p\in[1,2]}\cup\set{(p,v_c)\colon p\in [2,w_c]}\,.
	\]
	Thus ${\mathcal D}_c$ agrees with the old curve  ${\mathcal C}_c$  on the interval $[1,2]$, and then it becomes the horizontal line $q=v_c$ 
until this line meets the curve ${\mathcal C}_{1/r}$, at which point it terminates. Note that $\mathcal D_{1/r}$ is the curve $\mathcal C_{1/r}$
 restricted to the interval $[1,2]$.\s 
\end{enumerate}

Note that $\bigcup\set{{\mathcal D}_c\colon 0\le c<1}=\mathcal T$. Note also that, unlike the curves $\mathcal C_c$, the curves $\mathcal D_c$ depend 
on the value of $r$. The case where $r>2$ is illustrated in the diagram, in which we present in bold the curves $\mathcal D_c$ when $c\ge 1/2$, 
when $c\in (1/r,1/2)$, when $c=1/r$, when $c\in (0,1/r)$, and when $c=0$.
\smallskip

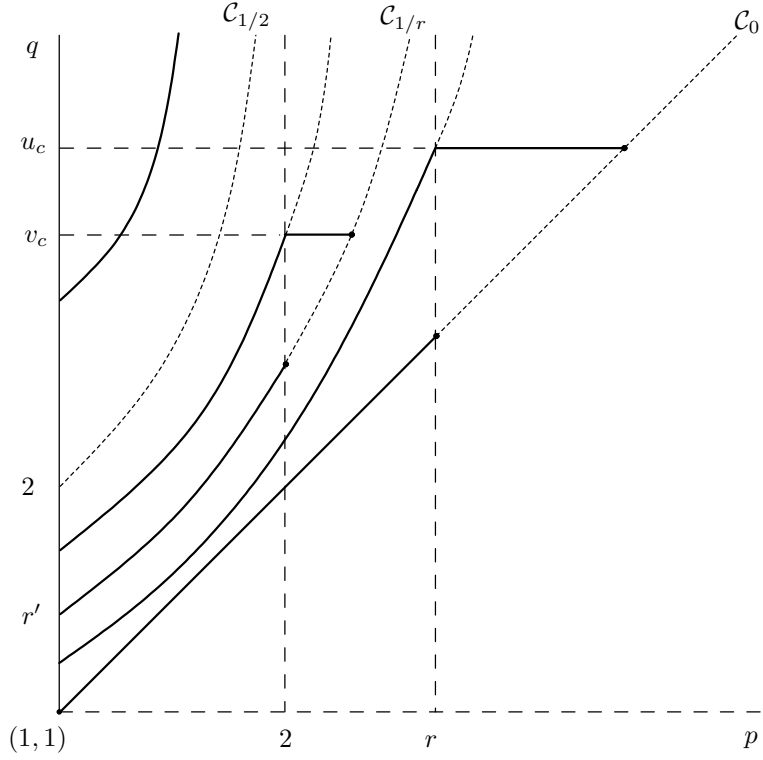
\unitlength 1pt
\begin{picture}(299.8209,278.1257)( 10.6698,-291.9964)\label{PICTURE2}
%
\special{pn 8}%
\special{sh 1}%
\special{ar 586 3938 10 10 0  6.28318530717959E+0000}%
\special{sh 1}%
\special{ar 586 3938 10 10 0  6.28318530717959E+0000}%
%
\special{pn 8}%
\special{pa 586 3938}%
\special{pa 4297 3938}%
\special{da 0.070}%
%
\special{pn 8}%
\special{pa 586 3938}%
\special{pa 586 394}%
\special{fp}%
%
\special{pn 8}%
\special{pa 2555 394}%
\special{pa 2555 1969}%
\special{da 0.070}%
%
\special{pn 8}%
\special{pa 2555 1969}%
\special{pa 2555 3938}%
\special{da 0.070}%
%
\special{pn 8}%
\special{pa 0590 1440}%
\special{pa 1740 1440}%
\special{da 0.070}%
%
\special{pn 8}%
\special{pa 0590 985}%
\special{pa 2500 985}%
\special{da 0.070}%
\put(34.1433,-295){\makebox(0,0){$(1,1)$}}%
\put(304.0888,-295){\makebox(0,0){$p$}}%
\put(32.2726,-34){\makebox(0,0){$q$}}%
\put(125.4346,-298){\makebox(0,0)[lb]{$2$}}%
\put(180.6288,-298){\makebox(0,0)[lb]{$r$}}%
%
\special{pn 8}%
\special{pa 1767 394}%
\special{pa 1767 2756}%
\special{da 0.070}%
%
\special{pn 8}%
\special{pa 1767 2756}%
\special{pa 1767 3938}%
\special{da 0.070}%
%
\special{pn 13}%
\special{pa 1767 1438}%
\special{pa 2122 1438}%
\special{fp}%
%
\special{pn 13}%
\special{pa 586 3682}%
\special{pa 612 3664}%
\special{pa 637 3645}%
\special{pa 664 3627}%
\special{pa 689 3609}%
\special{pa 715 3591}%
\special{pa 741 3572}%
\special{pa 766 3554}%
\special{pa 792 3535}%
\special{pa 816 3516}%
\special{pa 842 3498}%
\special{pa 868 3479}%
\special{pa 892 3460}%
\special{pa 917 3440}%
\special{pa 942 3421}%
\special{pa 967 3401}%
\special{pa 991 3381}%
\special{pa 1015 3362}%
\special{pa 1040 3341}%
\special{pa 1063 3320}%
\special{pa 1087 3300}%
\special{pa 1111 3279}%
\special{pa 1134 3258}%
\special{pa 1157 3237}%
\special{pa 1181 3215}%
\special{pa 1203 3193}%
\special{pa 1226 3171}%
\special{pa 1249 3149}%
\special{pa 1270 3126}%
\special{pa 1293 3104}%
\special{pa 1314 3081}%
\special{pa 1336 3059}%
\special{pa 1358 3035}%
\special{pa 1379 3011}%
\special{pa 1400 2988}%
\special{pa 1421 2964}%
\special{pa 1442 2940}%
\special{pa 1462 2917}%
\special{pa 1483 2892}%
\special{pa 1503 2869}%
\special{pa 1523 2844}%
\special{pa 1543 2819}%
\special{pa 1563 2795}%
\special{pa 1582 2769}%
\special{pa 1602 2745}%
\special{pa 1621 2720}%
\special{pa 1639 2694}%
\special{pa 1658 2669}%
\special{pa 1677 2643}%
\special{pa 1695 2618}%
\special{pa 1713 2592}%
\special{pa 1731 2566}%
\special{pa 1749 2540}%
\special{pa 1766 2514}%
\special{pa 1784 2488}%
\special{pa 1802 2462}%
\special{pa 1818 2436}%
\special{pa 1835 2409}%
\special{pa 1852 2382}%
\special{pa 1869 2356}%
\special{pa 1885 2329}%
\special{pa 1902 2303}%
\special{pa 1918 2275}%
\special{pa 1934 2249}%
\special{pa 1950 2221}%
\special{pa 1966 2194}%
\special{pa 1982 2167}%
\special{pa 1997 2139}%
\special{pa 2012 2112}%
\special{pa 2027 2084}%
\special{pa 2043 2057}%
\special{pa 2058 2029}%
\special{pa 2072 2001}%
\special{pa 2087 1974}%
\special{pa 2102 1945}%
\special{pa 2117 1918}%
\special{pa 2131 1890}%
\special{pa 2145 1862}%
\special{pa 2160 1834}%
\special{pa 2174 1806}%
\special{pa 2188 1777}%
\special{pa 2202 1750}%
\special{pa 2216 1721}%
\special{pa 2230 1692}%
\special{pa 2244 1664}%
\special{pa 2257 1635}%
\special{pa 2271 1607}%
\special{pa 2285 1578}%
\special{pa 2299 1550}%
\special{pa 2312 1521}%
\special{pa 2325 1493}%
\special{pa 2339 1464}%
\special{pa 2352 1436}%
\special{pa 2365 1406}%
\special{pa 2378 1377}%
\special{pa 2391 1349}%
\special{pa 2405 1319}%
\special{pa 2418 1291}%
\special{pa 2431 1262}%
\special{pa 2443 1233}%
\special{pa 2457 1204}%
\special{fp}%
\special{pa 2457 1204}%
\special{pa 2470 1176}%
\special{pa 2483 1146}%
\special{pa 2496 1118}%
\special{pa 2508 1088}%
\special{pa 2521 1060}%
\special{pa 2534 1030}%
\special{pa 2547 1001}%
\special{pa 2555 985}%
\special{sp}%
%
\special{pn 13}%
\special{pa 2555 985}%
\special{pa 2555 985}%
\special{pa 2555 985}%
\special{pa 2555 985}%
\special{pa 2555 985}%
\special{pa 2555 985}%
\special{pa 2555 985}%
\special{pa 2555 985}%
\special{pa 2555 985}%
\special{pa 2555 985}%
\special{pa 2555 985}%
\special{pa 2555 985}%
\special{pa 2555 985}%
\special{pa 2555 985}%
\special{pa 2555 985}%
\special{pa 2555 985}%
\special{pa 2555 985}%
\special{pa 2555 985}%
\special{pa 2555 985}%
\special{pa 2555 985}%
\special{pa 2555 985}%
\special{pa 2555 985}%
\special{sp}%
%
\special{pn 13}%
\special{pa 2555 985}%
\special{pa 2555 985}%
\special{pa 2555 985}%
\special{pa 2555 985}%
\special{pa 2555 985}%
\special{pa 2555 985}%
\special{pa 2555 985}%
\special{pa 2555 985}%
\special{pa 2555 985}%
\special{pa 2555 985}%
\special{pa 2555 985}%
\special{pa 2555 985}%
\special{pa 2555 985}%
\special{pa 2555 985}%
\special{pa 2555 985}%
\special{pa 2555 985}%
\special{pa 2555 985}%
\special{pa 2555 985}%
\special{pa 2555 985}%
\special{pa 2555 985}%
\special{pa 2555 985}%
\special{pa 2555 985}%
\special{sp}%
%
\special{pn 8}%
\special{pa 2555 985}%
\special{pa 2555 985}%
\special{pa 2555 985}%
\special{pa 2555 985}%
\special{pa 2555 985}%
\special{pa 2555 985}%
\special{pa 2555 985}%
\special{pa 2555 985}%
\special{pa 2555 985}%
\special{pa 2555 985}%
\special{pa 2555 985}%
\special{pa 2555 985}%
\special{pa 2555 985}%
\special{pa 2555 985}%
\special{pa 2555 985}%
\special{pa 2555 985}%
\special{pa 2555 985}%
\special{pa 2555 985}%
\special{pa 2555 985}%
\special{pa 2555 985}%
\special{pa 2555 985}%
\special{pa 2555 985}%
\special{sp}%
%
\special{pn 13}%
\special{pa 2555 985}%
\special{pa 3539 985}%
\special{fp}%
%
\special{pn 8}%
\special{pa 2555 985}%
\special{pa 2567 956}%
\special{pa 2579 927}%
\special{pa 2592 898}%
\special{pa 2605 869}%
\special{pa 2617 839}%
\special{pa 2628 811}%
\special{pa 2639 781}%
\special{pa 2651 751}%
\special{pa 2662 722}%
\special{pa 2672 692}%
\special{pa 2682 663}%
\special{pa 2691 632}%
\special{pa 2700 602}%
\special{pa 2708 572}%
\special{pa 2716 542}%
\special{pa 2724 511}%
\special{pa 2732 480}%
\special{pa 2739 449}%
\special{pa 2746 419}%
\special{pa 2751 394}%
\special{sp 0.020}%
%
\special{pn 8}%
\special{pa 1772 1438}%
\special{pa 1783 1408}%
\special{pa 1794 1378}%
\special{pa 1804 1348}%
\special{pa 1814 1318}%
\special{pa 1825 1289}%
\special{pa 1835 1258}%
\special{pa 1845 1229}%
\special{pa 1855 1199}%
\special{pa 1865 1169}%
\special{pa 1875 1139}%
\special{pa 1883 1109}%
\special{pa 1892 1078}%
\special{pa 1901 1049}%
\special{pa 1909 1018}%
\special{pa 1917 988}%
\special{pa 1924 957}%
\special{pa 1931 927}%
\special{pa 1938 896}%
\special{pa 1943 865}%
\special{pa 1949 834}%
\special{pa 1955 804}%
\special{pa 1960 772}%
\special{pa 1965 741}%
\special{pa 1970 710}%
\special{pa 1975 679}%
\special{pa 1979 647}%
\special{pa 1983 617}%
\special{pa 1988 585}%
\special{pa 1991 554}%
\special{pa 1995 522}%
\special{pa 1999 491}%
\special{pa 2002 459}%
\special{pa 2005 428}%
\special{pa 2008 404}%
\special{sp 0.020}%
%
\special{pn 8}%
\special{pa 591 2756}%
\special{pa 614 2734}%
\special{pa 635 2711}%
\special{pa 657 2688}%
\special{pa 680 2666}%
\special{pa 701 2643}%
\special{pa 723 2621}%
\special{pa 746 2598}%
\special{pa 767 2574}%
\special{pa 789 2552}%
\special{pa 810 2528}%
\special{pa 831 2505}%
\special{pa 852 2482}%
\special{pa 874 2458}%
\special{pa 894 2435}%
\special{pa 914 2410}%
\special{pa 935 2386}%
\special{pa 954 2362}%
\special{pa 974 2337}%
\special{pa 994 2313}%
\special{pa 1012 2288}%
\special{pa 1031 2262}%
\special{pa 1050 2237}%
\special{pa 1067 2211}%
\special{pa 1085 2185}%
\special{pa 1102 2159}%
\special{pa 1120 2132}%
\special{pa 1136 2105}%
\special{pa 1152 2078}%
\special{pa 1168 2051}%
\special{pa 1184 2023}%
\special{pa 1198 1996}%
\special{pa 1214 1968}%
\special{pa 1228 1939}%
\special{pa 1243 1911}%
\special{pa 1256 1882}%
\special{pa 1270 1854}%
\special{pa 1283 1825}%
\special{pa 1296 1797}%
\special{pa 1309 1767}%
\special{pa 1320 1738}%
\special{pa 1332 1708}%
\special{pa 1344 1679}%
\special{pa 1355 1649}%
\special{pa 1366 1620}%
\special{pa 1376 1590}%
\special{pa 1387 1560}%
\special{pa 1397 1530}%
\special{pa 1406 1500}%
\special{pa 1416 1470}%
\special{pa 1425 1439}%
\special{pa 1434 1409}%
\special{pa 1442 1379}%
\special{pa 1450 1349}%
\special{pa 1458 1318}%
\special{pa 1466 1288}%
\special{pa 1473 1257}%
\special{pa 1481 1226}%
\special{pa 1488 1195}%
\special{pa 1495 1165}%
\special{pa 1501 1134}%
\special{pa 1507 1103}%
\special{pa 1513 1072}%
\special{pa 1519 1042}%
\special{pa 1525 1010}%
\special{pa 1530 980}%
\special{pa 1536 948}%
\special{pa 1541 918}%
\special{pa 1546 886}%
\special{pa 1552 855}%
\special{pa 1556 824}%
\special{pa 1561 793}%
\special{pa 1565 761}%
\special{pa 1570 730}%
\special{pa 1574 699}%
\special{pa 1579 668}%
\special{pa 1583 636}%
\special{pa 1587 605}%
\special{pa 1591 573}%
\special{pa 1596 542}%
\special{pa 1600 511}%
\special{pa 1604 480}%
\special{pa 1608 448}%
\special{pa 1612 417}%
\special{pa 1615 394}%
\special{sp 0.020}%
%
\special{pn 13}%
\special{pa 591 1782}%
\special{pa 615 1760}%
\special{pa 638 1739}%
\special{pa 661 1717}%
\special{pa 685 1694}%
\special{pa 707 1673}%
\special{pa 730 1650}%
\special{pa 751 1627}%
\special{pa 774 1605}%
\special{pa 795 1582}%
\special{pa 815 1559}%
\special{pa 836 1535}%
\special{pa 856 1511}%
\special{pa 875 1487}%
\special{pa 892 1462}%
\special{pa 909 1437}%
\special{pa 926 1411}%
\special{pa 941 1384}%
\special{pa 957 1358}%
\special{pa 972 1331}%
\special{pa 986 1304}%
\special{pa 999 1276}%
\special{pa 1011 1248}%
\special{pa 1024 1219}%
\special{pa 1035 1190}%
\special{pa 1047 1161}%
\special{pa 1057 1132}%
\special{pa 1067 1103}%
\special{pa 1077 1072}%
\special{pa 1086 1043}%
\special{pa 1095 1012}%
\special{pa 1103 982}%
\special{pa 1111 950}%
\special{pa 1119 920}%
\special{pa 1126 888}%
\special{pa 1133 857}%
\special{pa 1139 825}%
\special{pa 1146 793}%
\special{pa 1152 761}%
\special{pa 1158 729}%
\special{pa 1164 696}%
\special{pa 1169 665}%
\special{pa 1175 632}%
\special{pa 1180 600}%
\special{pa 1185 566}%
\special{pa 1189 534}%
\special{pa 1194 501}%
\special{pa 1199 469}%
\special{pa 1204 436}%
\special{pa 1208 403}%
\special{pa 1211 384}%
\special{sp}%
%
\special{pn 8}%
\special{pa 1772 2107}%
\special{pa 1788 2079}%
\special{pa 1804 2052}%
\special{pa 1818 2024}%
\special{pa 1834 1997}%
\special{pa 1850 1969}%
\special{pa 1865 1941}%
\special{pa 1880 1914}%
\special{pa 1895 1886}%
\special{pa 1911 1858}%
\special{pa 1926 1830}%
\special{pa 1940 1803}%
\special{pa 1955 1775}%
\special{pa 1970 1747}%
\special{pa 1985 1719}%
\special{pa 1999 1690}%
\special{pa 2012 1663}%
\special{pa 2027 1634}%
\special{pa 2041 1606}%
\special{pa 2054 1578}%
\special{pa 2067 1550}%
\special{pa 2080 1521}%
\special{pa 2093 1493}%
\special{pa 2106 1463}%
\special{pa 2118 1435}%
\special{pa 2129 1406}%
\special{pa 2141 1376}%
\special{pa 2153 1347}%
\special{pa 2164 1318}%
\special{pa 2175 1289}%
\special{pa 2186 1259}%
\special{pa 2196 1230}%
\special{pa 2206 1200}%
\special{pa 2216 1170}%
\special{pa 2226 1140}%
\special{pa 2236 1111}%
\special{pa 2246 1080}%
\special{pa 2254 1051}%
\special{pa 2263 1020}%
\special{pa 2272 990}%
\special{pa 2281 959}%
\special{pa 2290 930}%
\special{pa 2299 899}%
\special{pa 2307 869}%
\special{pa 2314 838}%
\special{pa 2323 808}%
\special{pa 2331 777}%
\special{pa 2339 747}%
\special{pa 2347 715}%
\special{pa 2354 685}%
\special{pa 2362 654}%
\special{pa 2370 624}%
\special{pa 2376 592}%
\special{pa 2384 562}%
\special{pa 2391 531}%
\special{pa 2399 501}%
\special{pa 2406 469}%
\special{pa 2414 438}%
\special{pa 2421 408}%
\special{pa 2422 404}%
\special{sp 0.020}%
%
\special{pn 13}%
\special{pa 591 3938}%
\special{pa 2560 1969}%
\special{fp}%
%
\special{pn 13}%
\special{sh 1}%
\special{ar 3544 985 10 10 0  6.28318530717959E+0000}%
\special{sh 1}%
\special{ar 3544 985 10 10 0  6.28318530717959E+0000}%
%
\special{pn 13}%
\special{sh 1}%
\special{ar 2560 1969 10 10 0  6.28318530717959E+0000}%
\special{sh 1}%
\special{ar 2560 1969 10 10 0  6.28318530717959E+0000}%
%
\special{pn 13}%
\special{sh 1}%
\special{ar 2117 1438 10 10 0  6.28318530717959E+0000}%
\special{sh 1}%
\special{ar 2117 1438 10 10 0  6.28318530717959E+0000}%
%
\special{pn 13}%
\special{sh 1}%
\special{ar 1772 2117 10 10 0  6.28318530717959E+0000}%
\special{sh 1}%
\special{ar 1772 2117 10 10 0  6.28318530717959E+0000}%
%
\special{pn 13}%
\special{pa 1772 2117}%
\special{pa 1755 2143}%
\special{pa 1739 2171}%
\special{pa 1721 2197}%
\special{pa 1704 2224}%
\special{pa 1688 2251}%
\special{pa 1670 2278}%
\special{pa 1653 2305}%
\special{pa 1635 2331}%
\special{pa 1619 2358}%
\special{pa 1601 2384}%
\special{pa 1583 2411}%
\special{pa 1566 2437}%
\special{pa 1549 2463}%
\special{pa 1530 2490}%
\special{pa 1512 2515}%
\special{pa 1495 2541}%
\special{pa 1476 2566}%
\special{pa 1458 2592}%
\special{pa 1439 2618}%
\special{pa 1421 2643}%
\special{pa 1401 2668}%
\special{pa 1382 2692}%
\special{pa 1363 2718}%
\special{pa 1343 2742}%
\special{pa 1323 2766}%
\special{pa 1304 2790}%
\special{pa 1283 2814}%
\special{pa 1262 2838}%
\special{pa 1242 2861}%
\special{pa 1221 2884}%
\special{pa 1199 2907}%
\special{pa 1178 2930}%
\special{pa 1155 2951}%
\special{pa 1133 2974}%
\special{pa 1111 2996}%
\special{pa 1088 3017}%
\special{pa 1065 3039}%
\special{pa 1042 3060}%
\special{pa 1018 3080}%
\special{pa 995 3102}%
\special{pa 971 3123}%
\special{pa 947 3142}%
\special{pa 923 3163}%
\special{pa 899 3184}%
\special{pa 875 3203}%
\special{pa 850 3223}%
\special{pa 825 3243}%
\special{pa 801 3262}%
\special{pa 776 3282}%
\special{pa 751 3302}%
\special{pa 726 3321}%
\special{pa 701 3341}%
\special{pa 676 3361}%
\special{pa 651 3379}%
\special{pa 626 3399}%
\special{pa 601 3419}%
\special{pa 591 3426}%
\special{sp}%
\put(32.5867,-70.1693){\makebox(0,0){$u_c$}}%
\put(33.5867,-105.1693){\makebox(0,0){$v_c$}}%
\put(30.5867,-199.1693){\makebox(0,0){$2$}}%
\put(32.0094,-247.5390){\makebox(0,0){$r'$}}%
\put(113.3676,-21.3396){\makebox(0,0){$\mathcal C_{{1}/{2}}$}}%
\put(172.6958,-22.5396){\makebox(0,0){$\mathcal C_{{1}/{r}}$}}%
\put(302.6958,-23.5396){\makebox(0,0){$\mathcal C_{0}$}}%
%
\special{pn 8}%
\special{pa 2560 1969}%
\special{pa 4134 394}%
\special{da 0.020}%
%
\special{pn 13}%
\special{pa 591 3091}%
\special{pa 616 3071}%
\special{pa 641 3051}%
\special{pa 666 3031}%
\special{pa 691 3010}%
\special{pa 716 2991}%
\special{pa 741 2970}%
\special{pa 765 2950}%
\special{pa 790 2930}%
\special{pa 814 2909}%
\special{pa 839 2889}%
\special{pa 864 2869}%
\special{pa 887 2848}%
\special{pa 912 2826}%
\special{pa 936 2806}%
\special{pa 959 2785}%
\special{pa 982 2763}%
\special{pa 1005 2742}%
\special{pa 1028 2720}%
\special{pa 1051 2698}%
\special{pa 1073 2677}%
\special{pa 1096 2654}%
\special{pa 1118 2632}%
\special{pa 1139 2610}%
\special{pa 1160 2586}%
\special{pa 1182 2563}%
\special{pa 1202 2540}%
\special{pa 1222 2516}%
\special{pa 1242 2493}%
\special{pa 1261 2469}%
\special{pa 1280 2444}%
\special{pa 1299 2420}%
\special{pa 1317 2394}%
\special{pa 1335 2369}%
\special{pa 1353 2343}%
\special{pa 1370 2317}%
\special{pa 1386 2291}%
\special{pa 1403 2265}%
\special{pa 1419 2238}%
\special{pa 1435 2211}%
\special{pa 1449 2185}%
\special{pa 1465 2157}%
\special{pa 1480 2129}%
\special{pa 1494 2101}%
\special{pa 1508 2073}%
\special{pa 1522 2045}%
\special{pa 1536 2017}%
\special{pa 1550 1989}%
\special{pa 1563 1959}%
\special{pa 1575 1931}%
\special{pa 1588 1902}%
\special{pa 1601 1873}%
\special{pa 1614 1843}%
\special{pa 1626 1813}%
\special{pa 1638 1784}%
\special{pa 1650 1754}%
\special{pa 1662 1725}%
\special{pa 1674 1695}%
\special{pa 1686 1666}%
\special{pa 1697 1635}%
\special{pa 1709 1606}%
\special{pa 1720 1575}%
\special{pa 1732 1546}%
\special{pa 1743 1515}%
\special{pa 1754 1486}%
\special{pa 1765 1455}%
\special{pa 1772 1438}%
\special{sp}%
\end{picture}%

\phantom{a}

\medskip

\begin{theorem}\label{non equivalent of (p,q) mns on Lr, on different curves}
Take $r>1$,   let $\Omega$ be a measure space such that $\Lspace^r(\Omega)$ is infinite dimen\-sional, and  set $\bar{r}=\min\set{2,r}$. Take $c_1,c_2\in [0,1)$, and 
consider points $P_1\in {\mathcal D}_{c_1}$ and  $P_2\in {\mathcal D}_{c_2}$, respectively.\s
\begin{enumerate}
\item[{\rm (i)}] Suppose that $c_1,c_2\in [1/\bar{r},1)$. Then $P_1$ and $P_2$ are equivalent $($and the corresponding $(p,q)$-multi-norms are
 equivalent to the minimum multi-norm\,$)$ on $\Lspace^r(\Omega)$.\s
\item[{\rm (ii)}] Suppose that $c_1\in [1/\bar{r},1)$ and $c_2\in [0,1/\bar{r})$. Then $P_1$ and $P_2$ are not equivalent  on $\Lspace^r(\Omega)$.\s
\item[{\rm (iii)}] Suppose that $c_1,c_2\in [0,1/\bar{r})$ and that $c_1\neq c_2$. Then $P_1$ and $P_2$ are not equivalent on $\Lspace^r(\Omega)$. 
\end{enumerate}\s
\end{theorem}
\begin{proof}
Clause (i) follows from Theorem \ref{equivalent of (p,q)- and min mns on Lr}, whereas (ii) follows from Theorem \ref{values of phi sequence for Lr}. 

It remains to prove clause  (iii). For this, we suppose that $c_1,c_2\in [0,1/\bar{r})$ and that $c_1\neq c_2$.\s

Assume towards a contradiction that $P_1$ and $P_2$ are equivalent  on $\Lspace^r(\Omega)$.\s

\noindent\emph{Case 1:} $p_1,p_2\le \bar{r}$. In this case, the desired contradiction follows from Theorem
 \ref{values of phi sequence for Lr}(iii), noting that $P_i\in\mathcal C_{c_i}$ for both $i=1$ and $i=2$ in this case. \s

\noindent\emph{Case 2:} $p_1,p_2\ge\bar{r}$. In this case, we must have $q_1=q_2$ by Theorem \ref{values of phi sequence for Lr}(ii).
 From the definition of the curves $\mathcal D_c$, this can happen (with $c_1\neq c_2$) only when  $\min\set{p_1,p_2}<r$, and so $r>2$, and $\min\set{c_1,c_2}<1/r$.  
In particular, we must have $\bar{r}=2$. By Proposition \ref{non-equivalent mns on Lr}(i), we must have $\max\set{p_1,p_2}>r$. 
Thus, without loss of generality, suppose that $p_1> r>p_2\ge 2$. Proposition  \ref{non-equivalent mns on Lr}(iv) then implies that
\[
	\frac{1}{r}-\frac{1}{p_2}+\frac{1}{q_2}=\frac{1}{q_1}\,.
\]
Since $q_1=q_2$, this implies that $p_2=r$, a contradiction.\s

It remains to consider the case where $p_1<\bar{r}<p_2$; the case where $p_1>\bar{r}>p_2$ is dealt with similarly. We divide this case further
 into the following two cases.\s

\noindent\emph{Case 3:} $r\le 2$, so that $\bar{r}=r$, and $p_1<r<p_2$. In this case, it follows from either Theorem \ref{values of phi sequence for Lr} 
or Proposition \ref{non-equivalent mns on Lr}(iii) that
\[
	\frac{1}{r}-\frac{1}{p_1}+\frac{1}{q_1}=\frac{1}{q_2}\,.
\]
This implies first that $(r,q_2)\in\mathcal C_{c_1}\cap\mathcal D_{c_1}$, and then that $(p_2,q_2)\in \mathcal D_{c_1}$ by the definition of
 $\mathcal D_{c_1}$, a contradiction of the assumption that $c_1\neq c_2$.\s

\noindent\emph{Case 4:} $r> 2$, so that $\bar{r}=2$, and $p_1<2<p_2$. In this case, it follows from Theorem \ref{values of phi sequence for Lr} that
\[
	\frac{1}{2}-\frac{1}{p_1}+\frac{1}{q_1}=\frac{1}{q_2}\,.
\]
This implies that $(2,q_2)\in\mathcal C_{c_1}\cap\mathcal D_{c_1}$. So it follows from the definition of $\mathcal D_{c}$ and the assumption
 that $(p_2,q_2)\notin \mathcal D_{c_1}$ that $c_2<1/r$. By Proposition \ref{non-equivalent mns on Lr}(i), we deduce that $p_2>r$. But then 
Proposition \ref{non-equivalent mns on Lr}(iii) implies that 
\[
	\frac{1}{r}-\frac{1}{p_1}+\frac{1}{q_1}=\frac{1}{q_2}\,,
\]
and so $r=2$, again a contradiction.

This concludes the proof of the theorem.
\end{proof}

\medskip
\section{The role of Khintchine's inequalities}
The previous theorem reduces our problem to that of determining the equivalence of two points $P_1$ and $P_2$ lying on the same curve ${\mathcal D}_c$,
 where $c \in [1, 1/\overline{r})$.   For further progress, we shall use Khintchine's inequalities, for which  see \cite[Chapter 12]{Gar3}, for example.

Let $n\in\naturals$. We shall consider $(\varepsilon_{i,j})$ to be a fixed $n\times 2^n$ matrix with entries in $\set{-1,1}$ such that 
its $2^n$ columns range over all possible choices of $n$-tuples of $\pm 1$. The Khintchine inequality 
 tells us that, for each $r>1$, there exist constants $A_r, B_r>0$, depending only on $r$ (but not on $n$), such that
\[
	A_r\left(\sum_{i=1}^n\abs{\alpha_i}^2\right)^{1/2}\le \left(\frac{1}{2^n}\sum_{j=1}^{2^n}\abs{\sum_{i=1}^n\varepsilon_{i,j}\alpha_i}^r\right)^{1/r}
\le B_r\left(\sum_{i=1}^n\abs{\alpha_i}^2\right)^{1/2}
\]
for every $\alpha_1,\ldots,\alpha_n\in\C$ and every $n\in\naturals$.  These constants are those specified in   the next lemma.\s

\begin{lemma}\label{reduce lr to l2}
Let $r>1$,   and take $n\in\N$. Then there exists a linear monomorphism $R_n:\lspace^2_n\to \lspace^r$   such that
\[
	\frac{1}{B_{r'}}\,\lV(x_1,\ldots,x_n)\rV^{(p,q)}_n\le \lV(R_nx_1,\ldots,R_nx_n)\rV^{(p,q)}_n\le B_r\lV(x_1,\ldots,x_n)\rV^{(p,q)}_n
\]
whenever $1\le p\le q <\infty$ and $x_1,\ldots,x_n\in\lspace^2_n$.
\end{lemma}
\begin{proof}
Set $s=r'$, the conjugate index to $r$, so that $1<s<\infty$. For each $i\in\N_n$, set
\[
	g_i=\frac{1}{2^{n/r}}(\varepsilon_{i,1},\ldots,\varepsilon_{i,2^n},0,0,\ldots)\in\lspace^r\quad\textrm{and}\quad \varphi_i=\frac{1}{2^{n/s}}(\varepsilon_{i,1},\ldots,\varepsilon_{i,2^n},0,0,\ldots)\in\lspace^s\,.
\]
The maps $\delta_i\mapsto g_i$ and $\delta_i\mapsto \varphi_i$ extend linearly to linear operators $R:\lspace^2_n\to\lspace^r$ and $S:\lspace^2_n\to\lspace^s$, respectively. Moreover, by the Khintchine inequality, we see that
\[
	A_r\lV x\rV_{\ell^{\,2}}\le \lV Rx\rV_{\ell^{\,r}}\le B_r\lV x\rV_{\ell^{\,2}}
\quad\textrm{and}\quad 
A_s\lV x\rV_{\ell^{\,2}}\le \lV Sx\rV_{\ell^{\,s}}\le B_s\lV x\rV_{\ell^{\,2}}\quad(x\in\lspace^2_n)\,,
\]
so that, in particular, both $R$ and $S$ are linear monomorphisms. We  see also that
\[
	\duality{Rx}{Sy}=\duality{x}{y}\qquad(x,y\in\lspace^2_n)\,,
\]
where we identify $\dual{(\lspace^2_n)}=\lspace^2_n$ in an obvious manner.

Take $p,q\in {\mathcal T}$ and take $x_1,\ldots,x_n\in\lspace^2_n$. We then see that
\begin{align*}
	&\lV(Rx_1,\ldots,Rx_n)\rV^{(p,q)}_n\\
&=\sup\set{\left(\sum_{i=1}^n\abs{\duality{Rx_i}{\lambda_i}}^q\right)^{1/q}\colon\ \lambda_1,\ldots,\lambda_n\in \lspace^s,\ \mu_{p,n}(\lambda_1,\ldots,\lambda_n)\le 1}\\
&=\sup\set{\left(\sum_{i=1}^n\abs{\duality{x_i}{\dual{R}\lambda_i}}^q\right)^{1/q}\colon\ \lambda_1,\ldots,\lambda_n\in \lspace^s,\
 \mu_{p,n}(\lambda_1,\ldots,\lambda_n)\le 1}\\
&\le \sup\set{\left(\sum_{i=1}^n\abs{\duality{x_i}{y_i}}^q\right)^{1/q}\colon\ y_1,\ldots,y_n\in \lspace^2_n,\ \mu_{p,n}(y_1,\ldots,y_n)\le B_r}\\
& = B_r\lV(x_1,\ldots,x_n)\rV^{(p,q)}_n\,.
\end{align*}
On the other hand, from the first equation above, we also see that
\begin{align*}
	&\lV(Rx_1,\ldots,Rx_n)\rV^{(p,q)}_n\\
&\ge \sup\set{\left(\sum_{i=1}^n\abs{\duality{Rx_i}{Sy_i}}^q\right)^{1/q}\colon\ y_1,\ldots,y_n\in \lspace^2_n,\ \mu_{p,n}(y_1,\ldots,y_n)\le 1/B_s}\\
&=\sup\set{\left(\sum_{i=1}^n\abs{\duality{x_i}{y_i}}^q\right)^{1/q}\colon\ y_1,\ldots,y_n\in \lspace^2_n,\ \mu_{p,n}(y_1,\ldots,y_n)\le 1/B_s}\\
&= \frac{1}{B_s}\,\lV(x_1,\ldots,x_n)\rV^{(p,q)}_n\,.
\end{align*}
Thus, setting $R_n:=R$, we obtain the desired operator.
\end{proof}

\smallskip

\begin{theorem}\label{3.13} Suppose that  $(p_1, q_1),(p_2,q_2) \in {\mathcal T}$. Assume that the $(p_1,q_1)$- and $(p_2,q_2)$-multi-norms 
are not equivalent on $\lspace^2$.  Then, for every $r>1$, the $(p_1,q_1)$- and $(p_2,q_2)$-multi-norms are not equivalent on $\lspace^r$.
\end{theorem}

\begin{proof}
Take  $r>1$. Without loss of generality, by the assumption, we see that there exist a sequence $(\alpha_n)$ in $(0,\infty)$ with
 $\alpha_n\nearrow \infty$ and a sequence $(\tuple{x}_n)$ where, for each $n\in\N$, $\tuple{x}_n=(x_{n,1},\ldots,x_{n,n})\in (\lspace^2)^n$, such that
\[
	\lV(x_{n,1},\ldots,x_{n,n})\rV^{(p_1,q_1)}_n>\alpha_n\lV(x_{n,1},\ldots,x_{n,n})\rV^{(p_2,q_2)}_n\,.
\]
It is obvious that we may consider $x_{n,1},\ldots,x_{n,n}$ as belonging to $\lspace^2_n$. Now set $y_{n,i}=R_nx_{n,i}$ for each $i\in\N_n$,
 where $R_n$ is the map defined in the previous lemma. We then obtain, for each $n\in\N$, a tuple  $(y_{n,1},\ldots,y_{n,n})\in (\lspace^r)^n$ such that
\[
	\lV(y_{n,1},\ldots,y_{n,n})\rV^{(p_1,q_1)}_n>\frac{\alpha_n}{B_rB_{r'}}\,\lV(y_{n,1},\ldots,y_{n,n})\rV^{(p_2,q_2)}_n\,.
\]

Thus $(\norm^{(p_1,q_1)}_n)$ and $(\norm^{(p_1,q_1)}_n)$ are not equivalent on $\lspace^r$. This completes the proof.
\end{proof}

\begin{corollary}\label{tool} Let   $(p_1, q_1),(p_2,q_2) \in {\mathcal T}$. Suppose that  
 $\Pi_{q_1,p_1}(\lspace^2,F) \neq \Pi_{q_2,p_2}(\lspace^2,F)$ for some Banach space $F$. 
Then, for each $r>1$,  the $(p_1,q_1)$- and $(p_2,q_2)$-multi-norms are not equivalent on $\lspace^r$.
\end{corollary}

\begin{proof}
This follows from  the previous theorem and Theorem \ref{tool for general spaces}, using the Riesz representation theorem.
\end{proof}
\medskip

\section{Final classification}
Theorem \ref{3.13}  suggests that we study more closely the spaces  $\Pi_{q,p}(H)$, where $H$ is a Hilbert space, and this we shall do to obtain
 the final classification that we can achieve. 

We first state some results that identify $\Pi_{q,p}(H)$. 
Clause (i) of the following theorem combines  Corollaries 3.16 and 4.13 of \cite{DJT}, and the remaining clauses are stated on page  207 of \cite{DJT}.
In fact, clauses (ii) and (iii) of the following  theorem originate in \cite[Theorem~2]{Kwapien1} (where this result is attributed to
Mitjagin), and (iv) is from \cite{Ben2} and \cite[Theorem~3]{BGN}. Recall that $\mathcal S_r(H)$ and $\mathcal S_{2q/p,q}(H)$ were defined in \S\ref{Remark on Schatten class}.\s

\begin{theorem}\label{3.8}  Let $H$ be a Hilbert space, and take  $(p, q) \in {\mathcal T}$.\s
\begin{enumerate}
\item[{\rm (i)}] Suppose that $p=q$. Then $\Pi_p(H) = \Pi_2(H)= {\mathcal S}_2(H)$.\s

\item[{\rm (ii)}] Suppose that  $p\leq 2$ and $1/p - 1/q<1/2$. Then $\Pi_{q,p}(H) = {\mathcal S}_r(H)$, where 
\[
	1/r = 1/q - 1/p + 1/2\,. 
\]

\item[{\rm (iii)}]  Suppose that    $1/p - 1/q\geq 1/2$. Then $\Pi_{q,p}(H) = {\B}(H)$.\s

\item[{\rm (iv)}] Suppose that  $2<p<q<\infty$. Then  $\Pi_{q,p}(H) = \mathcal S_{2q/p,q}(H)$.\enproof
\end{enumerate}
\end{theorem}\s

In connection with clause (i), we note that the exact constants that determine the relations between the $\pi_p\,$-norm and the $\pi_2\,$-norm on 
(real and complex) Hilbert spaces of various dimensions are calculated in  \cite{Gar1}.

Recall that the point $x_c\in (r,u_c)$ was defined on page \pageref{x_c}.\s

\begin{theorem}\label{equivalent of (p,q) mns on Lr when r<2}
Take $r\in (1,2)$,  and  let $\Omega$ be a measure space such that $\Lspace^r(\Omega)$ is infinite dimen\-sional. Suppose that two distinct points 
$P_1=(p_1,q_1)$ and $P_2=(p_2,q_2)$ in $\mathcal T$ are equivalent on  $\Lspace^r(\Omega)$. Then one of the following cases must occur.\s
\begin{enumerate}
\item[{\rm (i)}] The points $P_1$ and $P_2$ both lie in the region 
\[
	\set{(p,q)\in\mathcal T\colon 1/p-1/q\ge 1/r}\,; 
\]
in this case, the $(p,q)$-multi-norms corresponding to points in this region are all equivalent to the minimum multi-norm on $\Lspace^r(\Omega)$.\s

\item[{\rm (ii)}] The points $P_1$ and $P_2$ both lie on the same curve $\set{(p,q)\in\mathcal D_c\colon  1/p-1/q\ge 1/2}$ for some $c\in [1/2,1/r)$.
Further, $p_1,p_2 \in [1,x_c]$. \s

\item[{\rm (iii)}] The points $P_1$ and $P_2$ both lie on the same curve $\set{(p,q)\in\mathcal C_c\colon  1\le p\le r}$ for some $c\in (0,1/2)$.\s

\item[{\rm (iv)}] The points $P_1$ and $P_2$ both lie on the line segment $\set{(p,p)\colon 1\le p<r}$; in this case, the $(p,p)$-multi-norms corresponding
 to points on this line segment are all equivalent to the maximum multi-norm on $\Lspace^r(\Omega)$.
\end{enumerate}\s
\end{theorem}
\begin{proof}
By Theorems \ref{3.6} and \ref{non equivalent of (p,q) mns on Lr, on different curves}, all that remains to be considered is the case where
 $P_1$ and $P_2$ both lie on the same curve $\mathcal D_c$\,, where $0< c<1/r$. Without loss of generality, we may suppose that $p_1<p_2$, and so $p_1<q_1$ and
\[
	1/p_1-1/q_1\ge 1/p_2-1/q_2\,.
\]

\noindent\emph{Case 1: $c\in [1/2,1/r)$ and $1/p_i-1/q_i<1/2$ for both $i=1,2$.}  
Then, by Theorem  \ref{3.8}(i), (ii), or (iv), we have, for each $i=1,2$,   
\[
	\Pi_{q_i,p_i}(\lspace^2)= \ \textrm{either}\ \mathcal S_{2q_i/p_i,q_i}(\lspace^2)\ \textrm{or}\ \mathcal S_{r_i}(\lspace^2)\,,
\]
where $1/r_i=1/2-1/p_i+1/q_i$. Note that, since $c\ge 1/2$ and $P_1\neq P_2$, we must have $1/p_1-1/q_1\neq 1/p_2-1/q_2$, and so $r_1\neq r_2$.
 Thus we see, by a remark on page \pageref{Remark on Schatten class}, that $\Pi_{q_1,p_1}(\lspace^2) \neq \Pi_{q_2,p_2}(\lspace^2)$, and hence,
 by Corollary \ref{tool}, $P_1$ and $P_2$ are not equivalent on $\lspace^r$. This contradicts the hypothesis, and so this case cannot occur.
 \s

\noindent\emph{Case 2: $c\in [1/2,1/r)$ and $1/p_2-1/q_2<1/2\le 1/p_1-1/q_1$.} Then, by Theorem  \ref{3.8}, we have   
\[
	\Pi_{q_2,p_2}(\lspace^2)= \ \textrm{either}\ \mathcal S_{2q_2/p_2,q_2}(\lspace^2)\ \textrm{or}\ \mathcal S_{r_2}(\lspace^2)\,,
\]
where $1/r_2=1/2-1/p_2+1/q_2$. On the other hand, the same theorem implies that $\Pi_{q_1,p_1}(\lspace^2)=\operators(\lspace^2)$. 
So again we see that $\Pi_{q_1,p_1}(\lspace^2) \neq \Pi_{q_2,p_2}(\lspace^2)$, and hence, by Corollary \ref{tool},  $P_1$ and $P_2$ are not equivalent on $\lspace^r$. This contradicts 
the hypothesis, and so this case cannot occur.\s

We have shown in the above two cases that we cannot have both $1/p_2 -1/q_2 < 1/2$ and $c \in [1/2,1/r)$, and so necessarily $1 \leq p_1\leq p_2 \leq x_c$ when 
$c \in [1/2,1/r)$.\s

\noindent\emph{Case 3: $c\in (0,1/2)$.} Now $1/p_i-1/q_i<1/2$ for each $i=1,2$, and so, by Theorem  \ref{3.8}, we have  
\[
	\Pi_{q_i,p_i}(\lspace^2)= \ \textrm{either}\ \mathcal S_{2q_i/p_i,q_i}(\lspace^2)\ \textrm{or}\ \mathcal S_{r_i}(\lspace^2)\,,
\]
for each $i=1,2$, where $1/r_i=1/2-1/p_i+1/q_i$. Note that $r_1$ and $r_2$ cannot both be equal to  $2$. Thus, since $P_1$ and $P_2$ are equivalent on
 $\lspace^r$, by Corollary \ref{tool}, we must have   $\Pi_{q_1,p_1}(\lspace^2) = \Pi_{q_2,p_2}(\lspace^2)$. The only way this can happen, by the
 remark on page \pageref{Remark on Schatten class}, is when $p_i\le 2$ ($i=1,2$) and $1/p_1-1/q_1=1/p_2-1/q_2$. By the definition of the curve 
$\mathcal D_c$, this can happen only if $p_i\le r$ ($i=1,2$).\s

The three cases above complete the proof.
\end{proof}\s

\begin{remark} \label{remark 3.17}
In \cite{BDP}, it will be shown that $P_1$ and $P_2$ are equivalent whenever $1<r<2$ and both points lie on the same curve $\mathcal C_c$ 
for some $c \in (0,1/r)$.  Thus we know in every case whether $P_1$ and $P_2$ are equivalent, save for the case where both points lie on
 the same horizontal line $q = u_c$  and where $r\leq p_1<p_2 \leq x_c$. In this case,  $\Pi_{q_1,p_1}(\lspace^2) = \Pi_{q_2,p_2}(\lspace^2)$, a necessary 
condition for equivalence of the two multi-norms by Theorems \ref{tool for general spaces} and \ref{3.13}.\qed
\end{remark}\s

\begin{theorem}\label{equivalent of (p,q) mns on Lr when r >=2}
Take $r\geq 2$,  and  let $\Omega$ be a measure space such that $\Lspace^r(\Omega)$ is infinite dimen\-sional. Suppose that two distinct points 
$P_1=(p_1,q_1)$ and $P_2=(p_2,q_2)$ in $\mathcal T$ are equivalent on  $\Lspace^r(\Omega)$. Then one of the following cases must occur.\s

\begin{enumerate}
\item[{\rm (i)}] The points $P_1$ and $P_2$ both lie in the region 
\[
	\set{(p,q)\in\mathcal T\colon 1/p-1/q\ge 1/2}\,;
\]
in this case, the $(p,q)$-multi-norms corresponding to points in this region are all equivalent to the minimum multi-norm on $\Lspace^r(\Omega)$.\s

\item[{\rm (ii)}] The points $P_1$ and $P_2$ both lie on the same curve $\set{(p,q)\in\mathcal C_c\colon  1\le p\le 2}$ for some $c\in (0,1/2)$.\s

\item[{\rm (iii)}] The points $P_1$ and $P_2$ both lie on the line segment $\set{(p,p)\colon 1\le p\le 2}$; in this case,
the $(p,p)$-multi-norms corresponding to points on this line segment are all equivalent to the maximum multi-norm on $\Lspace^r(\Omega)$. 
\end{enumerate}\s
\end{theorem}

\begin{proof}
As in Theorem \ref{equivalent of (p,q) mns on Lr when r<2}, all that remains to be considered is the case where $P_1$ and $P_2$ both lie on the same curve
 $\mathcal D_c$\,, where $0< c<1/2$. We again need to consider only the space $\lspace^r$. For this, suppose without loss of generality that
$p_1<p_2$, and so $p_1<q_1$. Assume towards a contradiction that $p_2>2$. Then, by Theorem  \ref{3.8}(i) or (iv), we have   
\[
	\Pi_{q_2,p_2}(\lspace^2)= \ \textrm{either}\ \mathcal S_{2q_2/p_2,q_2}(\lspace^2)\ \textrm{or}\ \mathcal S_2(\lspace^2)\,.
\]

First, suppose that $p_1>2$. Then, by Theorem  \ref{3.8}(iv), we have   
\[
	\Pi_{q_1,p_1}(\lspace^2)= \mathcal S_{2q_1/p_1,q_1}\,.
\] 
But we know that $\mathcal S_{2q_1/p_1,q_1}(\lspace^2)$ is never equal to either $\mathcal S_{2q_2/p_2,q_2}(\lspace^2)$ or 
$\mathcal S_2(\lspace^2)$, and so $P_1$ and $P_2$ are not equivalent on $\lspace^r$ by Corollary \ref{tool}.

Second, suppose that $p_1\le 2$. Then $\Pi_{q_1,p_1}(\lspace^2) = \mathcal S_{r_1}(\lspace^2)$  by  Theorem \ref{3.8}(ii), where $1/r_1= 1/2-c$,
 and so $r_1>2$.  Thus again we see that $\Pi_{q_1,p_1}(\lspace^2) \neq \Pi_{q_2,p_2}(\lspace^2)$ by a remark on page \pageref{Remark on Schatten class}.

In both cases, we arrive at a contradiction to the assumption that $P_1$ and $P_2$ are equivalent on $\lspace^r$. Therefore $p_2\le 2$, and the proof is completed.
\end{proof}
 
In the Hilbert spaces case, i.e. when $r=2$, using Theorems~\ref{thm:hilb22}, we see that the $(p,p)$-multi-norms corresponding to points in the 
clause (iii) above are all equivalent to the Hilbert space multi-norm.\s
 
\begin{remark}\label{remark 3.19}
There remains the case where $P_1,P_2$ both lie on a curve $\mathcal C_c$ such that $0\le c<1/2$ and $p_1,p_2\in [1,2]$. Then again
$\Pi_{q_1,p_1}(\lspace^2) = \Pi_{q_2,p_2}(\lspace^2)$, a necessary condition for equivalence by Theorems \ref{tool for general spaces} and \ref{3.13}.
 In \cite{BDP}, it will be shown that $P_1$ and $P_2$ are indeed equivalent whenever $r\geq 2$ and both points lie on the same curve $\mathcal C_c$ for some $c \in (0,1/2)$.  
 Thus we have a complete classification whenever $r\geq 2$.\qed
\end{remark}
 \medskip
 
 \section{The relation with standard $t$-multi-norms}
 
 Let $\Omega$ be a measure space, and take $r\ge 1$. Then we have defined the standard $t$-multi-norm $(\norm^{[t]}_n)$ on $\Lspace^r(\Omega)$ 
whenever $t\ge r$, and we have defined the $(p,q)$-multi-norm $(\norm^{(p,q)}_n)$ on $\Lspace^r(\Omega)$ whenever  $(p, q) \in {\mathcal T}$. 
 We \emph{conjecture} that $(\norm^{[t]}_n)\not\cong (\norm^{(p,q)}_n)$  whenever $r>1$ and $\Lspace^r(\Omega)$ is infinite dimensional.

The first result proves rather more than the conjecture, but only in the special case in which $t=r$.  In the
following theorem, we suppose that $\co\otimes \Lspace^r(\Omega)\subset \Lspace^r(\Omega,\co)$ has the  norm from $\Lspace^r(\Omega,\co)$ corresponding to the 
standard $r\,$-multi-norm based on $\Lspace^r(\Omega)$ in  the manner explained above.

\begin{theorem} \label{nonequivalence}   Let $\Omega$ be a measure space, and take $r>1$. Suppose that  $\Lspace^r(\Omega)$ is an  infinite-dimensional
 space. Then the $\co$-norm on $\co\otimes \Lspace^r(\Omega)$ induced by the standard $r\,$-multi-norm  $(\norm^{[r]}_n:n\in\naturals)$ based 
on $\Lspace^r(\Omega)$ is not equivalent to any uniform $\co$-norm.  
\end{theorem}

\begin{proof} The following theorem  is  proved in  \cite[Section~7.3]{DF}. Suppose that $S\in {\B}(\Lspace^r(\Omega))$.  Then  the operator 
$I\otimes S:\co\otimes \Lspace^r(\Omega) \rightarrow \co\otimes \Lspace^r(\Omega)$ extends to a bounded operator on $\Lspace^r(\Omega,\co)$ if and only if $S$ is 
\emph{regular}, in the sense that  it is  a linear combination  of positive operators on the Banach lattice $\Lspace^r(\Omega)$.  However, since $\Lspace^r(\Omega)$ 
is an  infinite-dimensional space,  not all the operators  $S\in {\B}(\Lspace^r(\Omega))$ are regular.

Indeed, for a concrete example of an operator in  ${\B}(\Lspace^r(\Omega))$ which is not regular, we follow \cite[Section~7.6]{DF}.  Set $s=r'$, and let
 $S:\ell^{\,s}(\mathbb Z)\rightarrow \ell^{\,s}(\mathbb Z)$ be the discrete Hilbert transform given by
\[
 S(\delta_k) = \sum_{m\not=k} \frac{1}{m-k} \delta_m\quad (k\in\naturals)\,. 
\]
Then $S$  is bounded on $\lspace^s(\mathbb Z)$, but $I\otimes S$ is not bounded on the space
 $\ell^{\,1}\otimes\ell^{\,s}(\mathbb Z) \subset  \ell^{\,s}(\mathbb Z,\ell^{\,1})$. 
By duality, we see that $I\otimes S'$ is not bounded on the space $\co\otimes\,\ell^{\,r}(\mathbb Z) \subset  \ell^{\,r}(\mathbb Z,\co)$. 
 In the case where  $\Lspace^r(\Omega)$ is infinite dimensional, this  latter space contains a $1$-complemented copy of $\ell^{\,r}(\mathbb Z)$, and so we obtain
 an example of an operator on $\Lspace^r(\Omega)$ that is not regular.
 
For a stronger example, it  is shown by Arendt and Voigt \cite{AV} that the subalgebra of regular operators on $\Lspace^r(\Omega)$ is not even dense 
in ${\B}(\Lspace^r(\Omega))$ whenever $r>1$ and $\Lspace^r(\Omega)$  is  infinite dimensional.

We conclude that the standard $r\,$-multi-norm cannot be equivalent to any uniform $\co$-norm on $\co\otimes \Lspace^r(\Omega)$. 
\end{proof}\s

\begin{corollary}\label{nonequivalence with (p,q)}
Let $\Omega$ be a measure space, and take $r>1$. Suppose that  $\Lspace^r(\Omega)$ is an  infinite-dimensional
 space. Then  the standard $r\,$-multi-norm is not equivalent to  the maximum or  minimum multi-norms or to any $(p,q)$-multi-norm   on $\Lspace^r(\Omega)$ 
for $(p, q) \in {\mathcal T}$.\end{corollary}
\begin{proof}
This follows from the theorem because the projective, injective, and Chevet--Saphar norms are uniform $\co$-norms.
\end{proof}  

Again suppose that $\Omega$ is a measure space.  
Since $\Lspace^1(\Omega)$ is Dedekind complete as a Banach lattice, 
it follows from a remark on page 13 of \cite{AB} that every order-bounded operator on $\Lspace^1(\Omega)$ is regular. Since $\Lspace^1(\Omega)$ is an $AL$-space, 
and hence a $KB$-space (see \cite{AB}), it follows from \cite[Theorem 15.3]{AB} that every bounded operator on $\Lspace^1(\Omega)$ is order-bounded. Thus, 
in this case, every  $S \in{\B}(\Lspace^1(\Omega))$ is regular.  Thus the argument of the above proof does not apply.  Indeed, the conclusion of 
the preceding paragraph does not hold: by Theorem \ref{2.13}, $(\norm_n^{[q]})=(\norm^{(1,q)}_n)$ on
 $\Lspace^1(\Omega)$ for every $q\ge 1$ ({\it cf.} Theorem \ref{equivalency for L1}).

 The following theorem subsumes Theorem \ref{2.12} and part of Corollary \ref{nonequivalence with (p,q)}. \s

 \begin{theorem} \label{equivalence of (p,q) and [q] multi-norms}
Let $\Omega $ be a measure space, and take $r>1$, where
$\Lspace^r(\Omega)$ is  infinite dimensional.   Suppose that $t\ge r$ and that $(p, q) \in {\mathcal T}$, and assume that 
\[
	(\left\Vert \,\cdot\, \right\Vert^{(p,q)}_n)\cong (\left\Vert \,\cdot\, \right\Vert^{[t]}_n)\quad\textrm{on}\quad \Lspace^r(\Omega)\,.
\]  
Then $r<2$, $t\ge 2r/(2-r)$, and $(p,q)$ lies on the same curve $\mathcal D_c$ as $(r,t)$ with $p\le 2t/(2+t)$, so that $1/p-1/q\ge 1/2$. 
Moreover, we must also have $(\left\Vert \,\cdot\, \right\Vert^{[t]}_n)\cong (\left\Vert \,\cdot\, \right\Vert^{(r,t)}_n)$ on $\Lspace^r(\Omega)$.
\end{theorem}

\begin{proof} 
We need to consider only the space $\lspace^r$.   Set $\bar{r}=\min\set{r,2}$, as before. By \eqref{varphi sequence for 
standard q-multi-norm}, $\varphi^{[t]}_n(\lspace^r)=n^{1/t}$, and so it follows from Theorem \ref{values of phi sequence for Lr}  that one of the following must happen:
\begin{enumerate}
	\item[{\rm (i)}] either $p\ge \bar{r}$ and $q=t$\,;
	\item[{\rm (ii)}] or $p\le \bar{r}$ and $1/t=1/\bar{r}-1/p+1/q$\,.
\end{enumerate}

Let $n\in \N$, and take $\tuple{g}=(g_1, \dots, g_n)\in(\lspace^r)^n$ to be as in the proof of Lemma \ref{reduce lr to l2}. Then we see that
\[
\left\Vert \tuple{g} \right \Vert^{[t]}_n =\frac{1}{2^{n/r}}\sup\left\{\left(m_1^{t/r}  +
\cdots + m_k^{t/r}\right)^{1/t}\colon\ m_1 + \cdots+ m_k =2^n\right\}\,.
\]
Since $t/r\geq 1$, we have $m_1^{t/r}  + \cdots + m_k^{t/r}\leq 2^{nt/r}$, and so
$\left\Vert \tuple{g} \right \Vert^{[t]}_n  \leq 1$.  
On the other hand, Lemma \ref{reduce lr to l2} tells us that
\[
	\left\Vert \tuple{g} \right \Vert^{(p,q)}_n\sim \left\Vert (\delta_1,\ldots,\delta_n) \right \Vert^{(p,q)}_n\,,
\]
where $(\delta_k)$ is the standard basis sequence for $\lspace^2$. These and Example \ref{a calculation of (p,q)-norm on lr space} imply that $1/p-1/q\ge 1/2$. 

The previous two paragraphs now imply the claimed  result.
\end{proof}\s

Thus $(\left\Vert \,\cdot\, \right\Vert^{(p,q)}_n)$ is not equivalent to  $(\left\Vert \,\cdot\, \right\Vert^{[t]}_n)$ on $\Lspace^r(\Omega)$ in each 
of the following cases: 
\begin{enumerate}
\item[(i)] $r \ge 2$; 
\item[(ii)] $1<r< 2$ and $t< 2r/(2-r)$; 
\item[(iii)] $1/p-1/q<1/2$;
\item[(iv)] $(p,q)$ and $(r,t)$ lie on different curves $\mathcal D_c$.
\end{enumerate}
Moreover, our conjecture would be established if we could prove that $\norm^{[t]}_n\not\cong\norm^{(r,t)}_n$ on $\lspace^r$ for any $t\ge r$;
 this is open only when $1<r<2$ and $t \geq 2r/(2-r)$. Some further partial results will be given in \cite{BDP}. \medskip

\chapter{The Hilbert space multi-norm}

\section{Equivalent norms}  Let $H$ be a (complex) Hilbert space with inner product denoted by $\inner{\,\cdot\,,\,\cdot\,}$, and take $p\in [1,2]$. 
We know from Propositions \ref{equivalence of (p,p)-multi-norms on Lr}, \ref{2.10}, and \ref{thm:hilb22} that there is a constant $C_p$ such that 
\[
\left\Vert  \tuple{x}\right\Vert^H_n  = \left\Vert  \tuple{x}\right\Vert^{(2,2)}_n \leq  \left\Vert  \tuple{x}\right\Vert^{(p,p)}_n \le 
 \left\Vert  \tuple{x} \right\Vert^{\max}_n = \left\Vert  \tuple{x} \right\Vert^{(1,1)}_n\le
 C_p\left\Vert  \tuple{x} \right\Vert^{(p,p)}_n\quad(\tuple{x} \in H^n)
\]
for all $n\in\naturals$.  Our first theorem gives the best value of $C_2$.\smallskip

\begin{theorem}\label{litleGroth} 
Let $H$ be an infinite-dimensional, complex Hilbert space. Then 
\[
 \left\Vert  \tuple{x}\right\Vert^H_n = 
 \left\Vert  \tuple{x}\right\Vert^{(2,2)}_n \le \left\Vert  \tuple{x} \right\Vert^{\max}_n \le \frac{2}{\sqrt{\pi}}\left\Vert  \tuple{x} \right\Vert^{(2,2)}_n\quad(\tuple{x} \in H^n,\,n\in\N)\,; 
\] 
the constant $2/\sqrt{\pi}$ is best-possible in this inequality.
\end{theorem}

\begin{proof} By Theorem~\ref{connection with Chevet--Saphar norm}, the $(2,2)$-multi-norm on $H$ is the Chevet--Saphar norm $d_2$  on $\co\otimes H$. 
 Thus the dual space of $(\co\,\otimes H, \norm^H)$ is the space $\Pi_2(\co,H') = \Pi_2(\co,\overline H)$, 
where $\overline H$ is the conjugate of $H$.  Thus, by duality, the claim is equivalent to showing that 
\[
 \left\Vert T\right\Vert \leq \pi_2(T) \leq \frac{2}{\sqrt\pi} \left\Vert T\right\Vert\quad (T\in\B(\co,H))\,.
\]

The `Little Grothendieck Theorem'  says that every $T\in\B(\ell_n^{\,\infty},H)$ is $2$-summing, with $\pi_2(T)\leq (2/\sqrt\pi\,) \|T\|$ for each $n\in\N$.  See
\cite[Theorem~11.11]{DF} for the estimate, and \cite[Section~20.19]{DF}, where it is shown that this constant is the best possible 
(when the scalars are the complex numbers). 
 In particular, we see   that each operator $T\in\B(\co,H)$ is such that  $\pi_2(T)\leq (2/\sqrt\pi\,)\|T\|$; it follows that  
\[\sup\{ \pi_2(T) / \|T\| : T\in\B(\co,H) \} = 2/\sqrt\pi\,,
\] and this gives the required estimate.
\end{proof}\smallskip

The function $p\mapsto C_p,\,\;[1,2] \to [1, 2/\sqrt\pi]\,$, is increasing, with $C_1 =1 $ and $C_2 = 2/\sqrt\pi$; we do not have a formula for 
$C_p$. \medskip

\section{Equivalence at level $n$}

We now consider the best constant $c_n$, defined for each $n\in\naturals$, such that 
\[
	\left\Vert{\tuple{x}}\right\Vert^{\max}_n\le c_n\left\Vert{\tuple{x}}\right\Vert^{(2,2)}_n\quad (\tuple{x}\in H^n)\,.
\]
We know that $(c_n)$ is an increasing sequence in $[1,{2}/{\sqrt{\pi}}\,]$ with $c_1=1$ and that 
\[
	\lim_{n\to\infty} c_n=2/\sqrt{\pi}\,.
\]
We wonder which is the smallest value of $n$ such that $c_n>1$? The first fact that we can offer is that $c_2=c_3=1$, so that 
\[
\left\Vert{\tuple{x}}\right\Vert^{\max}_n = \left\Vert{\tuple{x}}\right\Vert^{(2,2)}_n= \left\Vert{\tuple{x}}\right\Vert^H_n\quad (\tuple{x}\in H^n) 
\]
for $n = 1,2,3$.

We start with some preliminary results. The following is a slight generalization of \cite[Proposition~2.8]{Jameson}.   In the result, we define $r$ by 
\[
\frac{1}{r} = \frac{1}{p}-\frac{1}{2} = \frac{1}{2} - \frac{1}{p'}
\]
in the case where $1<p<2$, so that $p=2r/(r+2)$.\s

\begin{proposition}\label{prop:2.1}
Let $1\leq p < \infty$, and let $(x_1,\ldots, x_n)$ be an orthogonal $n$-tuple   in a Hilbert space. Then
\[ \mu_{p,n}(x_1,\ldots, x_n) =
\begin{cases}
\max{\{ \|x_i\| : i\in\N_n\}} & (p\geq 2), \\
\Big( \sum_{i=1}^n \|x_i\|^r \Big)^{1/r} & (1< p < 2), \\
\Big( \sum_{i=1}^n\|x_i\|^2 \Big)^{1/2} & (p = 1).
\end{cases} \]
\end{proposition}

\begin{proof} We calculate simply  that
\begin{align*}
\mu_{p,n}(x_1,\ldots, x_n) &=\sup\Big\{ \Big\| \sum_{i=1}^n \alpha_ix_i \Big\| : \sum_{i=1}^n |\alpha_i|^{p'}\leq 1 \Big\} \\
&= \sup\Big\{ \Big( \sum_{i=1}^n |\alpha_i|^2 \|x_i\|^2 \Big)^{1/2}  : \sum_{i=1}^n|\alpha_i|^{p'} \leq 1 \Big\}\,.
\end{align*}

Suppose that $p\geq 2$. Then $p'\leq 2$, and we see that the supremum is attained when
$(\alpha_i) = (\delta_{i,i_0})$ for some $ i_0\in\N_n$.  

Next, suppose that $1<p<2$, so that $2<p'<\infty$. We set $R=r/2$, so that 
 $R=p'/(p'-2)$ and $R' = p'/2>1$.  Then, by $\ell^R$--$\ell^{R'}$-duality, we see that 
\begin{eqnarray*} \Big( \sum_{i=1}^n \|x_i\|^{2R} \Big)^{1/R}= \sup\Big\{ \sum_{i=1}^n |\alpha_i|^2 \|x_i\|^2 : \sum_{i=1}^n |\alpha_i|^{2R'} \leq 1 \Big\}
= \mu_{p,n}(x_1,\ldots, x_n)^2\,
\end{eqnarray*}
because  $2R'=p'$, and hence  
\[
	\left(\sum_{i=1}^n \|x_i\|^r\right)^{1/r} = \mu_{p,n}(x_1,\ldots, x_n)\,.
\] 

Suppose that $p=1$. Then we are really taking  the supremum over the collection of sequence $(\alpha_i)$ such that
$|\alpha_i|\leq 1$ ($i\in\N_n$), and the result follows immediately.

 The result follows in each case.
\end{proof}\s

Let $H$ be a Hilbert space, and take $r$ with  $2\leq r\leq \infty$.  For $n\in\N$, we denote by $S^r_n \subset  H^n$ the set of all 
orthogonal $n$-tuples $(x_1,\cdots,x_n)\in H^n$ with $\sum_{i=1}^n \|x_i\|^r = 1$.  In particular, we have 
\[ 
	S^2_n = \big\{ (x_i)\in H^n : (x_i)\text{ orthogonal and }\|x_1+\cdots+x_n\|=1 \big\}\,. 
\]

By Proposition~\ref{prop:2.1}, with $r$ as defined  for some $p\in(1,2)$, we have 
\[ \overline{\langle S^r_n \rangle} \subset  (H^n, \mu_{p,n})_{[1]}\,. \]
That is, the closed convex hull of $S^r_n$ is a subset of the closed unit ball of $H^n$ equipped with the norm $\mu_{p,n}$.  By Proposition~\ref{prop:2.1},
this result also holds when $p=1$ and $r=2$, and it holds for $r=\infty$ and any $p\geq 2$. For us, it is actually these cases which are of most interest:
\[ 
	\overline{ \langle S^\infty_n \rangle }\subset  (H^n,\mu_{2,n})_{[1]}\,, \qquad
\overline{ \langle S^2_n \rangle }\subset  (H^n,\mu_{1,n})_{[1]}\,. 
\]

The Russo--Dye theorem \cite[Theorem 3.2.18(iii)]{HGD} can be used  to show that the closed unit ball $(H^n,\mu_{2,n})_{[1]}$ is precisely 
$\overline{\langle S^\infty_n \rangle}$.  Thus we could ask for which $n\in \N$ is it true that $\overline{ \langle S^2_n \rangle } = (H^n,\mu_{1,n})_{[1]}$?
We shall show shortly that this is equivalent to asking if the Hilbert multi-norm
and the maximum multi-norm agree at level $n$.\s

\begin{lemma}\label{Srn as set of extreme points}
Let $H$ be a Hilbert space, and suppose that $2\leq r<\infty$  and $n\in\mathbb N$.  Then 
\[ 
	S^r_n\subset \ex \overline{ \langle S^r_n \rangle } \,.
\]
In the   case where $H$ is a finite dimensional,   $S^r_n= \ex \overline{ \langle S^r_n \rangle }$.
\end{lemma}

\begin{proof}
Let $X$ be the space $H^n$ with the norm $\norm$ given by 
\[\|(x_i)\|_X = \left(\sum_{i=1}^n \|x_i\|^r\right)^{1/r}\quad ((x_i) \in X)\,.
\]
Then $S^r_n$ is a subset of the closed unit ball of $X$, and hence also
$\overline{ \langle S^r_n \rangle }$ is a subset of the closed unit ball of $X$.
The space $X$ is strictly convex (see, for example \cite{boas}).

Assume towards a contradiction that $(y_i)\in S^r_n$, but that $(y_i)$ is not an extreme point of $\overline{ \langle S^r_n \rangle }$, so that we can find
$x,z\in \overline{ \langle S^r_n \rangle }$ with $x\not=z$ and $2y=x+z$. We  then have 
\[
1 = \|y\|_X \le  \frac{1}{2}(\|x\|_X + \|z\|_X) \leq 1\,,
\]
 and so $\|x\|_X = \|z\|_X=1$.  By the strict convexity of $X$, we have   $\|  (x+  z)/2\| < 1$ because $x\not=z$, a contradiction, as required.
 
 Now suppose that $H$ is finite dimensional. Then  the set  $S^r_n$ is closed, and so, by Mil'man's converse to the Krein--Mil'man theorem, 
$S^r_n= \ex \overline{ \langle S^r_n \rangle }$.
\end{proof}\s

Finally, we show the link with the Hilbert multi-norm.  In the result, we identify (anti-linearly) the dual space of $H^n$ with $H^n$;
 a  sequence $(x_n)$ in a Hilbert space is {\it orthogonal\/} if $[x_i,\,x_j] =0$ whenever $i\neq j$.\s

\begin{theorem}\label{S2n and unit ball of mu1n}
Let $H$ be a Hilbert space, and let $n\in\mathbb N$.   Then:\s

{\rm (i)}  the unit ball of the dual of $(H^n,\norm^H_n)$ is
$\overline{ \langle S^2_n \rangle }$;\s

{\rm (ii)}  the unit ball of the dual of $(H^n,\norm^{\max}_n)$ is the unit ball of $(H^n,\mu_{1,n})$.\s

\noindent In particular, $\norm^H_n = \norm^{\max}_n$ on $H^n$ whenever $S^2_n =  \ex \left((H^n,\mu_{1,n})_{[1]}\right)$. 
\end{theorem}

\begin{proof}
For (i), let $y=(y_i)\in H^n$ be an orthogonal family with $\sum_{i=1}^n \|y_i\|^2\leq 1$.  Let $x=(x_i)\in H^n$ satisfy $\|x\|^H_n\leq 1$, and 
then choose a family $(P_i)_{i=1}^n$ of mutually orthogonal projections summing to $I_H$ with $P_i(y_i)=y_i\,\;(i\in\N_n)$.  Then
\[ [x,y] = \Big|\sum_{i=1}^n [x_i, P_i(y_i)] \Big|
\leq \Big( \sum_{i=1}^n \|P_i(x_i)\|^2 \Big)^{1/2}\,\cdot\, \Big( \sum_{i=1}^n \|y_i\|^2 \Big)^{1/2}
\leq \|x\|^H_n \leq 1\,. \]
Thus the norm of $y$ as a functional on $(H^n,\norm^H_n)$ is at most $1$.
We conclude that $S^2_n \subset  (H^n,\norm^H_n)'_{[1]}$, and hence $\overline{\langle S^2_n \rangle} \subset (H^n,\norm^H_n)'_{[1]}$. 

Conversely, assume towards a contradiction that $\overline{\langle S^2_n \rangle}
\subsetneq (H^n,\norm^H_n)'_{[1]}$.  Then there exists  $x\in (H^n,\norm^H_n)'_{[1]}$
such that a small open ball about $x$ is disjoint from $\overline{\langle S^2_n \rangle}$.
By the Hahn--Banach theorem, there exists $z=(z_i)\in H^n$ and $\gamma\in\mathbb R$ such that
\[ \Re \left(\sum_{i=1}^n [z_i,x_i]\right) < \gamma \leq \Re\left(\sum_{i=1}^n [z_i,y_i]\right)
\quad \left( (y_i)\in \overline{\langle S^2_n \rangle} \right)\,. \]
Since $\overline{\langle S^2_n \rangle}$ is absolutely convex, we see that $\gamma<0$, and so actually
\[ -\Re\left(\sum_{i=1}^n [z_i,x_i]\right) > |\gamma| \geq \Big| \sum_{i=1}^n [z_i,y_i] \Big|
\quad \left( (y_i)\in \overline{\langle S^2_n \rangle} \right)\,. \]
Now observe that
\[
\sup\left\{ \left\vert \sum_{i=1}^n [z_i,y_i] \right\vert :
  (y_i)\in \overline{\langle S^2_n \rangle} \right\}
\]
is greater than or equal to 
\[
\sup  \left\vert \sum_{i=1}^n [z_i,y_i] \right\vert
\]
with the supremum taken over all orthogonal sequences  in $H$ with $\sum_{i=1}^n \|y_i\|^2\leq 1$, and that this supremum is equal to  
\[
\sup  \left\vert \sum_{i=1}^n \alpha_i[z_i,e_i] \right\vert
\]
taken over all orthonormal sequences $(e_i)$ in $H$ and all sequences $(\alpha_i)$ with $ \sum_{i=1}^n |\alpha_i|^2\leq 1$.  In its turn, this supremum is equal to 
\[
\sup  \left( \sum_{i=1}^n |[z_i,e_i]|^2 \right)^{1/2}
\]
taken over all orthonormal sequences $(e_i)$ in $H$, and hence, finally, to $\| (z_i) \|^H_n$. Thus 
\[
-\Re\left(\sum_{i=1}^n [z_i,x_i] \right)> \|(z_i)\|^H_n\,.
\] 
But this contradicts the fact that $(x_i) \in (H^n,\norm^H_n)'_{[1]}$.  Thus (i) holds.\s

For (ii), we know that $(H^n,\norm^{\max}_n) \cong \ell^{\,\infty}_n \proten H$, and that the dual of the  latter space is
 $\ell^{\,1}_n \injectivetensor H' = \mathcal B(H,\ell^{\,1}_n)$.  By definition, the space  $(H^n,\mu_{1,n})$ can be identified with 
$\mathcal B(H,\ell^{\,1}_n)$, and so (ii) follows.

In conclusion, it follows that $\norm^H_n = \norm^{\max}_n$ if and only if $\overline{ \langle S^2_n \rangle } = (H^n,\mu_{1,n})_{[1]}$.
By the previous lemma, this equality holds whenever $S^2_n =  \ex (H^n,\mu_{1,n})_{[1]}$.
\end{proof}\s

We shall show that indeed $S^2_n =\ex (H^n,\mu_{1,n})_{[1]}$ when $n=2$ or $n=3$; thus, in these cases, we have a description of the dual
space of $(H^n,\mu_{1,n})$, which may be of independent interest.\medskip

\section{Calculation of $c_2$}
We begin with an elementary result that shows that $c_2=1$.  

\begin{theorem}\label{4.3e}   Let $H$ be a complex Hilbert space. Then $\norm^H_2=\norm^{\max}_2$ on $H^2$.
\end{theorem}

\begin{proof}   
It is sufficient to prove the result in  the case where the dimension of $H$ is at least $2$. 

Set $L := (H^2, \mu_{1,2})_{[1]}$,  and recall that
\[
\mu_{1,2}(y_1,y_2) = \sup\{\lV \zeta_1y_1+\zeta_2y_2\rV : \zeta_1,\zeta_2 \in \T\}\quad (y_1,y_2\in H)\,.
\]
Let $ (y_1,y_2) \in \extremepoints L$. By replacing $y_1$ and $y_2$ by $\eta_1y_1$ and $\eta_2y_2$, respect\-ively, for suitable 
$ \eta_1,\eta_2 \in \T$, we may suppose that $\lV y_1+y_2\rV =1$, and so
\[
\lV y_1\rV^2 + \lV y_2\rV^2 + 2\Re\/(\zeta_1\overline{\zeta}_2[y_1,y_2])\leq \lV y_1\rV^2 + \lV y_2\rV^2 +2\Re\/[y_1,y_2] =1
\]
for each $\zeta_1,\zeta_2 \in \T$. We have $\Re\/(\zeta[y_1,\,y_2]) \leq \Re\/ [y_1,\,y_2]\;\,(\zeta\in \T)$, and so $ [y_1,\,y_2]\geq 0$.

Assume towards a contradiction that $[y_1,y_2]> 0$.  

Choose $u\in H$ with $\lV u\rV =1$ such that $[y_1,u] = [u,y_2]$,
 and then choose $\varepsilon >0$ with $\varepsilon^2 <  [y_1,y_2]$.  Set $w_1 = y_1+ \varepsilon u$ and $w_2 = y_2- \varepsilon u$. Then we have
\[
\lV w_1\rV^2+ \lV w_2\rV^2  = \lV y_1\rV^2+ \lV y_2\rV^2 +2\varepsilon^2
  \]
because  $[y_1,u] = [u,y_2]$, and
\[
\Re\ {(\zeta_1\overline{\zeta}_2}([y_1,y_2]-\varepsilon^2))\leq [y_1,y_2]- \varepsilon^2=[w_1,w_2]
\]
for each $\zeta_1,\zeta_2 \in \T$  because  $[y_1,y_2]- \varepsilon^2>0$. Hence 
\[
\lV \zeta_1w_1  + \zeta_2w_2\rV\leq \lV w_1  + w_2\rV=1\qquad (\zeta_1,\zeta_2 \in \T)\,,
\]  
and so $(y_1 +\varepsilon u,\,y_2-\varepsilon u) \in L$. Similarly, $(y_1 -\varepsilon u,\,y_2+\varepsilon u) \in L$. However
\[
2(y_1,y_2) = (y_1 +\varepsilon u,\,y_2-\varepsilon u) + (y_1 -\varepsilon u,\,y_2+\varepsilon u)\,.
\]
It follows that $(y_1,y_2)$ is not an extreme point of $L$, the required contradiction.

We have shown that  $(y_1,y_2) \in S^2_2$.  
Thus $\ex L\subset S^2_2$, and so $L\subset \closure{\langle S^2_2\rangle}$. This implies that $L= \closure{\langle S^2_2\rangle}$ 
(and, by Lemma \ref{Srn as set of extreme points}, we must also have $\ex L= S^2_2$.)
\end{proof}\medskip

\section{Calculation of $c_3$} Next we consider the case where $n=3$. In fact, there is now a difference between real  and complex   Hilbert spaces.\s

\begin{proposition}\label{4.3f} Let $H$ be a real Hilbert space of dimension at least $3$. Then $\norm^H_3$ and $\norm^{\max}_3$ are not equal on $H^3$.
\end{proposition}

\begin{proof}  It is sufficient to consider $H$ to be the real 3-dimensional Hilbert space $\ell_3^{\,2}(\R)$. Set $L = (H^3, \mu_{1,3})_{[1]}$.

For  $y_1,y_2,y_3 \in H$, we  now have
\[
\mu_{1,3}(y_1,y_2,y_3) = \sup\{\lV t_1y_1+t_2y_2+t_3y_3\rV : t_1,t_2,t_3 \in\{\pm1\}\}\,.
\]
Consider the vectors
\[
y_1=  \frac{1}{\sqrt{11}}(1,0,0 )\,,\quad y_2=  \frac{1}{\sqrt{11}}(1,1,0 )\,,\quad y_3= \frac{1}{\sqrt{11}}(-1,2,1 )\,.
\]
We see that   $[y_1,y_2] =[y_2,y_3] =1/11 $ and $[y_1,y_3]=-1/11 $, and so
\begin{eqnarray*}\lV y_1+y_2+ y_3\rV^2  = \sum_{j=1}^3\lV y_j\rV^2 + 2\sum_{i<j}[y_i,y_j] =
\frac{1}{11}(9+2\,\cdot\, 1)  =1\,.
\end{eqnarray*}
For each $t_1,t_2,t_3 \in\{\pm1\}$, we have $t_1t_2-t_1t_3+t_2t_3 \leq 1$,  and so it follows immediately that $\mu_{1,3}(y_1,y_2,y_3)=1$, showing that $(y_1,y_2,y_3) \in L$.
Note that the  expression  $t_1t_2-t_1t_3+t_2t_3$ takes its maximum value of $1$  when  $t_1=t_2=t_3 =1$, when $t_1=t_2  =1$ and $t_3=-1$, and when 
$t_1=1$ and $t_2=t_3 =-1$.

We {\it claim\/} that $\tuple{y} := (y_1,y_2,y_3)$ is an extreme point of $L$.

Assume towards a contradiction that there exists $\tuple{u}  \in H^3$ with $\tuple{u}\neq 0$ such that  $\tuple{y}\pm \tuple{u} \in L$, say 
$\tuple{u}= (u_1,u_2,u_3)$, with $u_1,u_2,u_3\in H$.

Take $t_1,t_2,t_3 \in \{\pm1\}$ with $t_1t_2-t_1t_3+t_2t_3=1$.  Then clearly $\lV t_1y_1+t_2y_2+ t_3y_3\rV =1$. However
\[
\lV t_1(y_1+u_1)+t_2(y_2+u_2)+ t_3(y_3+u_3)\rV \leq 1
\] 
and 
\[
\lV t_1(y_1-u_1)+t_2(y_2-u_2)+ t_3(y_3-u_3)\rV \leq 1\,.
\]
Since  $H$  is strictly convex, it follows that $ t_1u_1+t_2u_2+ t_3u_3=0$.   By taking the various possibilities for $t_1,t_2,t_3$
 such that $t_1t_2-t_1t_3+t_2t_3=1$ specified above, we see that   $ u_1+u_2+ u_3=0$, that $ u_1+u_2- u_3=0$,  and that  $ u_1-u_2- u_3=0$. 
Thus $u_1=u_2= u_3=0$, a contradiction.  Hence  $(y_1,y_2,y_3)\in \extremepoints L$.

Since $\{y_1,y_2,y_3\} $ is manifestly not an orthogonal set in $H$, it follows that $\tuple{y}$ is not in the   set $S^2_3$, and so the two multi-norms are not equal.
\end{proof}\smallskip

We shall now show that we obtain a different result from the above in the case where $H$ is a complex Hilbert space. Indeed $\norm^H_3=\norm^{\max}_3$
 on each complex Hilbert space $H$. But now the (elementary) calculations seem to be much more challenging.  \s

\begin{lemma}
    \label{whenismaximum}
    Take $(y_1,y_2,y_3)\in H^3$. Suppose that   $(\xi_1,\xi_2,\xi_3)\in\T^{\,3}$ is  such that
    \[
        \lV{\xi_1y_1+\xi_2y_2+\xi_3y_3}\rV=\max\set{\lV{\eta_1y_1+\eta_2y_2+\eta_3y_3}\rV:\
        (\eta_1,\eta_2,\eta_3)\in\T^{\,3}}.
    \]
    Then
    \[
        \Im \inner{\xi_1y_1,\xi_2y_2}=\Im \inner{\xi_2y_2,\xi_3y_3}=\Im \inner{\xi_3y_3,\xi_1y_1}\,.
    \]
\end{lemma}

\begin{proof}
 We see that $\lV{\eta_1y_1+\eta_2y_2+\eta_3y_3}\rV$ for $(\eta_1,\eta_2,\eta_3)\in\T^{\,3}$ attains its  maximum at the point $(\eta_1,\eta_2,\eta_3)=(\xi_1,\xi_2,\xi_3)$ 
 whenever
\[    \Re\left(\eta_1\conjugate{\eta_2}\inner{y_1,y_2}\right)+\Re\left(\eta_2\conjugate{\eta_3}\inner{y_2,y_3}\right)+
\Re\left(\eta_3\conjugate{\eta_1}\inner{y_3,y_1}\right)
 \]
 attains its maximum  at $(\eta_1,\eta_2,\eta_3)=(\xi_1,\xi_2,\xi_3)$. 

Next set  $\inner{y_1,y_2}=a\exp(\Imath \alpha)$, 
$\inner{y_2,y_3}=b\exp(\Imath \beta)$, and $\inner{y_3,y_1}=c\exp(\Imath \gamma)$, where $a,b,c\ge 0$ and $\alpha,\beta, \gamma \in \R$.  Also,
take $t_1,t_2,t_3 \in \R$ with   $\eta_i=\exp(\Imath t_i)$ for   $i=1,2,3$.  Then the fact that the real-valued function
 \[
    F: (t_1,t_2,t_3)\mapsto a\cos(t_1-t_2+\alpha)+b\cos(t_2-t_3+\beta)+c\cos(t_3-t_1+\gamma)
 \]
 attains its maximum at $(t_1,t_2,t_3)$ implies that  
 \[
    0=\frac{\partial F}{\partial t_1}(t_1,t_2,t_3)=\frac{\partial F}{\partial
    t_2}(t_1,t_2,t_3)=\frac{\partial F}{\partial t_3}(t_1,t_2,t_3)\,,
 \]
and hence that 
 \[
    a\sin(t_1-t_2+\alpha)=b\sin(t_2-t_3+\beta)=c\sin(t_3-t_1+\gamma)\,.
 \]
 This gives the specified equations.
\end{proof}\s

In the following lemmas, $A$ is the angle at the vertex $A$ of the triangle $ABC$, and $BC$ is the length of the side from $B$ to $C$, etc.   In the first two   lemmas, $ABC$ is a triangle (if such a triangle exists) with  $BC=1/a$, $CA=1/b$, and $AB=1/c$, where $a,b,c>0$.  Further, we shall consider the function
\[
 F: (r,s,t)\mapsto a\cos r+b\cos s+c\cos t\,,\quad \R^{\,3} \to \R\,.
\]
 \s

\begin{lemma}  \label{anglepi}  Consider $F_{\pi}$ to be the restriction of $F$ to the set 
\[
	\{(r,s,t)\in \R^{\,3} : r+s+t\equiv \pi\,\;{\rm(mod} \;2\pi{\rm)}\}\,.
\]
\begin{enumerate}
\item[{\rm (i)}] Suppose that  the triangle ABC exists. Then 
$F_{\pi}$ attains its maximum at exactly two points  $(r,s,t)=(A,B,C)$ or  $(r,s,t)=(-A,-B,-C)$ 
{\rm(mod} $2\pi${\rm)}.

\item[{\rm (ii)}] Suppose that  the triangle ABC does not exist and that $a\le b\le c$. Then $F_{\pi}$ attains its maximum at exactly the point $(r,s,t)=(\pi,0,0)$ 
{\rm(mod} $2\pi${\rm)}.
\end{enumerate}
\end{lemma}

\begin{proof} This is elementary. \end{proof}\s

\begin{lemma} \label{anglenotpi}
Suppose that  $M\not\equiv \pi$ $($mod $2\pi)$, and consider $F_M$ to be the restriction of $F$ to the set 
\[
	\{(r,s,t)\in \R^{\,3} : r+s+t\equiv M\,\;{\rm(mod} \;2\pi{\rm)}\}\,.
\]
    Then $F_M$  attains its  maximum at exactly one tuple $(r,s,t)\,\;  {\rm(mod} \;2\pi{\rm)}$.
\end{lemma}

\begin{proof}  Without loss of generality, we may suppose that $a\le b\le c$. The case where $M=0$ (mod $2\pi$) is obvious.
    Replacing $M$ by $M+2k\pi$ or $2k\pi-M$, if necessary, we may suppose that
    $0<M<\pi$. Note that the maximum of $F_M$ is at least
    \[
    	a\cos M+b+c> b+c-a\ge c\,.
    \]
     Set
    \[
    	p=\arcsin(a/b)\quad\textrm{and}\quad q=\arcsin(a/c)\,,
    \]
    so that we have the picture below.

\medskip\medskip

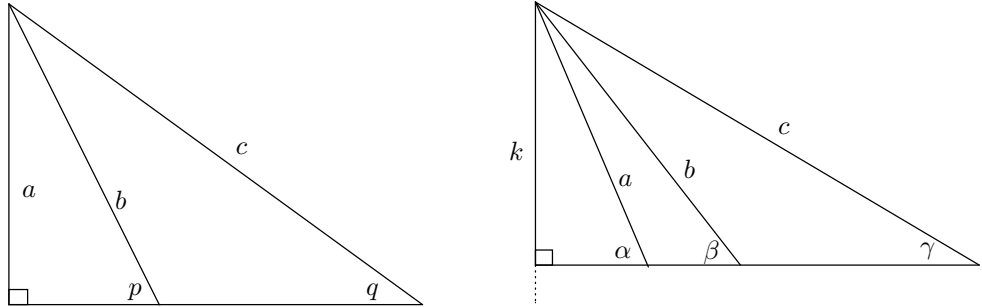
\unitlength 1pt
\begin{picture}(367.0406,114.5223)( 42.6791,-170.7165)
%
\special{pn 8}%
\special{pa 591 788}%
\special{pa 591 2363}%
\special{fp}%
%
\special{pn 8}%
\special{pa 591 2363}%
\special{pa 2756 2363}%
\special{fp}%
%
\special{pn 8}%
\special{pa 591 788}%
\special{pa 2756 2363}%
\special{fp}%
%
\special{pn 8}%
\special{pa 591 788}%
\special{pa 1378 2363}%
\special{fp}%
\special{pa 1378 2363}%
\special{pa 1378 2363}%
\special{fp}%
%
\special{pn 8}%
\special{pa 3347 778}%
\special{pa 3347 2156}%
\special{fp}%
%
\special{pn 8}%
\special{pa 3347 778}%
\special{pa 3938 2166}%
\special{fp}%
%
\special{pn 8}%
\special{pa 3347 778}%
\special{pa 4420 2156}%
\special{fp}%
%
\special{pn 8}%
\special{pa 3347 778}%
\special{pa 5670 2156}%
\special{fp}%
%
\special{pn 8}%
\special{pa 5670 2156}%
\special{pa 3347 2156}%
\special{fp}%
%
\special{pn 8}%
\special{pa 3347 2077}%
\special{pa 3436 2077}%
\special{pa 3436 2156}%
\special{pa 3347 2156}%
\special{pa 3347 2156}%
\special{pa 3347 2077}%
\special{fp}%
%
\special{pn 8}%
\special{pa 591 2284}%
\special{pa 689 2284}%
\special{pa 689 2363}%
\special{pa 591 2363}%
\special{pa 591 2363}%
\special{pa 591 2284}%
\special{fp}%
\put(50.3489,-128.0374){\makebox(0,0){$a$}}%
\put(85.0488,-130.8827){\makebox(0,0){$b$}}%
\put(130.7280,-111.6771){\makebox(0,0){$c$}}%
\put(275.9917,-123.7695){\makebox(0,0){$a$}}%
\put(300.1558,-119.5016){\makebox(0,0){$b$}}%
\put(335.8765,-103.8526){\makebox(0,0){$c$}}%
\put(234.7352,-112.3884){\makebox(0,0){$k$}}%
\put(90.4922,-166.0260){\makebox(0,0){$p$}}%
\put(180.4071,-166.0260){\makebox(0,0){$q$}}%
\put(275.0125,-150.7996){\makebox(0,0){$\alpha$}}%
\put(308.5784,-150.7996){\makebox(0,0){$\beta$}}%
\put(390.2462,-150.3770){\makebox(0,0){$\gamma$}}%
%
\special{pn 8}%
\special{pa 3347 2156}%
\special{pa 3347 2353}%
\special{dt 0.030}%
\special{pa 3347 2353}%
\special{pa 3347 2352}%
\special{dt 0.030}%
\end{picture}
\medskip\medskip

Suppose that $(r,s,t)$ is any point where $F_M$ attains its maximum; say $r,s,t\in (-\pi,\pi]$. We have seen that  $(r,s,t)$ must satisfy 
    \begin{align}\label{equation for maximum of F}
    	a\sin r=b\sin s =c\sin t\quad\textrm{as well as\quad $r+s+t\equiv M$ (mod $2\pi$)}.
 \end{align}
 Set $h=a\sin r$. Then $h\neq 0$ and 
\[
	\cos r=\pm\sqrt{1-\frac{h^2}{a^2}}\,,\quad\cos s=\pm\sqrt{1-\frac{h^2}{b^2}}\,,\quad\textrm{and}\quad \cos t=\pm\sqrt{1-\frac{h^2}{c^2}}\,.
\]
Since $a\le b\le c$ and $F_M(r,s,t)>c$, we deduce that
\[
	\cos s=\sqrt{1-\frac{h^2}{b^2}}\quad\textrm{and}\quad \cos t=\sqrt{1-\frac{h^2}{c^2}}\,,
\]
so that $s=\arcsin(h/b)$ and $t=\arcsin(h/c)$. In particular, we must have 
\[
	\abs{s}\le p\quad\textrm{and}\quad\abs{t}\le q\,.
\]

Assume toward a contradiction that $h<0$. In the case where $\cos r\ge 0$, we see that  $r,s,t\in [-\pi/2,0)$ and so $r+s+t=M-2\pi$. This implies that
\[
	p+q\ge \frac{3\pi}{2}-M>\frac{\pi}{2}\,.
\]
In particular, we must have ${1}/{b^2}+{1}/{c^2}>{1}/{a^2}$, so that $ABC$ is an (acute) triangle. Since $3\pi/2\le 2\pi+r<\pi-s-t\le 2\pi$, we see that
\[
	F_M(r,s,t)=F_M(2\pi+r,s,t)<F(\pi-s-t,s,t)\le F(A,B,C)<F(A',B',C')\,,
\]
where the second inequality follows from Lemma \ref{anglepi} and where $A'\in (0,A),B'\in (0,B)$, and $C'\in (0,C)$ are such that
 $A'+B'+C'=M$ (this is possible since $0<M<\pi$). This contradicts the assumption that $F_M$ attains its maximum at $(r,s,t)$.

In the case where $\cos r< 0$, we see that  $r\in (-\pi,-\pi/2)$, whereas $s,t\in [-\pi/2,0)$, and so $r+s+t=M-2\pi$. It follows that $\pi>-r>M-\pi-r=\pi+s+t>0$, and so
\[
	F_M(r,s,t)<a\cos(\pi+s+t)+b\cos(-s)+c\cos(-t)=F(\pi+s+t,-s,-t)\,.
\]
Choosing $u\in (0,\pi+s+t)$, $v\in(0,-s)$, and $w\in(0,-t)$ such that $u+v+w=M$, which is possible since $0<M<\pi$, the above implies that
\[
	F_M(r,s,t)<F_M(u,v,w)\,.
\]
This again contradicts the assumption that $F_M$ attains its maximum at $(r,s,t)$.

Thus we must have $h>0$, so that $r\in (0,\pi)$, $s\in (0,p]$, and $t\in (0,q]$. We consider the following two cases.
\medskip

\textbf{Case 1:} $M\le  {\pi}/{2}+p+q$. 
Assume toward a contradiction that $\cos r< 0$. Then  $r\in (\pi/2,\pi)$, whereas $s,t\in (0,\pi/2]$, and so $r+s+t=M$. Consider the function $g$ defined by 
    \[
        g(k)=\pi-\arcsin\left(\frac{k}{a}\right)+\arcsin\left(\frac{k}{b}\right)+\arcsin\left(\frac{k}{c}\right)\quad (0\le k\le a)\,.
    \]
Then $g(h)=M$ and $g(a)=\pi/2+p+q$. If $p+q\ge\pi/2$, then $g(h)<\pi\le g(a)$, and so there exists $k\in (h,a]$ such that $g(k)=\pi$. 
But this means that $\pi-\arcsin(k/a)$, $\arcsin (k/b)$, and $\arcsin (k/c)$ are three angles of a triangle whose sides are $1/a$, $1/b$ and $1/c$.
 In particular, this implies that $ABC$ is a triangle with $A\ge \pi/2$, so that $1/a^2\ge 1/b^2+1/c^2$, which means that $p+q\le \pi/2$.
 Thus  we must have $p+q\le \pi/2$ anyways, so that $1/a^2\ge 1/b^2+1/c^2$.  

We see that
    \[
        g'(k)=-\frac{1}{\sqrt{a^2-k^2}}+\frac{1}{\sqrt{b^2-k^2}}+\frac{1}{\sqrt{c^2-k^2}}\,,
    \]
    and, for $k\in (0,a)$, since $ 1/a^2 \ge 1/b^2 + 1/c^2$, we have
    \begin{align*}
        g''(k)&=-\frac{k}{(a^2-k^2)^{3/2}}+\frac{k}{(b^2-k^2)^{3/2}}+\frac{k}{(c^2-k^2)^{3/2}}\\
                &< -\frac{\displaystyle{{k}/{a^3}}}{\left(1-\displaystyle{\frac{k^2}{a^2}}\right)^{3/2}}+
                \frac{\displaystyle{{k}/{b^3}}}{\left(1-\displaystyle{\frac{k^2}{a^2}}\right)^{3/2}}+
                \frac{\displaystyle{{k}/{c^3}}}{\left(1-\displaystyle{\frac{k^2}{a^2}}\right)^{3/2}}<  0\,.
    \end{align*}
    Note that $g'(0)>0$ and $g'(a)=-\infty$. So we see that there exists a unique $k_0\in (0,a)$ such that
    $g'(k_0)=0$, $g$ is strictly increasing on $(0,k_0)$, and $g$ is strictly decreasing on $(k_0,a)$. In particular, since $h\in (0,a)$, we must have $M=g(h)>\min\set{g(0),g(a)}= {\pi}/{2}+p+q$; a contradiction of the assumption of Case 1.
    
Thus we must have $\cos r\ge 0$, and so $r,s,t\in (0,\pi/2]$. Hence $(r,s,t)$ must be the unique triple $(\alpha,\beta,\gamma)$ that satisfies \eqref{equation for maximum of F} and such that $\alpha,\beta,\gamma\in (0,\pi/2]$ (see the picture).

    \medskip
    
\textbf{Case 2:} $M>  {\pi}/{2}+p+q$. In this case, there exists no triple $(\alpha,\beta,\gamma)$ that satisfies \eqref{equation for maximum of F} 
and such that $\alpha,\beta,\gamma\in (0,\pi/2]$, and so $r\in(\pi/2,\pi]$. It follows that $r+s+t=M$. We also see from the assumption that 
$p+q<\pi/2$, so that $1/a^2>1/b^2+1/c^2$. Consider the function $g$ defined as in Case 1. Then $g(h)=M$. We again find a unique $k_0\in (0,a)$ 
such that $g$ is strictly increasing on $(0,k_0)$, and $g$ is strictly decreasing on $(k_0,a)$. Since $g(a)=\pi/2+p+q<M<g(0)=\pi$, $h$ is the
 unique point $l \in (k_0,a)$ such that  $g(l)=M$. This shows that $(r,s,t)$ is the unique triple (mod $2\pi$) at which $F_M$ attains its maximum.
\end{proof}

\medskip

 We summarize the above lemmas in the setting of our problem as follows. 
 
Let $(y_1,y_2,y_3)\in L$, where $L= (H^3, \mu_{1,3})_{[1]}$. For
$(\eta_1,\eta_2,\eta_3)\in\T^{\,3}$, set
\[
    N(\eta_1,\eta_2,\eta_3)=\lV{\eta_1y_1+\eta_2y_2+\eta_3y_3}\rV
\]
and  
\[
    F(\eta_1,\eta_2,\eta_3)=\Re\left(\eta_1\conjugate{\eta_2}\inner{y_1,y_2}\right)+\Re\left(\eta_2\conjugate{\eta_3}\inner{y_2,y_3}\right)
+\Re\left(\eta_3\conjugate{\eta_1}\inner{y_3,y_1}\right)\,,
\]
so that $N$ and $F$ attain their maxima at the same tuple(s) $(\eta_1,\eta_2,\eta_3)$.

 We shall now use square bracket-notation  $[\eta_1,\eta_2,\eta_3]$ to denote the class of all tuples $(\zeta\eta_1,\zeta\eta_2,\zeta\eta_3)$\,\; ($\zeta\in \T$);
we shall also call $[\eta_1,\eta_2,\eta_3]$ a `tuple', with the understanding that we are identifying all those 
$[\zeta\eta_1,\zeta\eta_2,\zeta\eta_3]$ for which $\zeta\in \T$.
 
Set
\[
  a=   \left\vert\, \inner{y_1,y_2}\,\right\vert\,,\,\;  b=   \left\vert\, \inner{y_2,y_3}\,\right\vert\,,\,\;   c=   \left\vert\, \inner{y_1,y_2}\,\right\vert\,,
\]
and then set
$M=\arg\inner{y_1,y_2}+\arg\inner{y_2,y_3}+\arg\inner{y_3,y_1}$.

Suppose that we have $a,b,c>0$. Then, by the previous three lemmas (and inspecting their proofs as well), we have $\max F(\eta_1,\eta_2,\eta_3)>\max\set{a,b,c}$,
 and there are the  following cases:\s

{\bf I) $M\equiv 0$ (mod $2\pi$):} $N$ attains its maximum at the unique $[\xi_1,\xi_2,\xi_3]$ in $\T^{\,3}$ satisfying the conditions that 
$\xi_1\conjugate{\xi_2}\inner{y_1,y_2}>0$, that $\xi_2\conjugate{\xi_3}\inner{y_2,y_3}>0$, and that
$\xi_3\conjugate{\xi_1}\inner{y_3,y_1}>0$. 
(Actually, if any  two of these inequalities hold, then the third must also hold.)

{\bf II) $M\equiv\pi$ (mod $2\pi$) and $ {1}/{a}$, ${1}/{b}$, and ${1}/{c}$ are the sides of a  triangle:} $N$ attains its maximum
at those $[\xi_1,\xi_2,\xi_3]$ in $\T^{\,3}$ such that 
\[
  \Im(\xi_1\conjugate{\xi_2}\inner{y_1,y_2})= \Im(\xi_2\conjugate{\xi_3}\inner{y_2,y_3})= \Im(\xi_3\conjugate{\xi_1}\inner{y_3,y_1})=:k\neq 0\,.
\]
There are exactly 2 such tuples $[\xi_1,\xi_2,\xi_3]$, and, moreover, for one such tuple, $k>0$ and, for the other, $k<0$.

{\bf III) $M\equiv\pi$ (mod $2\pi$) and $ {1}/{a}$, ${1}/{b}$, ${1}/{c}$ cannot be the sides of any triangle:} $N$ attains its 
maximum at the unique $[\xi_1,\xi_2,\xi_3]$ in $\T^{\,3}$ such  that  
\[
  \Im(\xi_1\conjugate{\xi_2}\inner{y_1,y_2})= \Im(\xi_2\conjugate{\xi_3}\inner{y_2,y_3})= \Im(\xi_3\conjugate{\xi_1}\inner{y_3,y_1})=0\,.
\]

{\bf IV) $0<M<\pi$ (mod $2\pi$):} $N$ attains its maximum at the unique $[\xi_1,\xi_2,\xi_3]$ in $\T^{\,3}$ such  that  
\[
  \Im(\xi_1\conjugate{\xi_2}\inner{y_1,y_2})= \Im(\xi_2\conjugate{\xi_3}\inner{y_2,y_3})= \Im(\xi_3\conjugate{\xi_1}\inner{y_3,y_1})=:k> 0\,.
\]

{\bf V) $\pi<M<2\pi$ (mod $2\pi$):} $N$ attains maximum at the unique $[\xi_1,\xi_2,\xi_3]$ in $\T^{\,3}$ such  that 
\[
  \Im(\xi_1\conjugate{\xi_2}\inner{y_1,y_2})= \Im(\xi_2\conjugate{\xi_3}\inner{y_2,y_3})= \Im(\xi_3\conjugate{\xi_1}\inner{y_3,y_1})=:k<0\,.
\]  
\medskip

Now  take  $(y_1,y_2,y_3)\in L$,  and suppose that  $N$
attains its maximum on $\T^{\,3}$ at the point $(\xi_1,\xi_2,\xi_3)\in \T^{\,3}$. Consider the elements  $u= (u_1,u_2,u_3)\in H^3$ with $u\neq 0$, if any,  such that
\[
        (y_1+\varepsilon u_1,y_2+\varepsilon u_2,y_3+\varepsilon u_3)\in L
\]
for $\varepsilon =-1$ and $\varepsilon =1$, and hence, by convexity, for all $\varepsilon\in [-1, 1]$. Since
\[
	\xi_1y_1+\xi_2y_2+\xi_3y_3\in \ex H_{[1]}\,,
\]
it follows that $\xi_1u_1+\xi_2u_2+\xi_3u_3=0$. So, for each $\varepsilon\in [-1, 1]$,  the function
\[
        (\eta_1,\eta_2,\eta_3)\mapsto\lV{\eta_1(y_1+\varepsilon u_1)+\eta_2(y_2+\varepsilon u_2)+\eta_3(y_3+\varepsilon  u_3)}\rV\,,\quad \T^{\,3}\to \R^+\,,
\]
also attains its maximum at $(\xi_1,\xi_2,\xi_3)$. Lemma \ref{whenismaximum} then implies that
\begin{align*}
        \Im (\xi_1\conjugate{\xi_2}\inner{y_1+\varepsilon u_1,y_2+\varepsilon u_2})=
        \Im (\xi_2\conjugate{\xi_3}\inner{y_2+\varepsilon u_2,y_3+\varepsilon u_3})=\Im  (\xi_3\conjugate{\xi_1}\inner{y_3+\varepsilon u_3,y_1+\varepsilon u_1})\,.
\end{align*}
Since $\xi_1u_1+\xi_2u_2+\xi_3u_3=0$, the coefficients of $\varepsilon^2$ are equal. Comparing the coefficients of
$\varepsilon$, we see that the above equalities are equivalent to
\[
    \inner{u_i,\xi_1y_1+\xi_2y_2+\xi_3y_3}=0\qquad (i=1,2,3)\,.
\]

\s

\begin{theorem}
Let $H$ be a complex Hilbert space. Then $\ex (H^3,\mu_{1,3})_{[1]}=S^2_3$, and $\norm^H_3=\norm^{\max}_3$ on $H^3$.
\end{theorem}

\begin{proof} 
It is sufficient to consider only the case where $H$ has dimension at least $3$. 

Let $(y_1,y_2,y_3)\in \ex L$, where $L=(H^3,\mu_{1,3})_{[1]}$ as before. For $(\eta_1,\eta_2,\eta_3)\in\T^{\,3}$, we define $N(\eta_1,\eta_2,\eta_3)$ and $F(\eta_1,\eta_2,\eta_3)$, and then $a,b,c,M$, as before. 
 
Suppose that $N$  attains its maximum, which is $1$, at $[\xi_1,\xi_2,\xi_3]$ in $\T^{\,3}$.  Let $(u_1,u_2,u_3)$ in $H^3$ be non-zero and such that
\begin{align*}
    \xi_1u_1+\xi_2u_2+\xi_3u_3&=0\\
    \inner{u_i,\xi_1y_1+\xi_2y_2+\xi_3y_3}&=0\quad (i\in\naturals_3)\,.
\end{align*}
In the case where $N$ attains maximum   at another (different) tuple $[\zeta_1,\zeta_2,\zeta_3]$ in $\T^{\,3}$, we require, further, that $(u_1,u_2,u_3)$ also satisfies
\begin{align*}
    \zeta_1u_1+\zeta_2u_2+\zeta_3u_3&=0\\
    \inner{u_i,\zeta_1y_1+\zeta_2y_2+\zeta_3y_3}&=0\quad (i\in\naturals_3)\,.
\end{align*}
It is easy to see that such $(u_1,u_2,u_3)$ always exists.

For each $\varepsilon\in\R$, set $y_{i,\varepsilon}=y_i+\varepsilon u_i$. For
$(\eta_1,\eta_2,\eta_3)\in\T^{\,3}$, set
\[
    N_\varepsilon(\eta_1,\eta_2,\eta_3)=\lV{\eta_1y_{1,\varepsilon}+\eta_2y_{2,\varepsilon}+\eta_3y_{3,\varepsilon}}\rV
\]
and  
\[
    F_\varepsilon(\eta_1,\eta_2,\eta_3)=\Re\left(\eta_1\conjugate{\eta_2}\inner{y_{1,\varepsilon},y_{2,\varepsilon}}\right)
    +\Re\left(\eta_2\conjugate{\eta_3}\inner{y_{2,\varepsilon},y_{3,\varepsilon}}\right)+
    \Re\left(\eta_3\conjugate{\eta_1}\inner{y_{3,\varepsilon},y_{1,\varepsilon}}\right)\,.
\]
Finally, set
\[
 a_\varepsilon =  \left\vert\,[y_{1,\varepsilon},y_{2,\varepsilon}]\,\right\vert\,,\quad
 b_\varepsilon=\left\vert\,[y_{2,\varepsilon},y_{3,\varepsilon}]\,\right\vert\,,\quad
  c_\varepsilon =  \left\vert\,[y_{3,\varepsilon},y_{1,\varepsilon}]\,\right\vert\,, 
\]
and set
$M=\arg\inner{y_{1,\varepsilon},y_{2,\varepsilon}}+\arg\inner{y_{2,\varepsilon},y_{3,\varepsilon}}+\arg\inner{y_{3,\varepsilon},y_{1,\varepsilon}}$.
Then we see that
\begin{align*}
    N_\varepsilon(\xi_1,\xi_2,\xi_3)=N_\varepsilon(\zeta_1,\zeta_2,\zeta_3)=1\,,
\end{align*}
and, from the discussion preceding this theorem, we have 
\begin{align*}
    \Im (\xi_1\conjugate{\xi_2}\inner{y_{1,\varepsilon},y_{2,\varepsilon}})=
        \Im (\xi_2\conjugate{\xi_3}\inner{y_{2,\varepsilon},y_{3,\varepsilon}})=\Im
        (\xi_3\conjugate{\xi_1}\inner{y_{3,\varepsilon},y_{1,\varepsilon}})=:I_\varepsilon\\
    \Im (\zeta_1\conjugate{\zeta_2}\inner{y_{1,\varepsilon},y_{2,\varepsilon}})=
        \Im (\zeta_2\conjugate{\zeta_3}\inner{y_{2,\varepsilon},y_{3,\varepsilon}})=\Im
        (\zeta_3\conjugate{\zeta_1}\inner{y_{3,\varepsilon},y_{1,\varepsilon}})=:J_\varepsilon\,.
\end{align*}
(The above equalities about $[\zeta_1,\zeta_2,\zeta_3]$ are considered only when the relevant tuple exists.)

First, we \emph{claim} that, in the case where both $I_0=0$ and
$F(\xi_1,\xi_2,\xi_3)>0$, for $\abs{\varepsilon}$ sufficiently
small, the sign of $I_\varepsilon$ and the sign of
\begin{align*}
   \Im \bigg(\inner{y_{1,\varepsilon},y_{2,\varepsilon}}
    \inner{y_{2,\varepsilon},y_{3,\varepsilon}}
           \inner{y_{3,\varepsilon},y_{1,\varepsilon}} \bigg)=\Im \bigg(\xi_1\conjugate{\xi_2}\inner{y_{1,\varepsilon},y_{2,\varepsilon}}
    \xi_2\conjugate{\xi_3}\inner{y_{2,\varepsilon},y_{3,\varepsilon}}
           \xi_3\conjugate{\xi_1}\inner{y_{3,\varepsilon},y_{1,\varepsilon}} \bigg)
\end{align*}
are the same. Indeed, since $I_0=0$, this can be verified by
considering the cases where the coefficients of $\varepsilon$ or $\varepsilon^2$ in $I_\varepsilon$ are non-zero. This claim implies
that, in the case where both $I_0=0$ and $F(\xi_1,\xi_2,\xi_3)>0$:\s

(i) $0<M_\varepsilon<\pi$ (mod $2\pi$) implies that $I_\varepsilon>0\,$;\s

 (ii) $\pi<M_\varepsilon<2\pi$ (mod $2\pi$) implies that $I_\varepsilon<0\,$;\s

 (iii) $M_\varepsilon\equiv 0$ or $\pi$ (mod $2\pi$) implies that $I_\varepsilon =0\,$.\s

Assume toward a contradiction that $a,b,c>0$. Then, for $\abs{\varepsilon}$ sufficiently small, we have
$a_\varepsilon,b_\varepsilon,c_\varepsilon>0$. As discussed above, there are five  cases:\s

{\bf Case 1:} $(y_1,y_2,y_3)$ falls in class {\bf I}. Then, for sufficiently small $\abs{\varepsilon}$, we also have
\begin{align*}
&\Re
(\xi_1\conjugate{\xi_2}\inner{y_{1,\varepsilon},y_{2,\varepsilon}})>0\,,\quad  \Re (\xi_2\conjugate{\xi_3}\inner{y_{2,\varepsilon},y_{3,\varepsilon}})>0\,,\quad  \Re
        (\xi_3\conjugate{\xi_1}\inner{y_{3,\varepsilon},y_{1,\varepsilon}})>0\,,\\
    &F_\varepsilon(\xi_1,\xi_2,\xi_3)>\max\set{a_\varepsilon,b_\varepsilon,c_\varepsilon},\quad\textrm{and}\quad M_\varepsilon\
 \textrm{is `close' to}\ 0\ \textrm{mod}\ 2\pi.
\end{align*}
By the claim, if $0<M_\varepsilon <\pi$ (mod $2\pi$), then $I_\varepsilon >0$, so that $(y_{1,\varepsilon},y_{2,\varepsilon},y_{3,\varepsilon})$ belongs
to class {\bf IV}, and $N_\varepsilon$ attains its maximum at  $[\xi_1,\xi_2,\xi_3]$. If $\pi<M_\varepsilon <2\pi$ (mod $2\pi$), then
$I_\varepsilon <0$, so that  $(y_{1,\varepsilon},y_{2,\varepsilon},y_{3,\varepsilon})$ belongs to class {\bf V}, and $N_\varepsilon$ attains its maximum at
 $[\xi_1,\xi_2,\xi_3]$. Finally, if $M_\varepsilon =0$ (mod $2\pi$), then $I_\varepsilon =0$, so that  $(y_{1,\varepsilon},y_{2,\varepsilon},y_{3,\varepsilon})$ belongs
to class {\bf I}, and $N_\varepsilon$ again attains its maximum at  $[\xi_1,\xi_2,\xi_3]$. Thus we always have 
$(y_{1,\varepsilon},y_{2,\varepsilon},y_{3,\varepsilon})\in  L$ for $\abs{\varepsilon}$ sufficiently small, and so  $(y_1,y_2,y_3)$ cannot 
be an extreme point of $L$.

{\bf Case 2:} $(y_1,y_2,y_3)$ falls into class {\bf II}. Suppose that $I_0>0$ and $J_0<0$. Then, for sufficiently small $\abs{\varepsilon}$, we also have
\begin{align*}
& {1}/{a_\varepsilon},  {1}/{b_\varepsilon},
 {1}/{c_\varepsilon}\  \textrm{are the sides of a triangle}\,,\quad I_\varepsilon>0\,,\quad J_\varepsilon<0\,,\\
  &  F_\varepsilon(\xi_1,\xi_2,\xi_3)>\max\set{a_\varepsilon,b_\varepsilon,c_\varepsilon}\,,\quad
    F_\varepsilon(\zeta_1,\zeta_2,\zeta_3)>\max\set{a_\varepsilon,b_\varepsilon,c_\varepsilon}\,,\\
  &\textrm{and}\quad  M_\varepsilon\ \textrm{is `close' to}\ \pi\ \textrm{mod}\ 2\pi\,.
\end{align*}
If $0<M_\varepsilon <\pi$ (mod $2\pi$), then $(y_{1,\varepsilon},y_{2,\varepsilon},y_{3,\varepsilon})$ belongs
to class {\bf IV}, and $N_\varepsilon$ attains its maximum at  $[\xi_1,\xi_2,\xi_3]$. If $\pi<M_\varepsilon <2\pi$ (mod $2\pi$), then
 $(y_{1,\varepsilon},y_{2,\varepsilon},y_{3,\varepsilon})$ belongs to class {\bf V}, and $N_\varepsilon$ attains its maximum at
 $[\zeta_1,\zeta_2,\zeta_3]$. Finally, if $M_\varepsilon =\pi$ (mod $2\pi$), then $(y_{1,\varepsilon},y_{2,\varepsilon},y_{3,\varepsilon})$ belongs
to class {\bf II}, and $N_\varepsilon$ attains its maximum at both $[\xi_1,\xi_2,\xi_3]$ and $[\zeta_1,\zeta_2,\zeta_3]$. Thus we always have 
$(y_{1,\varepsilon},y_{2,\varepsilon},y_{3,\varepsilon})\in  L$ for $\abs{\varepsilon}$ sufficiently small, and so  $(y_1,y_2,y_3)$ cannot be an extreme point of $L$.  

The other cases where $(y_1,y_2,y_3)$ falls into classes {\bf III, IV, V} can be covered by  similar arguments  to obtain  contradictions.

Thus we have proved that one of $a,b,c$ must be $0$. Say $a=0$. Assume toward a contradiction that $b,c>0$. Then we see that $N$
attains its maximum at the unique $[\xi_1,\xi_2,\xi_3]$ in $\T^{\,3}$ such that $\xi_2\conjugate{\xi_3}\inner{y_2,y_3}>0$  and
$\xi_3\conjugate{\xi_1}\inner{y_3,y_1}>0$. We also see easily that $F(\xi_1,\xi_2,\xi_3)>\max\set{a=0,b,c}$. If $a_\varepsilon\equiv 0$, then obviously, when $\abs{\varepsilon}$ is sufficiently small $N_\varepsilon$ again attains its maximum at  $[\xi_1,\xi_2,\xi_3]$, and so  
$(y_{1,\varepsilon},y_{2,\varepsilon},y_{3,\varepsilon})\in  L$, hence  $(y_1,y_2,y_3)$ cannot 
be an extreme point of $L$. So $a_\varepsilon\neq 0$ for sufficiently small $\abs{\varepsilon}$. Again, we can argue as
above, checking  
$(y_{1,\varepsilon},y_{2,\varepsilon},y_{3,\varepsilon})$ against each of the classes {\bf I} and {\bf III-V} (we can avoid class {\bf II}) and
the case where  $\inner{y_{1,\varepsilon},y_{2,\varepsilon}}=0$ to arrive at a contradiction.

Now we have proved that two of $a$, $b$, or $c$ must be $0$. We
can now argue as above    
to show that all $a$, $b$, $c$ are $0$. Hence $(y_1,y_2,y_3)\in S^2_3$. 

Thus we have proved that $\ex L\subset S^2_3$. This implies that $L\subset\langle S^2_3\rangle\subset L$. Hence $\ex L=S^2_3$ and the proof is complete.
\end{proof}\medskip
 
\section{Calculation of $c_4$} 

We can give some information about the constant $c_4$.\s

\begin{theorem}
 Let $H$ be a complex Hilbert space of dimension at least $3$. Then $\norm^H_n$ is not equal to $\norm^{\max}_n$ on $H^n$ for every $n\ge 4$.
\end{theorem}

\begin{proof}
It is sufficient to consider the case where $n=4$ and $H=\lspace^2_3$. Set $L:=(H^4,\mu_{1,4})_{[1]}$.  
Set $x_1=(1,0,0)$, $x_2=(-1,2,0)$, $x_3=(-1,-1,3)$, and $x_4=(-1,-1,-1)$. Then we have $[x_i,x_j]=-1$ for every $i,j \in \N_4$ with $i\neq j$. 
For each $(\xi_1,\xi_2,\xi_3,\xi_4)\in \T^4$, we have 
\[
	\Re\sum_{i< j} \xi_i\xi_j\ge -2\,,
\]
with the minimum attained  at those $(\xi_1,\xi_2,\xi_3,\xi_4)\in \T^4$ for which  $\xi_1+\cdots+\xi_4=0$, and so it follows that the function 
\[
	(\xi_1,\xi_2,\xi_3,\xi_4)\mapsto\lV\xi_1x_1+\cdots+\xi_4x_4\rV\,,\quad \T^4\to \R\,,
\]
attains its maximum at each $(\xi_1,\xi_2,\xi_3,\xi_4)\in S$,  where we set
\begin{align*}
	S:&=\set{(\xi_1,\xi_2,\xi_3,\xi_4)\in \T^4\colon \xi_1+\cdots+\xi_4=0}\\
	&=\set{(\xi_1,\xi_2,-\xi_1,-\xi_2)\ \textrm{and}\ (\xi_1,\xi_2,-\xi_2,-\xi_1)\colon\ \xi_1,\xi_2\in\T}\,.
\end{align*}

Let $\tuple{y}=(y_1,\ldots,y_4)$ be a scaling of $(x_1,\ldots,x_4)$ such  that $\mu_{1,4}((y_1,\ldots,y_4))=1$. In particular, $\tuple{y}\in L\setminus S^2_4$. 
We also have
\[
	\lV\xi_1y_1+\cdots+\xi_4y_4\rV\le 1
\]
for every $(\xi_1,\xi_2,\xi_3,\xi_4)\in \T^4$, and the equality is attained whenever  $(\xi_1,\xi_2,\xi_3,\xi_4)\in S$. 

Suppose that $\tuple{u}=(u_1,\ldots,u_4)\in H^4$ is  such that $\tuple{y}\pm\tuple{u}\in L$. Then, for every $(\xi_1,\xi_2,\xi_3,\xi_4)\in \T^4$ 
and every $\varepsilon\in[-1,1]$, we have
\[
	\lV\xi_1(y_1+\varepsilon u_1)+\cdots+\xi_4(y_4+\varepsilon u_4)\rV\le 1\,.
\]
In particular, for each $(\xi_1,\xi_2,\xi_3,\xi_4)\in S$, since $\xi_1y_1+\cdots+\xi_4y_4$, being of norm $1$, is an extreme point of $H_{[1]}$, we obtain
\[
	\xi_1u_1+\cdots+\xi_4u_4=0\,.
\]
This implies that $u_1=\cdots=u_4=:u$.

Fix an $\varepsilon\in\R$ with $\abs{\varepsilon}$ sufficiently small so that $a_i,b_i>0$ and  $A_i,B_i\in (\pi/2,3\pi/2)$ ($i\in\N_3$) 
can be chosen to satisfy the following equations:
\begin{align*}
	a_i\exp(\Imath A_i)&=[y_i+\varepsilon u,y_4+\varepsilon u]\quad (i=1,2,3)\,\qquad b_1\exp(\Imath B_1)=[y_2+\varepsilon u,y_3+\varepsilon u]\,,\\
	b_2\exp(\Imath B_2)&=[y_3+\varepsilon u,y_1+\varepsilon u]\,,\quad\textrm{and}\quad
	b_3\exp(\Imath B_3)=[y_1+\varepsilon u,y_2+\varepsilon u]\,;
\end{align*}
this can be done since $[y_i,y_j]<0$ for every $i, j\in\N_4$ with $i\neq j$. Using 
\[
	\xi_i=\exp(\Imath \alpha_i)\quad(i\in\N_3),\quad\textrm{and}\quad \xi_4=1\,,
\]
the previous paragraph then implies that the function $f : \R^3 \to \R $ defined by 
\begin{align*}
	f(\alpha_1,\alpha_2,\alpha_3):=&a_1\cos(\alpha_1+A_1)+a_2\cos(\alpha_2+A_2)+a_3\cos(\alpha_3+A_3)\\
	&\ +b_1\cos(\alpha_2-\alpha_3+B_1)+b_2\cos(\alpha_3-\alpha_1+B_2)+b_3\cos(\alpha_1-\alpha_2+B_3)\,,
\end{align*}
 attains its maximum at $(\alpha,\pi,\alpha+\pi)$ and $(\pi,\alpha,\alpha+\pi)$ for every $\alpha\in\R$. In particular,
these triples must be solutions of the equations
\[
	0=\frac{\partial f}{\partial \alpha_1}(\alpha_1,\alpha_2,\alpha_3)=\frac{\partial f}{\partial
    \alpha_2}(\alpha_1,\alpha_2,\alpha_3)=\frac{\partial f}{\partial \alpha_3}(\alpha_1,\alpha_2,\alpha_3)\,.
\]
This implies that $A_i=B_i=\pi$ $(i\in\N_3)$ and $a_1=a_2=a_3=b_1=b_2=b_3$. 

Thus we have shown that, for each $\varepsilon\in\R$ with sufficiently small $\abs{\varepsilon}$, all the numbers
\[
	[y_i+\varepsilon u,y_j+\varepsilon u]\quad(i,j\in \N_4,\,i\neq j)
\]
are equal to the same negative real number. Thus, the numbers
\[
	[y_i,u]+[u,y_j]\quad(i,j\in \N_4,\,i\neq j)
\]
are all equal, and since $\tuple{y}=(y_1,\ldots,y_4)$ is a scaling of $(x_1,\ldots,x_4)$, we deduce that
\[
	[u,x_1]=[u,x_2]=[u,x_3]=[u,x_4]\,.
\]
Solving these linear equations, we obtain $u=0$. This implies that $\tuple{y}$ is an extreme point of $L$. Hence $S^2_4\subsetneq\ex L$,
 and so $\norm^H_4\neq\norm^{\max}_4$ on $H^4$. 
\end{proof}

The above calculation shows that $1< c_4 \leq c_n \leq 2/\sqrt{\pi}$ for all $n \geq 4$. 
However, we have not calculated the actual value of $c_4$, or of any $c_n$ for $n \geq 4$.\medskip 

\end{document}